\documentclass[12pt]{article}

\usepackage{amsfonts,amssymb,amstext,graphics,amsmath,color}

\usepackage{amsmath,amssymb,amsbsy,upref}

\usepackage[english]{babel}
\usepackage{graphicx}
\usepackage{graphics} 
\usepackage{epsfig} 
\usepackage{color}
\usepackage{amsthm}  

\newcommand{\ee}{\varepsilon}

\newcommand{\Bb}{{\mathcal B}}

\newcommand{\Hh}{{\mathcal H}}

\newcommand{\lie}{\mathop{\displaystyle
{\stackrel{\scriptstyle{\varepsilon \to 0}}{\hbox to 40 pt
{\rightarrowfill}}  } }}
\newcommand{\lik}{\mathop{\displaystyle
{\stackrel{\scriptstyle{k \to +\infty}}{\hbox to 40 pt
{\rightarrowfill}}  } }}

\numberwithin{equation}{section}

\newtheorem{Theorem}{Theorem}[section]
\newtheorem{Lemma}[Theorem]{Lemma}
\newtheorem{Proposition}[Theorem]{Proposition}
\newtheorem{Corollary}[Theorem]{Corollary}

\newtheorem{Remark}[Theorem]{Remark}

\usepackage{titlesec}
\titleformat{\section}
{\bf\large}{\thesection}{7pt}{\bf\large}
\titleformat{\subsection}
{\bf}{\thesubsection}{5pt}{\bf}
\titleformat{\subsubsection}
{\bf}{\thesubsubsection}{3pt}{\bf}

\begin{document}

\title{\bf\Large Spectral gaps in a double-periodic perforated \\ Neumann waveguide}
\author{\normalsize{Delfina G{\'o}mez$^a$, Sergei A. Nazarov$^b$, Rafael Orive-Illera$^{c,d}$, Maria-Eugenia P\'erez-Mart\'inez$^e$ } \vspace{0.2cm}\\
\small{$^a$ Departamento de Matem\'aticas, Estad\'{\i}stica y Computaci\'on, Universidad de Cantabria,} \vspace{-0.1cm} \\
\small{Santander, Spain, \texttt{gomezdel@unican.es}} \vspace{0.1cm} \\
\small{$^b$ Institute for Problems in Mechanical Engineering of the Russian Academy of Sciences,} \vspace{-0.1cm}\\
\small{Saint Petersburg,  Russia, \texttt{srgnazarov@yahoo.co.uk}}\vspace{0.1cm}\\
\small{$^c$ Instituto de Ciencias Matem\'aticas, CSIC-UAM-UC3M-UCM, Madrid, Spain,}\vspace{-0.1cm}\\
\small{$^d$ Departamento de Matem\'aticas, Universidad Aut\'onoma de Madrid,}\vspace{-0.1cm}\\
\small{Madrid, Spain, \texttt{rafael.orive@icmat.es}}\vspace{0.1cm}\\
\small{$^e$ Departamento de Matem\'atica Aplicada y Ciencias de la Computaci\'on, Universidad de Cantabria,}\vspace{-0.1cm}\\
\small{Santander, Spain, \texttt{meperez@unican.es}}
}
\date{}

\maketitle

\vspace*{-0.7cm}
\begin{center}
\begin{minipage}{6in}
\thispagestyle{empty} \setlength{\baselineskip}{6pt}
{\small {\bf Abstract}
{We examine the band-gap structure of the spectrum of
the Neumann problem for the Laplace operator in a strip with
periodic dense transversal perforation by identical holes of a small
diameter $\varepsilon>0$.
The periodicity cell  itself contains a string of holes at a distance $O(\varepsilon)$ between them.
Under assumptions on the symmetry of the
holes, we derive and justify asymptotic formulas for the endpoints
of the spectral bands in the low-frequency range of the spectrum  as $\varepsilon \to 0$.
We  demonstrate that, for $\varepsilon$ small enough, some spectral gaps
are open. The position and size of the opened gaps depend on
the strip width, the perforation period, and certain integral
characteristics of the holes. The asymptotic behavior of the
dispersion curves near the band edges is described by means of a
`fast Floquet variable' and involves boundary layers in the vicinity of the
perforation string of holes. The dependence on the Floquet parameter
of the model problem in the periodicity cell requires  a serious
modification of the standard justification scheme in homogenization
of spectral problems. Some open questions and possible
generalizations are listed. }

\vspace*{0.2cm}

\vspace*{0.2cm}
}
\end{minipage}
\end{center}

\vspace*{0.5cm}

\section{Introduction}\label{sec1}

In this section, we formulate the spectral problem under consideration, cf. Section~\ref{subsec11}, and provide some background which relates it with a parametric family of homogenization problems, the so-called model problem.
In Section~\ref{subsec13} we provide the structure of the paper while its framework  in the literature is in Section~\ref{subsec12}.

\subsection{Formulation of the problem}\label{subsec11}
Let
\begin{equation}\label{(1)}
\Pi=\{x=(x_1,x_2)\in{\mathbb R}^2:\, x_1\in{\mathbb R}, x_2\in(0,H)\}
\end{equation}
be an open strip of width $H>0$ and let $\omega$ be a domain in the
plane ${\mathbb R}^2$ which is bounded by a smooth simple closed curve
$\partial \omega$ and has the compact closure
$\overline{\omega}=\omega\cup\partial\omega$ inside $\Pi$.
Let $\varepsilon=N^{-1} $ where  $N$ is a large natural number.
We introduce the strip $\Pi^\varepsilon$, see Figure~\ref{fig1} a), obtained from $\Pi$  perforated by the family of holes
 \begin{equation}\label{(2)}
 \omega^\varepsilon(j,k)=\{x:\,\varepsilon^{-1}(x_1-j,x_2-  \varepsilon k  H)\in\omega\},\quad
j\in{\mathbb Z},\,\,k=0,1,\dots,N-1,
\end{equation}
distributed periodically along line segments parallel to the
ordinate $x_2$-axis. Each hole is homothetic to $\omega$ of ratio $\varepsilon$ and translation of $\varepsilon\omega=\omega^\varepsilon(0,0).$ Namely,
\begin{equation}\label{(3)}
\Pi^\varepsilon=\Pi\setminus\overline{\Omega^\varepsilon}\quad\mbox{\rm
where}\quad\Omega^\varepsilon= \bigcup\limits_{j\in{\mathbb Z}}
\bigcup\limits_{k=0}^{N-1} \omega^\varepsilon(j,k).
\end{equation}
The period of perforation along the abscissa $x_1$-axis in
the domain $\Pi^\varepsilon$ is made equal to one by rescaling,  which also fixes the dimensionless
width $H>0$.
The period along the $x_2$-axis is $\varepsilon H $ with $\varepsilon\ll 1.$

\begin{figure}[t]
\begin{center}
\resizebox{!}{1.8cm}{\includegraphics{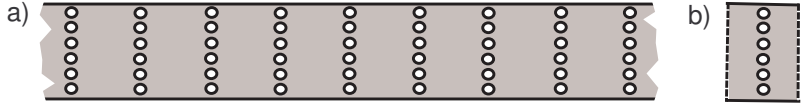}}
\caption{  a)  The perforated strip $\Pi^\ee$. b) The periodicity cell  $\varpi^\ee$.}
\label{fig1}
\end{center}
\end{figure}

We consider the spectral Neumann problem
 \begin{equation}\label{(4)}
 -\Delta_x u^\varepsilon(x)=\lambda^\varepsilon u^\varepsilon(x),\,\, x\in\Pi^\varepsilon,
\end{equation}
 \begin{equation}\label{(5)}
 \partial_\nu u^\varepsilon(x)=0,\,\, x\in\partial\Pi^\varepsilon,
\end{equation}
where $\partial_\nu$ is the directional derivative along the outward
normal while $\partial_\nu=\pm\partial/\partial x_2$ at the lateral
sides $\Upsilon_\pm=\{x:\,x_1\in{\mathbb R},x_2=(H\pm H)/2\}$ of the
strip \eqref{(1)}. The variational formulation of the problem
\eqref{(4)}, \eqref{(5)} reads: to find a function $u^\varepsilon$
in the Sobolev space $H^1(\Pi^\varepsilon)$, $u^\varepsilon\not\equiv 0$,  and a number $\lambda^\varepsilon\in {\mathbb C}$ such that
the integral identity
 \begin{equation}\begin{array}{c}
 (\nabla_xu^\varepsilon,\nabla_xv^\varepsilon)_{\Pi^\varepsilon}=\lambda^\varepsilon
 ( u^\varepsilon, v^\varepsilon)_{\Pi^\varepsilon}\quad \forall v^\varepsilon\in H^1(\Pi^\varepsilon)
 \end{array}\label{(7)}\end{equation}
is valid, cf.~\cite{Lad}.
Here, $\nabla_x=\mbox{\rm grad}$, $\Delta_x=\nabla_x\cdot\nabla_x$
is the Laplace operator and $(\,,)_{\Pi^\varepsilon}$ stands for the
natural scalar product in the Lebesgue space $L^2(\Pi^\varepsilon)$.

Since the bi-linear form on the left of \eqref{(7)} is positive,
symmetric, and closed in $H^1(\Pi^\varepsilon)$,
problem \eqref{(7)}  is associated with a positive self-adjoint
operator $A^\varepsilon$ in the Hilbert space $L^2(\Pi^\varepsilon)$
with the domain
 $$
{\cal D}(A^\varepsilon)=\{u^\varepsilon\in
H^2(\Pi^\varepsilon):\,\eqref{(5)}\,\,\mbox{\rm is verified}\}.
 $$
Clearly, the spectrum $\sigma(A^\varepsilon)$ belongs to the closed
real positive semi-axis $[0,+\infty)=\overline{{\mathbb
R}_+}\subset{\mathbb C}$. Moreover, according to the
Floquet--Bloch--Gelfand theory, see for instance
\cite{ReedSimon,Skrig,NaPl,Kuchbook,AllaireConca_JMPA}, the spectrum gets
the band-gap structure
 \begin{equation}\begin{array}{c}
 \sigma(A^\varepsilon)=\bigcup\limits_{p\in{\mathbb
 N}}\beta_p^\varepsilon\, ,
 \end{array}\label{(8)}\end{equation}
where the bands $\beta_p^\varepsilon$ are connected and compact sets
in $\overline{{\mathbb R}_+}=[0,+\infty)$.
The $\beta_p^\varepsilon$ are related to the eigenvalues, cf.~\eqref{(19)},
of the model problem in the periodicity cell
\begin{equation}\begin{array}{c}
\varpi^\varepsilon=\{x\in\Pi^\varepsilon:\,\vert x_1 \vert<1/2\},
\end{array}\label{(10)}\end{equation}
see Figure~\ref{fig1} b), which itself constitutes  a homogenization problem, cf.~\eqref{(11)}--\eqref{(14)}.
The {\it spectral bands} $\beta_p^\varepsilon$ and $\beta_{p+1}^\varepsilon$ may intersect
each other but  can also be disjoint so that the {\it spectral gap}
$\gamma_p^\varepsilon$ becomes open between them.
Recall that an open spectral gap is recognized as a nontrivial open interval in
${\mathbb R}_+$ which is free of the essential spectrum but has both
endpoints in it. If
$\beta^\varepsilon_p\cap\beta^\varepsilon_{p+1}\not=\varnothing$,
then we say that the gap $\gamma^\varepsilon_p$ is closed.
In Figure~\ref{fig3} the open spectral gaps correspond with the projections  of the shaded bands on the ordinate axis.

The main goal of our paper is to show that, under certain
restrictions on the width $H$ and the perforation shape, the problem
\eqref{(4)}, \eqref{(5)} can get at least one open gap in its
spectrum. Also, we aim to derive asymptotic formulas for the position
and geometric characteristics of several bands and gaps in the
low-frequency range of the spectrum. It should be mentioned that the
traditional homogenization procedure in the problem \eqref{(4)},
\eqref{(5)} does not help to detect open gaps.
The crucial role is played by the boundary layer phenomenon, cf.~Section~\ref{sec3}, while the width of the gaps is
expressed in terms of certain integral characteristics of the Neumann  hole $\omega$ of unit size in the strip $\Pi$  with the periodicity conditions at its lateral sides,  cf. \eqref{(MEP94)}, \eqref{(MEP99)} and  Remark~\ref{Remark_m1}.  At the same time, we construct explicitly only the main correction term in the
asymptotics of eigenvalues of the model problem in the
periodicity cell and analyze different
 situations when this term is not
sufficient to conclude whether  a concrete spectral gap is actually open
or not (see Section~\ref{sec8}). Moreover, for a  technical reason, cf. Section~\ref{subsec45}, and for simplification of asymptotic structures, we make the
assumption
 \begin{equation}\label{(symm)}
\omega=\{\xi=(\xi_1,\xi_2)\in{\mathbb
R}^2:\,(\xi_1,H-\xi_2)\in\omega\}.
\end{equation}
which means that the holes possess the mirror symmetry (see Figure~\ref{fig-mirror}).
Also, for simplicity, we assume that
the boundary of $\omega$ is of class $C^\infty.$

\begin{figure}[t]
\begin{center}
\scalebox{0.35}{\includegraphics{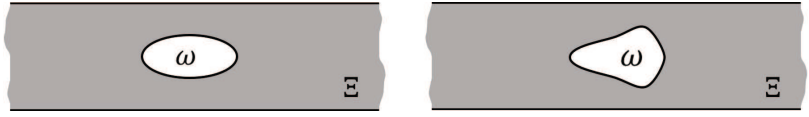}}
\end{center}
\caption{The strip $\Xi$ with two different possible geometries for the hole $\omega$.}
\label{fig-mirror}
\end{figure}

\subsection{State of art}\label{subsec12}

The continuous spectrum in a  cylindrical
waveguides of different physical nature is always a ray
$[\lambda_\dagger,+\infty)$ so that, above the cutoff value
$\lambda_\dagger\geq 0$ wave processes surely occur. The spectrum
of a periodic waveguide gets far complicated band-gap
structure \eqref{(8)} and the spectral bands implying {\it passing
zones} for waves can be separated from each other by spectral gaps
which do not permit propagation of waves with the corresponding
frequencies and, therefore, become {\it stopping zones}. This
phenomenon is used in different engineering devices, such as  wave
filters and wave dampers.

Within the Floquet--Bloch--Gelfand theory, see e.g. \cite{BLP,ReedSimon,Skrig,SHSP,Kuchbook,NaPl,CoPlaVa}, mathematical studies of
spectra with the band-gap structures need to find out the eigenvalues
of spectral elliptic boundary value problems which are posed in the
periodicity cell and involve an additional continuous parameter
$\eta\in[-\pi,\pi]$, the {\em Floquet parameter} or the {\em Gelfand dual variable}.
It is a very rare situation when such a problem admits
explicit solutions while computational methods become rather
expensive to present the whole family of {\em dispersion curves},
projections of which on the ordinate $\lambda$-axis involve the
spectral bands. As usual, variational and asymptotic methods help to
prove or disprove the existence of open spectral gaps in a certain
range of the spectrum and to estimate their geometrical
characteristics.

There are numerous  publications in which open spectral
gaps are detected due to high-contrast of coefficients in differential
operators or shape irregularities of the periodicity cells, see
\cite{Ger1,Ger2,Zhikov,na461,BaNaRu,BaNa,BaCaNaTa,BaMEP} and
\cite{na454,na501,Post,NaTa,BaTa} {\color{black}among others}. Such singular perturbations
often provide disintegration of the periodicity cells in the limit
and, as a result, the appearance of sufficiently wide gaps in the low-
and/or middle-frequency ranges of the spectrum. Both variational
and asymptotic methods have been employed in the cited papers to
detect and describe those gaps.

Another way to open spectral gaps related to the {\it splitting of
band edges}, is used in our paper. In the case when two spectral
bands of the limit problem, cf. Section~\ref{subsec22}, intersect but just
touch each other at a point, that is, there is a common edge of the bands,
small perturbations of the coefficients or of
the boundary may lead to a separation of these bands and the  opening of a
narrow gap between them, cf. Sections~\ref{subsec61}, \ref{subsec64}, \ref{subsec71} and \ref{subsec72}. This effect
is well-known in the physical literature, but its mathematical  study using operators theory and spectral perturbation methods started in \cite{na453,BoPa,BoPa_Phys,na537}.
In this paper a new type of the singular
perturbation of the periodicity cell is analyzed by means of the
homogenization technique and several ways to open spectral gaps are highlighted.

Finally, let us mention that, from a geometrical viewpoint, \cite{SOME} is the closest paper in the literature. It addresses
the Dirichlet perforation in a quantum waveguide. The asymptotic of the spectrum of the  equation \eqref{(4)}, with the Dirichlet condition $u^\varepsilon=0$ on $\partial \Pi^\varepsilon$ is considered, finding out the position and sizes of the spectral gaps and bands.
However, the results differ very much from those in this paper.
Indeed,  roughly speaking, the Dirichlet spectrum consists  of small, of order $O(\varepsilon)$, spectral bands which are separated from each other by spectral gaps of width $O(1)$.
In contrast, the Neumann spectrum here considered consists of long, of order $O(1)$, bands which are separated from each other by short spectral gaps of order $O(\varepsilon)$, or even less. The latter makes the asymptotic analysis much more complicated and delicate; in particular, it becomes multiscale in several variables, not only in the geometrical ones, but also in the Floquet parameter. As outlined above,  the justification procedure also becomes much more complicated.
For a link between the model problem in the waveguide with Neumann or Dirichlet conditions, let us mention \cite{GoNaOrPe_ruso}.

\begin{figure}[t]
\begin{center}
\resizebox{!}{5cm} {\includegraphics{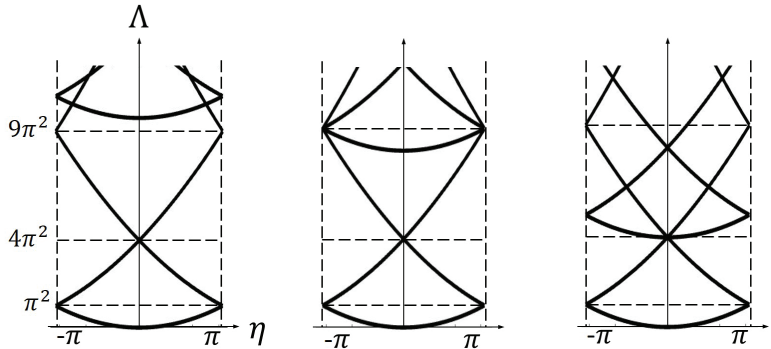}}
\caption{The dispersion curves of the limit problem in the cases  $H<1/3$,  $H=1/\sqrt{8}$ and $H=1/2.$}
\label{fig_new}
\end{center}
\end{figure}

\subsection{Architecture of the paper}\label{subsec13}

In  Section~\ref{sec2} we formulate
the model spectral problem in the periodicity cell, {b) in Figure~\ref{fig1},
which is itself a parametric spectral homogenization problem. We obtain the homogenized problem by the classical homogenization theory  in perforated media, see, e.g., \cite{nuestroOleinik},  that  is, a problem in the rectangular periodicity cell without perforations. We list explicit solutions of the homogenized problem and we study the dispersion curves which form the trusses in Figures~\ref{fig_new} and \ref{fig2}, while we classify the truss nodes, namely, the crossing points of the dispersion curves.
In Section~\ref{subsec23},  we show the convergence  result for the spectrum of the model problem towards that of the homogenized one as a consequence of another stronger one, which also  allows a perturbation of the Floquet-parameter.

In Section~\ref{sec3}  we discuss the boundary layer phenomenon arising in the
vicinity of the perforation. In particular, we examine several
solutions of the Laplace equation in the unbounded strip $\Pi$ with the only
hole $\overline \omega$, and we introduce the integral characteristics
$m_1(\Xi)$, $m_2(\Xi)$ and $m_3(\Xi)$ {\color{black}for} the Neumann problem in the
domain
 $\Xi=\Pi\setminus{\overline \omega}$
 (see Figure~\ref{fig-mirror} and \eqref{def_Xi}) with the periodicity conditions on the lateral sides, cf.~the
traditional harmonic polarization and virtual mass tensors in the
exterior domain ${\mathbb R}^2\setminus\overline{\omega}$ in \cite{PoSe}.

In Section~\ref{sec4} we perform the preliminary formal asymptotic analysis for simple eigenvalues using the method of matched asymptotic
expansions, cf.~\cite{VD,SHSP,Ilin,MaNaPl} for two scale asymptotic expansions. In Section~\ref{sec5}
we derive  error estimates in the case of simple eigenvalues  which will help us to detect open gaps
after a  much more thorough  analysis of multiple eigenvalues. The
perturbation of crossing dispersion curves require serious modifications of the standard asymptotic procedures because we can
no longer deal with a fixed Floquet parameter but we must investigate
the asymptotic behavior of the eigenvalues in a neighborhood of each
truss node, i.e., with the Floquet parameter in a certain short
interval. Recalling an idea from paper \cite{na453}, in Section~\ref{sec6} we
introduce a fast Floquet parameter to describe this behavior and detect, in
different situations, open spectral gaps of width $O(\varepsilon)$,
cf. {Figure~\ref{fig3} a)--b), which appear due to splitting of the nodes
marked with $\circ$ and $\mbox{\tiny$\square$}$ in Figure~\ref{fig2}.
This involves the characterization
 of the projections of the shaded rectangles on the ordinate axis in
Figure~\ref{fig3}, which represent the narrow gaps.

\begin{figure}[t]
\begin{center}
\resizebox{!}{5cm} {\includegraphics{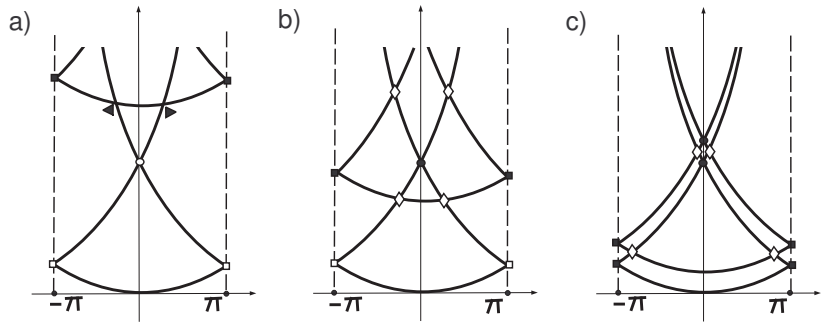}}
\caption{The dispersion curves in the limit problem in the cases
a) $H\in(1/\sqrt{8},1/2)$, b) $H\in(1/2,1)$ but $H\neq 1/\sqrt{3}$ and c) $H\in(1,+\infty).$}
\label{fig2}
\end{center}
\end{figure}

It should be noted that our detailed calculation in Section~\ref{sec5}
demon\-stra\-tes that the first correction term in the eigenvalue
asymptotics is not able to assure the gap opening and we need to discuss higher-order asymptotic terms. In fact,
Section~\ref{sec6} is devoted to deriving the formal asymptotic analysis and its justification in the case where the eigenvalue under consideration is multiple and therefore gives rise to a node  of the dispersion curves in Figure~\ref{fig2} a)--b) for the homogenized problem; in particular, we consider the nodes $(\eta_{\mbox{\small{$\circ$}}},\Lambda_{\mbox{\small{$\circ$}}})=(0,4\pi^2)$ and
$(\eta_{\mbox{\tiny$\square$}}, \Lambda_{\mbox{\tiny$\square$}})= (\pm\pi,\pi^2)$.
Providing the error estimates for the whole range of the
Floquet parameter adds the most complication to the justification
scheme (see Theorems~\ref{Th5.1}, \ref{Theorem_multiple1} and \ref{Theorem_multiple2}). The common procedure for deriving  error estimates in the
homogenization theory does not support our conclusions of opening
spectral gaps (see Section~\ref{sec7}) because the model problem in the periodicity cell
$\varpi^\varepsilon$ depends on the Floquet parameter
$\eta\in[-\pi,\pi]$ and the eigenvalues \eqref{(S2)} of the limit
problem in $\varpi^0$ change their multiplicity at the nodes.
As usual, to provide appropriate error estimates, we use a well-known result on almost eigenvalues and eigenfunctions from the spectral perturbation theory (see \cite{ViLu} and  Lemma~\ref{Lemma_Visik}).
However, we need to construct different approximations for eigenfunctions in the
vicinity of the nodes and at a certain distance from them. This is performed in Sections~\ref{subsec61} and \ref{subsec64}.
As a result, we
find proper small bounds  for asymptotic remainders  that justify our formal computations of the band edges and gap width.
It turns out that these bounds are uniform in $\eta$ but in different regions.

As regards the spectral model problem, the somehow classical convergence of the spectrum towards that of the homogenized problem is in Corollary~\ref{Lemma_convergence}.
We obtain this result as a consequence of a more general convergence result, cf.~Theorem~\ref{corollary_convergence}, which allows a certain perturbation of the Floquet variable. {\color{black} This result is  new in the literature of model problems for waveguides,
and shows    somehow a strong stability of the model problem on the parameter $\eta$.  }
It becomes essential to control the number of eigenvalues below certain constants, cf.~Propositions~\ref{PropoR12} and \ref{PropoR34}.
Theorem~\ref{Th5.1} provides some estimates which establish the closeness of  eigenvalues  depending on $\varepsilon$ and  the first three dispersion curves.    As a consequence, Corollary~\ref{Th5.2}  gives a uniform bound for the convergence rate of the first eigenvalue  at a certain distance from the nodes.
Theorems~\ref{Theorem_multiple1} and \ref{Theorem_multiple2} involve a correcting term and improve convergence rates in a small neighborhood of the  above mentioned nodes $(0,4\pi^2)$ and $(\pm\pi,\pi^2)$.
Combining the results in Sections~\ref{sec5} and \ref{sec6}, in Section~\ref{sec7}, we determine the existence of opening gaps and their width depending on $H$, cf.~Theorems~\ref{Corollary_dos} and \ref{Corollary_uno}.
Finally, in Section~\ref{sec8}, we provide some hints on open problems for other nodes in Figure~\ref{fig2} and other geometrical configurations, cf.~Figures~\ref{fig9_new} and \ref{fig10}.

\section{The model problem in the periodicity cell}\label{sec2}

In this section, we introduce the spectral model problem and its limit problem, both of which depend on the Floquet parameter $\eta\in [-\pi, \pi]$, see Sections~\ref{subsec21} and \ref{subsec22} respectively. In Section~\ref{subsec23}, we  show the spectral convergence as $\ee\to 0$, and its stability under a certain perturbation of the parameter $\eta$. In particular, this proves useful for controlling the eigenvalue number of the model problem  below some bounds.

\subsection{The FBG-transform and the quasi-periodicity conditions}\label{subsec21}

The Floquet--Bloch--Gelfand  transform   ({\em the   FBG-transform} in short),  see
\cite{Gel,ReedSimon,NaPl,Skrig,Kuchbook},
 \begin{equation}\begin{array}{c}\displaystyle
u^\varepsilon(x)\,\,\mapsto\,\,U^\varepsilon(x;\eta)=
\frac{1}{\sqrt{2\pi}}\,\sum\limits_{p\in{\mathbb Z}}
e^{-ip\eta}u^\varepsilon(x_1+p,x_2)
\end{array}\label{(9)}\end{equation}
converts the problem \eqref{(4)}, \eqref{(5)} in the infinite
waveguide $\Pi^\varepsilon$ into a boundary value problem in the
periodicity cell $\varpi^\varepsilon$ defined by \eqref{(10)}, cf.~Figure~\ref{fig1} b).

This problem consists of the differential equation
 \begin{equation}\label{(11)}
-\Delta_x U^\varepsilon(x;\eta)=\Lambda^\varepsilon(\eta)U^\varepsilon(x;\eta),
\,\, x\in\varpi^\varepsilon,
\end{equation}
the quasi-periodicity conditions on the lateral walls
 \begin{equation}\label{(12)}
 U^\varepsilon\Big(\frac{1}{2},x_2;\eta\Big)=e^{i\eta}U^\varepsilon\Big(-\frac{1}{2},x_2;\eta\Big),
\,\, x_2\in(0,H),
\end{equation}
 \begin{equation}\label{(13)}
\frac{\partial U^\varepsilon}{\partial x_1}\Big(\frac{1}{2},x_2;\eta\Big)=e^{i\eta}
\frac{\partial U^\varepsilon}{\partial x_1}\Big(-\frac{1}{2},x_2;\eta\Big),
\,\, x_2\in(0,H),
\end{equation}
and the Neumann condition on the remaining part
 of the boundary of the periodicity cell \eqref{(10)}
 \begin{equation}\label{(14)}
\partial_\nu U^\varepsilon(x{\color{black};\eta})=0,\,\, x\in\{x\in\partial
\varpi^\varepsilon:\,\vert x_1 \vert<1/2\}.
\end{equation}
Here, $\eta\in[-\pi,\pi]$ is the Floquet parameter
 while $\Lambda^\varepsilon(\eta)$ and
$U^\varepsilon(\cdot;\eta)$, respectively, are the new notations for
the eigenvalues and eigenfunctions in the model problem. Notice that
$x\in\Pi^\varepsilon$ on the left of \eqref{(9)} but
$x\in\varpi^\varepsilon$ on the right. Basic properties of the
FBG-transform can be found in the above-cited publications.

\medskip

\begin{figure}
\begin{center}
 \resizebox{!}{5cm} {\includegraphics{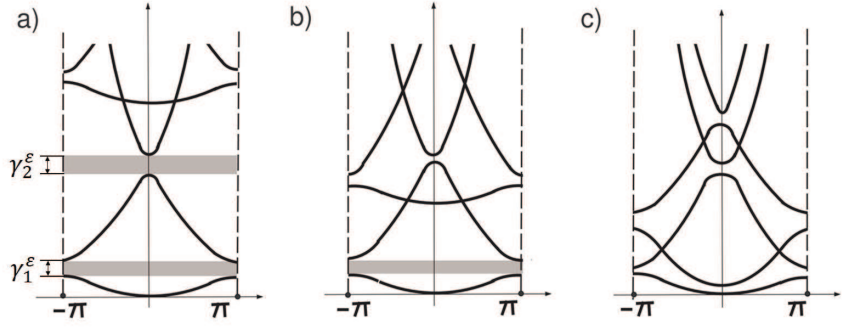}}
\caption{The dispersion curves  of the perturbed problem with the mirror symmetry of the hole $\omega$ in the cases
a) $H\in(1/\sqrt{8},1/2)$, b) $H\in(1/2,1)$ but $H\neq 1/\sqrt{3}$ and c) $H\in(1,+\infty).$
Spectral gaps are the projections of the shaded rectangles on the ordinate $\Lambda$-axis.}
\label{fig3}
\end{center}
\end{figure}

\medskip

The variational statement of the problem \eqref{(11)}--\eqref{(14)}
appeals to the integral identity \cite{Lad}
\begin{equation}\begin{array}{c}
 (\nabla_xU^\varepsilon,\nabla_xV^\varepsilon)_{\varpi^\varepsilon}=\Lambda^\varepsilon
 (U^\varepsilon,V^\varepsilon)_{\varpi^\varepsilon}\quad \forall V^\varepsilon\in H^{1,\eta}_{per}(\varpi^\varepsilon),
\end{array}\label{(16)}\end{equation}
where $H^{1,\eta}_{per}(\varpi^\varepsilon)$ is the Sobolev space of
functions satisfying the stable quasi-periodicity conditions
  {\color{black}  \eqref{(12)} and  \eqref{(13)}}. In view of the compact embedding
$H^1(\varpi^\varepsilon) \subset L^2(\varpi^\varepsilon)$, the
positive self-adjoint operator $A^\varepsilon(\eta)$ in
$L^2(\varpi^\varepsilon)$ associated  with the problem \eqref{(16)}, cf.~\cite[Section~10.2]{BiSo}, has a discrete spectrum constituting the
unbounded monotone sequence of eigenvalues
 \begin{equation}\begin{array}{c}
0\leq\Lambda^\varepsilon_1(\eta)\leq\Lambda^\varepsilon_2(\eta)\leq\dots\leq
\Lambda^\varepsilon_p(\eta)\leq\dots\to+\infty, \quad {\color{black} \mbox{ as } p\to +\infty},
\end{array}\label{(17)}\end{equation}
where their multiplicity is taken into account. Furthermore, the functions
 \begin{equation}\begin{array}{c}
[-\pi,\pi]\ni \eta\,\,\mapsto\,\, \Lambda^\varepsilon_p(\eta),\quad
p\in{\mathbb N},
\end{array}\label{(18)}\end{equation}
are continuous and $2\pi$-periodic (see again any of the above-cited references).
Hence, the sets in \eqref{(8)}
 \begin{equation}\begin{array}{c}
 \beta^\varepsilon_p=\{\Lambda^\varepsilon_p(\eta):\,
 \eta\in[-\pi,\pi]\}\subset \overline{{\mathbb R}_+}
 \end{array}\label{(19)}\end{equation}
are closed, connected, and finite segments. Indeed,  formulas
\eqref{(8)} and \eqref{(19)}  for the spectrum of the operator
$A^\varepsilon(\eta)$ and the boundary-value problem \eqref{(4)},
\eqref{(5)} are well-known in the framework of the
Floquet--Bloch--Gelfand theory.

\subsection{The limit problem and the limit dispersion curves}\label{subsec22}

In Section~\ref{subsec23} we will prove the   relationship
\begin{equation}\begin{array}{c}\displaystyle
\Lambda^\varepsilon_p(\eta)\,\,\to \Lambda^0_p(\eta)\,\,\mbox{\rm
as}\,\,\varepsilon\to +0
\end{array}\label{(27)}\end{equation}
between entries of the sequence \eqref{(17)} and those of the sequence
 \begin{equation}\begin{array}{c}
0\leq\Lambda^0_1(\eta)\leq\Lambda^0_2(\eta)\leq\dots\leq
\Lambda^0_p(\eta)\leq\dots\to+\infty,\quad {\color{black} \mbox{ as } p\to +\infty},
\end{array}\label{(S2)}\end{equation}
which consists of eigenvalues of the limit
problem in the rectangle
$$\varpi^0=\{x:\,\vert x_1 \vert<1/2, \, x_2\in(0,H)\}$$
obtained from the periodicity cell \eqref{(10)} by filling all voids,  cf.~\eqref{(32)}.  Above, the convention of repeated eigenvalues has been adopted, and the limit problem    is also referred to as {\em  homogenized problem}.  It involves the differential
equation
 \begin{equation}\begin{array}{c}
-\Delta{\color{black}{_x}} U^0(x;\eta)=\Lambda^0(\eta)U^0(x;\eta), \,\, x\in\varpi^0,
\end{array}\label{(29)}\end{equation}
the Neumann conditions on the horizontal sides of the rectangle
 \begin{equation}\begin{array}{c}\displaystyle
\frac{\partial U^0}{\partial x_2}(x_1,0{\color{black};\eta})=\frac{\partial
U^0}{\partial x_2}(x_1,H{\color{black};\eta})=0,\,\,
x_1\in\Big(-\frac{1}{2},\frac{1}{2}\Big),
\end{array}\label{(30)}\end{equation}
and the quasi-periodicity conditions on its vertical sides, cf.
\eqref{(12)} and \eqref{(13)},
 \begin{equation}\begin{array}{c}\displaystyle
U^0\Big(\frac{1}{2},x_2;\eta\Big)=
e^{i\eta}U^0\Big(-\frac{1}{2},x_2;\eta\Big),
\\\\\displaystyle
 \frac{\partial U^0}{\partial
x_1}\Big(\frac{1}{2},x_2;\eta\Big)= e^{i\eta}\frac{\partial
U^0}{\partial x_1}\Big(-\frac{1}{2},x_2;\eta\Big), \,\, x_2\in(0,H).
\end{array}\label{(25)}\end{equation}
This problem has the following explicit eigenvalues and eigenfunctions
\begin{equation}\displaystyle
\begin{array}{ll}
\Lambda^0_{jk}(\eta)=(\eta+2\pi j)^2+\pi^2\frac{k^2}{H^2},  \\ &\qquad
 j\in{\mathbb Z},\,\,k\in{\mathbb N}_0={\mathbb N}\cup\{0\}. \\
U^0_{jk}(x;\eta)=e^{i(\eta+2\pi j)x_1}\cos\Big(\pi  k\frac{x_2}{H}\Big),
\end{array}
\label{(31)}
\end{equation}
It should be mentioned that renumeration of the eigenvalues in
\eqref{(31)} is needed to compose the monotone sequence
\eqref{(S2)}.

\medskip
\begin{figure}
\begin{center}
\resizebox{!}{5cm} {\includegraphics{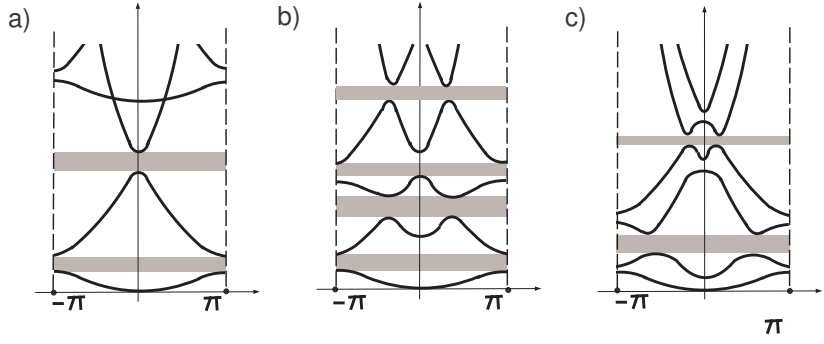}}
\caption{The hypothetical dispersion curves {\color{black}of the perturbed problem without the mirror symmetry of the hole $\omega$ in the cases
a) $H\in(1/\sqrt{8},1/2)$, b) $H\in(1/2,1)$ but $H\neq 1/\sqrt{3}$ and c) $H\in(1,+\infty).$}
Many more spectral gaps are opened than in  {\color{black}Figure~\ref{fig3}}.}
\label{fig4}
\end{center}
\end{figure}

\medskip

Graphs of several
eigenvalues  \eqref{(31)} of the problem  \eqref{(29)}--\eqref{(25)}, that is, {\em the dispersion curves,}
are drawn in {\color{black}Figure~\ref{fig2}} a)--c), respectively, for the following cases:
$$
\mbox{a)}\quad H\in\Big(\frac{1}{\sqrt{8}},\frac{1}{2}\Big) \qquad
\mbox{b)}\quad H\in\Big(\frac{1}{2},1\Big) \qquad
\mbox{c)}\quad H\in(1,+\infty).$$
Figure~\ref{fig_new} also displays dispersion curves of the limit problem in the cases  $H\in(0,1/3)$,  $H=1/\sqrt{8}$ and $H=1/2$, respectively.
Figures~\ref{fig_new} and \ref{fig2} show the great variety of behaviors of the dispersion curves  and, consequently, the complexity to open spectral gaps depending on $H$.

As was mentioned in Section~\ref{subsec12} and depicted in Figure~\ref{fig3}, due to the
perturbation by holes of the periodicity cell, the dispersion curves for problem \eqref{(11)}--\eqref{(14)} may separate at the truss nodes marked with the signs
$\mbox{\tiny$\square$}$ and $\circ$ in Figure~\ref{fig2} (see Section~\ref{sec7}).

\subsection{The convergence results}\label{subsec23}

First, we obtain some estimates for the eigenvalues  and provide an extension for eigenfunctions over the whole $\varpi^0$ necessary to show the convergence.
The constants appearing throughout the section $c_m$ and  $C_m$, are independent of both variables $ \ee$ and $\eta$.

\medskip

Let $\eta\in[-\pi,\pi]$ be fixed
and let $U^\varepsilon_m(\cdot;\eta)\in
H^{1,\eta}_{per}(\varpi^\varepsilon)$ be an eigenfunction of the
problem \eqref{(16)} corresponding to the eigenvalue
$\Lambda^\varepsilon_m(\eta)$. The {\color{black}minimax} principle assures the estimate
\begin{equation}\begin{array}{c}\displaystyle
\Lambda^\varepsilon_m(\eta)\leq c_m\,\,\mbox{\rm for}
\,\,\varepsilon\in(0,\varepsilon_m]
\end{array}\label{(KK1)}\end{equation}
with some positive $\varepsilon_m$ and $c_m$.
{\color{black}
Indeed, we write
$$
\Lambda^\varepsilon_m(\eta)=
\min_{E_m^\varepsilon\subset H^{1,\eta}_{per}(\varpi^\varepsilon)} \max_{V\in E_m^\varepsilon, V\neq 0}
\frac{(\nabla_x V, \nabla_x V)_{\varpi^\varepsilon}}{(V,V)_{\varpi^\varepsilon}},
$$
where the minimum is computed over the set of subspaces $E_m^\varepsilon$ of $H^{1,\eta}_{per}(\varpi^\varepsilon)$
with dimension $m$. To prove \eqref{(KK1)}, we take a particular $E_m^\varepsilon$ that we construct as follows:
we consider the eigenfunctions corresponding to the $m$ first eigenvalues of the mixed
eigenvalue problem in the rectangle $(1/4, 1/2) \times (0,H)$, with Neumann condition
on the part of the boundary $\{1/2\} \times (0,H)$, and Dirichlet condition on the rest of
the boundary. We extend these eigenfunctions by zero for $x \in [0,1/4]\times (0,H)$, and by
symmetry for $x\in [-1/2,0]\times (0,H)$. Finally, multiplying these eigenfunctions by $e^{i\eta x_1}$ gives
$E_m^\varepsilon$ and the right hand side of \eqref{(KK1)} holds.}

\medskip

Let us construct an extension for the eigenfunctions to be uniformly bounded in $H^1(\varpi^0)$.
Being normalized in $L^2(\varpi^\varepsilon)$, on account of \eqref{(KK1)}, the eigenfunctions of \eqref{(16)} satisfy
$$
\|\nabla_x
U^\varepsilon_m;L^2(\varpi^\varepsilon)\|^2=\Lambda^\varepsilon_m(\eta)\|
U^\varepsilon_m;L^2(\varpi^\varepsilon)\|^2=\Lambda^\varepsilon_m(\eta)\leq
c_m\,\,\mbox{\rm for} \,\,\varepsilon\in(0,\varepsilon_m].
$$
Therefore, we can   extend $U^\varepsilon_m(\cdot;\eta)$ over the holes \eqref{(32)}
and obtain a function ${\widehat U}^{\,\varepsilon}_m(\cdot;\eta)\in
H^{1,\eta}_{per}(\varpi^0)$ such that for
$\varepsilon\in(0,\varepsilon_m]$,
\begin{equation}\begin{array}{c}\displaystyle {\widehat
U}^{\,\varepsilon}_m(x;\eta)=U^\varepsilon_m(x;\eta) \,\, \mbox{  for }
 x\in\varpi^\varepsilon \quad   \mbox{ and } \vspace{0.2cm}\\
\|{\widehat U}^{\,\varepsilon}_m(\cdot;\eta); H^1(\varpi^0)\|\leq C
 \|U^\varepsilon_m;H^1(\varpi^\varepsilon)\|\leq C_m,
\end{array}\label{(KK3)}\end{equation}
 see, for example, Section I.4.2 in \cite{Oleinikbook} for {\color{black} such} an extension. In addition, from \eqref{(KK3)} and the estimate
$$
\|V;L^2(\varpi^0 \cap \{\vert x_1 \vert<\varepsilon\})\|^2 \leq C\varepsilon \|V;H^1(\varpi^0 )\|^2\quad \forall V\in H^1(\varpi^0)
$$
(see, e.g., Lemma~2.4 in \cite{Marchenkobook}), they satisfy
\begin{equation}\begin{array}{c}\displaystyle
\|{\widehat U}^{\,\varepsilon}_m(\cdot;\eta);
L^2(\varpi^0\setminus\varpi^\varepsilon)\|^2\leq c\varepsilon
 \|U^\varepsilon_m;H^1(\varpi^\varepsilon)\|^2\leq C_m\varepsilon.
\end{array}\label{(KK4)}\end{equation}

\medskip

Moreover, the following results state the spectral convergence for problem \eqref{(16)}, {\color{black} as $\ee \to 0$. As a matter of fact, Theorem \ref{corollary_convergence}  also shows the stability of the limit of the spectrum of the perturbation problem  \eqref{(11)}--\eqref{(14)}  under any perturbation of the Floquet-parameter $\eta$.}

\begin{Theorem}\label{corollary_convergence}
For   each sequence  $\{(\varepsilon_k, \eta_k)\}_{k=1}^\infty$ such that $\ee_k\to 0$ and $\eta_k\to \widehat \eta \in [-\pi,\pi]$ as $k\to +\infty$, the eigenvalues $\Lambda_m^{\varepsilon_k}(\eta_k)$ of problem \eqref{(11)}--\eqref{(14)} when $(\ee,\eta)\equiv (\ee_k,\eta_k)$  converge, as $ k\to +\infty,$ towards the eigenvalues of problem \eqref{(29)}--\eqref{(25)} for $\eta\equiv \widehat \eta$ and there is conservation of the multiplicity. Namely, for each $m=1,2,\cdots$, the convergence
$$
\Lambda^{\ee_k}_{m}(\eta_k)\to \Lambda^{0}_m (\widehat\eta),\quad \mbox{ as }   k\to + \infty,
$$
holds, where $\Lambda^0_m(\widehat \eta)$ is the m-th eigenvalue in the sequence
$$
0\leq\Lambda^0_1(\widehat \eta)\leq\Lambda^0_2(\widehat \eta)\leq\dots\leq
\Lambda^0_m(\widehat \eta)\leq\dots\to+\infty,\quad \mbox{ as } m\to +\infty,
$$
of eigenvalues of \eqref{(29)}--\eqref{(25)} for $\eta\equiv \widehat \eta$, which are counted according to their multiplicities.   In addition,   we can extract a subsequence, still denoted by $\varepsilon_k$, such that the extension
$\{{\widehat U}_m^{\,\varepsilon_k}\}_{m=1}^\infty$ converge in $L^2(\varpi^0)$, as $\ee_k\to 0$, towards the eigenfunctions of \eqref{(29)}--\eqref{(25)} for $\eta\equiv \widehat \eta$, which form an orthonormal basis of $L^2(\varpi^0)$.
\end{Theorem}

\begin{proof}
For each $\ee_k$ and $\eta_k$, we write the integral equation satisfied by the eigenvalue $\Lambda_m^{\varepsilon_k}(\eta_k)$  and   corresponding eigenfunction $ {  U}_m^{\,\varepsilon_k} (\cdot,\eta_k)$:
\begin{equation}\begin{array}{c}
(\nabla_x U_m^{\varepsilon_k},\nabla_x V^{\varepsilon_k})_{\varpi^{\varepsilon_k}}=\Lambda_m^{\varepsilon_k}(\eta_k)
(U_m^{\varepsilon_k},V^{\varepsilon_k})_{\varpi^{\varepsilon_k}}\quad \forall V^{\varepsilon_k}\in H^{1,\eta_k}_{per}(\varpi^{\varepsilon_k}),
\end{array}\label{(16r)}\end{equation}

Since the constants appearing \eqref{(KK3)} and \eqref{(KK4)}  are independent of $\eta\in [-\pi,\pi]$ and $\ee$, the estimates hold  $\ee$ and $\eta$ ranging in sequences of  $\{\ee_k\}_k$ and $\{\eta_k\}_k$, in the statement of the theorem.
Thus,  we use an extension  of $U^{\varepsilon_k}_m(\cdot;\eta)$ over the holes \eqref{(32)}
 denoted by ${\widehat U}^{\,\varepsilon_k}_m(\cdot;\eta)\in
H^{1,\eta_k}_{per}(\varpi^0)$,   which  satisfies
\begin{equation}\begin{array}{c}\displaystyle
\|{\widehat U}^{\,\varepsilon_k}_m(\cdot;\eta_k); H^1(\varpi^0)  \|\leq C_m,  \mbox{ \quad and \quad} \vspace{0.2cm}\\
\displaystyle \|{\widehat U}^{\,\varepsilon_k}_m(\cdot;\eta_k);
L^2(\varpi^0\setminus\varpi^{\varepsilon_k})\|^2\leq   C_m\varepsilon_k \|{  U}^{\,\varepsilon_k}_m(\cdot;\eta_k); H^1(\varpi^{\ee_k})\|^2 \leq C_m \ee_k;
\end{array}\label{(KK34)}\end{equation}
for sufficiently small $\ee_k$, with a constant $C_m$ independent of $\ee_k$ and $\eta_k$.

In view of \eqref{(KK1)} and \eqref{(KK34)},   for each fixed $m\geq 1$ and for each  subsequence of $ k $,  we can extract a subsequence, still denoted by   $k$,
such that
\begin{equation}\label{(KK5r)}
\begin{array}{c}\displaystyle
\Lambda^{\varepsilon_k}_m(\eta_k)\rightarrow
\mbox{\boldmath$ \widehat \Lambda$}^0_m, \quad \mbox{as  } k\to  +\infty, \vspace{0.2cm}\\
\!\!{\widehat U}^{\,\varepsilon_k}_m(\cdot;\eta_k)\rightarrow {\bf \widehat
U}^0_m \,\,\mbox{\rm weakly\,\,in}\,\,H^1(\varpi^0)\,\,
\mbox{\rm and\,\,strongly\,\,in}\,\, L^2(\varpi^0),\quad   \mbox{as  } k\to + \infty,
\end{array}\end{equation}
for some real number   $\mathbf{\widehat \Lambda}^0_m \geq 0$  and  some function    ${\bf\widehat
U}^0_m\in H^1(\varpi^0)$ which we determine below depending on $\widehat \eta$.

We take a test function $V\in C^\infty(\overline{\varpi^0})$
verifying the    periodicity condition at the lateral sides of $\varpi^0$ and consider $V^k=Ve^{i \eta_k x_1}$ which satisfies the quasi-periodicity condition in \eqref{(25)} with $\eta\equiv \eta_k$. For $V^{\ee_k}\equiv V^k$, we
rewrite the integral identity \eqref{(16r)} in the form
\begin{equation}
(\nabla_x {\widehat U}^{\,\varepsilon_k}_m,
\nabla_xV^k)_{\varpi^0}-\Lambda^{\varepsilon_k}_m(\eta_k)({\widehat
U}^{\,\varepsilon_k}_m,V^k)_{\varpi^0} =
(\nabla_x {\widehat U}^{\,\varepsilon_k}_m,\nabla_x V^k)_{\varpi^0\setminus\varpi^{\varepsilon_k}}-\Lambda^{\varepsilon_k}_m(\eta)({\widehat
U}^{\,\varepsilon_k}_m,V^k)_{\varpi^0\setminus\varpi^{\varepsilon_k}}.
\label{(KK6r)}\end{equation}
According to \eqref{(KK34)} and \eqref{(KK1)}, the modulo of the
right-hand side of \eqref{(KK6r)} does not exceed
\begin{equation}\label{numc}\begin{array}{c}\displaystyle
(1+c_m \ee_k)\|{\widehat
U}^{\,\varepsilon_k}_m;H^1(\varpi^0)\|\,\|V^k;H^1(\varpi^0\setminus\varpi^{\varepsilon_k})\|\leq
c_m(V^k)
\mbox{\rm mes}_2(\varpi^0\setminus\varpi^{\varepsilon_k})\\\\
\leq
c_m \max_{x\in \overline{\varpi^0}} (\vert V^k(x)\vert, \vert  \partial_{x_1} V^k(x)\vert, \vert \partial_{x_2}V^k(x)\vert) \varepsilon_k \leq C_m(V)(1+\eta_k) \varepsilon_k. \end{array}\end{equation}
Thus, the right hand side of \eqref{(KK6r)} converges towards zero as $k\to +\infty$. Let us analyze the left hand side in further detail.

In order to do this, let us consider the well-known change $\widehat U_m^{\ee_k} (\cdot; \eta_k )=V_m^{\ee_k} (\cdot; \eta_k) e^{ i \eta_k x_1}$
which converts the Laplacian into the differential operator
$$-\Big(\frac{\partial}{\partial x_1} + i \eta  \Big)  \Big(\frac{\partial}{\partial x_1} + i\eta  \Big)- \frac{\partial^2}{\partial x_2^2},
$$
while the $\eta_k$ quasi-periodicity condition for $ \widehat U_m^{\ee_k} (\cdot; \eta_k) $ becomes a periodicity condition for  $V_m^{\ee_k} (\cdot; \eta_k)$.
Consequently,  $V_m^{\ee_k} (\cdot; \eta_k)\in H^{1 }_{per} (\varpi^0 )$ and since  $\eta_k\to \widehat \eta $ as $k\to +\infty$ (equivalently, $\ee_k\to 0$), we also have the  bound  for $V_m^{\ee_k}$  in $    H^{1 }_{per} ( \varpi^0 )$   which holds uniformly in $\eta_k$ and $\ee_k$, and  a  convergence of $V_m^{\ee_k}$  (by subsequences, still denoted by $k$)    towards a function  $V^0_m\in H^{1 }_{per} (\varpi^0 ) $ holds in  the weak topology of $   H^{1 } (\varpi^0  )$. Let us show that
\begin{equation}\label{pereta} V_m^0={\bf \widehat   U}_m^0   e^{-i \widehat \eta  x_1}.\end{equation}
To do this, it suffices to show
$$\Vert \widehat U_m^{\ee_k} (\cdot; \eta_k) e^{-i \eta_k x_1} -{\bf \widehat U}_m^0  e^{-i \widehat \eta  x_1};  L^2(\varpi^0)\Vert \to 0\quad  \mbox{ as } k\to +\infty,$$
and we check this by  considering
\begin{equation*}
\begin{split}
&\Vert \widehat U_m^{\ee_k} (\cdot; \eta_k) e^{-i \eta_k x_1} - {\bf \widehat  U}_m^0  e^{-i \widehat \eta  x_1};  L^2(\varpi^0)\Vert\\
&\quad  \leq
\Vert \big( \widehat U_m^{\ee_k}(\cdot; \eta_k)    - {\bf \widehat U}_m^0 \big) e^{-i  \eta_k  x_1};  L^2(\varpi^0)\Vert
+ \Vert {\bf \widehat U}_m^0   \big (e^{-i \eta_k x_1} -   e^{-i \widehat \eta  x_1}\big); L^2(\varpi^0)\Vert,
\end{split}
\end{equation*}
the convergence \eqref{(KK5r)}, the smoothness of the exponential function and the convergence of $\eta_k$.

Introducing  the change $V^k=Ve^{i \eta_k x_1}$  in \eqref{(KK6r)}, we write
\begin{equation}\begin{array}{l}\displaystyle
(\nabla_x  {  \widehat U}^{\,\varepsilon_k}_m   ,
\nabla_x(V e^{i   \eta_k  x_1}))_{\varpi^0}-\Lambda^{\varepsilon_k}_m(\eta_k) ( {  \widehat U}^{\,\varepsilon_k}_m ,V e^{ i   \eta_k  x_1})_{\varpi^0} \vspace{0.2cm}\\
=(\nabla_x {  \widehat U}^{\,\varepsilon_k}_m   ,
\nabla_x(V e^{i \widehat  \eta  x_1}))_{\varpi^0} + (\nabla_x  {  \widehat U}^{\,\varepsilon_k}_m,
\nabla_x(V  (e^{i   \eta_k  x_1} -e^{i \widehat  \eta  x_1})  ))_{\varpi^0}\vspace{0.2cm}\\
\quad -\Lambda^{\varepsilon_k}_m(\eta_k)({
\widehat U}^{\,\varepsilon_k}_m ,V e^{ i  \widehat \eta   x_1})_{\varpi^0} -  \Lambda^{\varepsilon_k}_m(\eta_k) ( {  \widehat U}^{\,\varepsilon_k}_m,V (e^{i   \eta_k  x_1} -e^{i \widehat  \eta  x_1})   )_{\varpi^0} \,.
\end{array}\label{(KK6rb)}\end{equation}

Let us show
\begin{equation}\label{vcon} V   e^{i   \eta_k  x_1}  \to   V  e^{i \widehat  \eta  x_1}  \mbox{ in } H^1(\varpi^0) \quad  \mbox{ as } k\to + \infty \end{equation}
and therefore,  from \eqref{(KK5r)}, also the convergence
$$
(\nabla_x  {  \widehat U}^{\,\varepsilon_k}_m,
\nabla_x(V  (e^{i   \eta_k  x_1} -e^{i \widehat  \eta  x_1})  ))_{\varpi^0} -  \Lambda^{\varepsilon_k}_m(\eta_k)( {  \widehat U}^{\,\varepsilon_k}_m,V (e^{i   \eta_k  x_1} -e^{i \widehat  \eta  x_1})  ))_{\varpi^0}  \lik 0, $$ 
holds.
Indeed, on account of the smoothness of $V$, we have
\begin{equation*}
\begin{split}
&\Vert V  (e^{i   \eta_k  x_1}  - e^{i \widehat  \eta  x_1} ); H^1(\varpi^0) \Vert^2 \\
&\,\, \leq C(V)\Big(  \Vert     e^{i   \eta_k  x_1}   -e^{i \widehat  \eta  x_1} ; L^2(\varpi^0) \Vert^2
+\Vert    \eta_k e^{i   \eta_k  x_1}-\widehat \eta   e^{i \widehat  \eta  x_1} ; L^2(\varpi^0) \Vert^2\Big)\lik 0,
\end{split}
\end{equation*}
for a certain  positive constant $C(V)$, and this shows \eqref{vcon}.

Then, taking limits in \eqref{(KK6rb)} as $k\to \infty$,  on account of \eqref{(KK5r)}, \eqref{numc}, \eqref{vcon}  and  \eqref{pereta},  we obtain
the integral
identity
$$
(\nabla_x {\bf \widehat U}^0_m,\nabla_x (Ve^{i \widehat  \eta  x_1} ))_{\varpi^0}-\mbox{\boldmath$\widehat \Lambda$}^0_m ({\bf
\widehat U}^0_m,Ve^{i \widehat  \eta  x_1})_{\varpi^0}= 0 \quad \forall  V\in C^\infty(\overline{\varpi^0})\cap H^{1}_{per}(\varpi^0),
$$
while, by a completion argument, we can write
$$
  (\nabla_x {\bf \widehat U}^0_m,\nabla_x  (Ve^{i \widehat  \eta  x_1}))_{\varpi^0}-  \mbox{\boldmath$\widehat \Lambda$}^0_m({\bf
\widehat U}^0_m, Ve^{i \widehat  \eta  x_1} )_{\varpi^0}= 0
  \quad \forall   V\in H^{1}_{per}(\varpi^0),
$$
or equivalently,
\begin{equation}\label{(16limitrb)}
  (\nabla_x {\bf \widehat U}^0_m,\nabla_x  V)_{\varpi^0}-\mbox{\boldmath$\widehat \Lambda$}^0_m({\bf
\widehat U}^0_m, V  )_{\varpi^0}= 0
  \quad \forall   V\in H^{1,\widehat \eta}_{per}(\varpi^0).
\end{equation}
On account of \eqref{pereta}, also ${\bf \widehat U}_m^0 \in   H^{1,\widehat \eta}_{per}(\varpi^0)$ and, consequently, \eqref{(16limitrb)} is nothing but the weak
formulation of  \eqref{(29)}--\eqref{(25)} for $\eta\equiv\widehat  \eta $.

Furthermore,
$$
1=\|U^{\varepsilon_k}_m (\cdot ,\eta_k);L^2(\varpi^{\varepsilon_k})\|^2=\|{\widehat
U}^{\,\varepsilon_k}_m(\cdot ,\eta_k);L^2(\varpi^0)\|^2-\|{\widehat
U}^{\,\varepsilon_k}_m(\cdot ,\eta_k);L^2(\varpi^0\setminus\varpi^{\varepsilon_k})\|^2
$$
and taking limits as $k\to +\infty$, on account of  \eqref{(KK34)} and \eqref{(KK5r)}, gives
$$\|{\bf \widehat U}^0_m;L^2(\varpi^0)\|^2=1.$$
This, together with \eqref{(16limitrb)}, identifies $(\mbox{\boldmath$ \widehat \Lambda$}^0_m , {\bf \widehat U}_m^0) $ with an eigenpair of \eqref{(29)}--\eqref{(25)} when $\eta\equiv\widehat  \eta $.

Therefore, we conclude that $ {\bf \widehat \Lambda}_m^0$ is an eigenvalue with the corresponding eigenfunction  ${\bf \widehat U}_m^0$ of the limit problem \eqref{(29)}--\eqref{(25)} when $\eta\equiv\widehat  \eta $, and we get a dependence of ${  \bf \widehat \Lambda}^0_m$ on $\widehat \eta$, so we write ${ \bf \widehat \Lambda}^0_m:={ \bf \widehat \Lambda}^0_m(\widehat \eta)$.

Note that the extracted subsequences and limits may depend on $m$.  However,  using a diagonalization argument, for each   sequence of  $k$, we can extract another subsequence, still denoted by $k$,  but independent of $m$,  such that \eqref{(KK5r)}  holds $\forall m\in {\mathbb N}$.
Then, by the construction, we have obtained an increasing sequence
\begin{equation}\label{(S2r)}
0\leq {\bf \widehat \Lambda}^0_1(\widehat \eta)\leq{\bf \widehat \Lambda}^0_2(\widehat \eta)\leq\dots\leq
{ \bf \widehat \Lambda}^0_m(\widehat \eta)\leq\dots.
\end{equation}
In what follows we prove that the sequence $\{{\bf  \widehat \Lambda}^0_m(\widehat \eta)\}_{m=1}^\infty$  converges towards infinity while the whole sequence coincides with that in \eqref{(S2)} when $\eta\equiv \widehat \eta$.

From   the orthogonality of  $U^{\varepsilon_k}_m(\cdot; \eta_k)$ in $L^2(\varpi^{\varepsilon_k})$, we write
 $$ \displaystyle
({\widehat U}^{\,\varepsilon_k}_m(\cdot; \eta_k),{\widehat U}^{\,\varepsilon_k}_n(\cdot; \eta_k))_{\varpi^0}=
({\widehat U}^{\,\varepsilon_k}_m(\cdot; \eta_k),{\widehat U}^{\,\varepsilon_k}_n(\cdot; \eta_k))_{\varpi^0\setminus\varpi^{\varepsilon_k}} \quad \forall m,n \in    {\mathbb N}, \,\, m\not=n
$$
and use \eqref{(KK34)} to get the orthogonality of the eigenfunctions $\{{\bf \widehat U}^0_m\}_{m=1}^\infty$ in $L^2(\varpi^0)$. This shows that the sequence in \eqref{(S2r)} converges towards infinity as  $m\to \infty$.

In order to  show that with the above limits \eqref{(S2r)} we reach all the eigenvalues in the entry \eqref{(S2)} when $\eta\equiv\widehat\eta$, namely,  that  $ \Lambda^0_m(\widehat\eta)=\mathbf{\widehat \Lambda}^0_m(\widehat \eta)$, it suffices to show that the set  $\{{\bf\widehat  U}^0_m\}_{m=1}^\infty$ forms a basis of $L^2(\varpi^0)$. Indeed, this is a  classical   process of contradiction  (see, for instance, Section III.1 of  \cite{Oleinikbook}
and Section III.9.1 of \cite{Attouch}).
In this way, we have proved that \eqref{(27)} holds for any $p\in  \mathbb{N}$,  where  $\{{\bf \widehat \Lambda}^0_p(\widehat \eta)\}_{p=1}^\infty$ are the set of eigenvalues of \eqref{(29)}--\eqref{(25)} when $\eta\equiv \widehat \eta$  and the eigenfunctions  $\{{\bf \widehat  U}^0_p \}_{p=1}^\infty$ form an orthonormal  basis of $L^2(\varpi^0)$. Consequently, the sequence \eqref{(S2r)}  coincides with \eqref{(S2)}, and the theorem is proved.
\end{proof}

\begin{Corollary}\label{Lemma_convergence}
For any $\eta\in[-\pi,\pi]$, the eigenvalues $\Lambda_m^\varepsilon(\eta)$ of problem \eqref{(11)}--\eqref{(14)} in the sequence \eqref{(17)} converge, as $\varepsilon\to +0,$ towards the eigenvalues of problem \eqref{(29)}--\eqref{(25)} in the sequence \eqref{(S2)} and there is conservation of the multiplicity. In addition, for each sequence, we can extract a subsequence, still denoted by $\varepsilon$, such that the extensions
 $\{{\widehat U}_m^{\,\varepsilon}\}_{m=1}^\infty$ converge in $L^2(\varpi^0)$, as $\ee\to 0$, towards the eigenfunctions of \eqref{(29)}--\eqref{(25)}, which form an orthonormal basis of $L^2(\varpi^0)$.
Also, for each eigenfunction $U_p^0$ of \eqref{(29)}--\eqref{(25)} associated with the eigenvalue $\Lambda_p^0(\eta)$ of multiplicity $n_p$, $\Lambda_p^0(\eta)=\Lambda_{p+1}^0(\eta)=\cdots=\Lambda_{p+n_p-1}^0(\eta)$ in \eqref{(S2)}, there is a linear combination $U^\varepsilon$ of eigenfunctions corresponding to the eigenvalues $\Lambda_p^\varepsilon(\eta), \, \Lambda_{p+1}^\varepsilon(\eta), \, \dots, \, \Lambda_{p+n_p-1}^\varepsilon(\eta)$ in \eqref{(17)}, that converges towards $U_p^0$ in $L^2(\omega).$
\end{Corollary}

\begin{proof}
The first two assertions hold from Theorem~\ref{corollary_convergence} taking $\eta_k\equiv\eta$ fixed for all $k$ with minor modifications.
Moreover, the last result is obtained using the technique in Theorem~III.1.7 in \cite{Oleinikbook}.
\end{proof}

\medskip

In addition to the bounds \eqref{(KK1)}, we state the following lower bounds for the first eigenvalues of problem \eqref{(16)}.

\begin{Proposition}\label{PropoR12}
Let $H\in(0,1)$. Let $\delta_1>0$ (and $<\pi$). Then, there exists a positive constant $\varepsilon_1=\varepsilon_1(H,\delta_1)$ such that the entries
$\Lambda^\varepsilon_2(\eta)$ and $\Lambda^\varepsilon_3(\eta)$ of the eigenvalue sequence \eqref{(17)} meet the estimates
\begin{eqnarray}
\Lambda_2^\varepsilon(\eta)>\pi^2+K_1\qquad &&\mbox{for } \eta\in I_1=[-\pi+\delta_1,\pi-\delta_1], \, \varepsilon<\varepsilon_1, \label{EstL2} \\
\Lambda_3^\varepsilon(\eta)>\pi^2+K_2\qquad &&\mbox{for } \eta\in [-\pi,\pi],
\, \varepsilon<\varepsilon_1,  \label{EstL3}
\end{eqnarray}
where
\begin{equation}\label{def_K1yK2}
K_1=\min\left\{2\pi\delta_1,\frac{\pi^2(1-H^2)}{2H^2}\right\},\qquad
K_2=\min\left\{2\pi^2,\frac{2\pi^2(1-H^2)}{3H^2}\right\}.
\end{equation}
\end{Proposition}

\begin{proof}
We proceed by contradiction, denying \eqref{EstL2}. This implies that for any $\varepsilon_1>0$ there exist $\varepsilon<\varepsilon_1$ and $\eta\in I_1$  such that
$\Lambda_{2}^{\ee} (\eta) \leq  \pi^2+ K_1.$
It is clear that we can take a sequence $\{\varepsilon_k\}_{k=1}^\infty$ such that $\varepsilon_k\to 0$ as $k\to +\infty$, and an associated sequence $\{ {\eta_k}\}_{k=1}^\infty$ which is bounded from above and from below, and satisfies
\begin{equation}\label{etafijobox}
\Lambda_{2}^{\ee_k} (\eta_k) \leq  \pi^2+ K_1.
\end{equation}
By subsequences,  we can  construct a sequence (still denoted by $k$) such that
$$(\eta_k,  \ee_k)\to (\widehat\eta, 0) \hbox{ as } k\to +\infty,$$
for certain  $\widehat\eta\in I_1$.
Let us show that this last assertion leads us to a contradiction.

According to Theorem \ref{corollary_convergence}, taking limits in \eqref{etafijobox}, we get
$ \Lambda^0_2(\widehat\eta) \leq  \pi^2+K_1. $
Since $\Lambda^0_2(\widehat\eta)$
can only be $(2\pi-\vert\widehat\eta\vert)^2$ or $(\pi/H)^2+ \vert\widehat\eta\vert^2$ with $\widehat\eta\in I_1$, we obtain
$$
\min\left\{(\pi+\delta_1)^2,\frac{\pi^2}{H^2}\right\}\leq\Lambda^0_2(\widehat\eta)\leq \pi^2+K_1,
$$
and this contradicts the hypothesis on the chosen $K_1$.
Consequently, \eqref{EstL2} is proved.

Analogously, we proceed by contradiction, denying \eqref{EstL3}. Thus, as in the proof of \eqref{EstL2}, there exists $\widehat\eta\in [-\pi,\pi]$ such that
$
\Lambda^0_3(\widehat\eta) \leq  \pi^2+K_2.
$
Since $\Lambda^0_3(\widehat\eta)$ can be only $(2\pi+\vert\widehat\eta\vert)^2$ or $(\pi/H)^2+ \vert\widehat\eta\vert^2$ ($\widehat\eta\in [-\pi,\pi]),$
$$
\min\left\{4\pi^2,\frac{\pi^2}{H^2}\right\}\leq\Lambda^0_3(\widehat\eta)\leq  \pi^2+K_2,
$$
and this contradicts the hypothesis on the chosen $K_2$.
\end{proof}

\begin{Proposition}\label{PropoR34}
Let $H\in(0,1/2)$.  Let $\delta_3>0$ (and $<\pi$). Then, there exists a positive constant $\varepsilon_1=\varepsilon_1(H,\delta_3)$ such that the entries
$\Lambda^\varepsilon_3(\eta)$ and $\Lambda^\varepsilon_4(\eta)$ of the eigenvalue sequence \eqref{(17)} meet the estimates
\begin{eqnarray}
\Lambda_3^\varepsilon(\eta)>4\pi^2+K_3\qquad &&\mbox{for } \eta\in I_3=[-\pi,-\delta_3]\cup[\delta_3,\pi], \, \varepsilon<\varepsilon_1,  \label{EstL3b}\\
\Lambda_4^\varepsilon(\eta)>4\pi^2+K_4\qquad &&\mbox{for } \eta\in [-\pi,\pi], \, \varepsilon<\varepsilon_1, \label{EstL4}
\end{eqnarray}
where
\begin{equation}\label{def_K3yK4}
K_3=\min\left\{4\pi\delta_3,\frac{\pi^2(1-4H^2)}{H^2}\right\}>0,\qquad K_4=\min\left\{4\pi^2,\frac{\pi^2(1-4H^2)}{2H^2}\right\}>0.
\end{equation}
\end{Proposition}

\begin{proof}
First we prove \eqref{EstL3b} with the same analysis as in Proposition \ref{PropoR12}.  We proceed by contradiction, denying \eqref{EstL3b}. Thus, as in the proof of \eqref{EstL2}, there exists $\widehat\eta\in I_3$ such that
$
\Lambda^0_3(\widehat\eta) \leq  4\pi^2+K_3.
$
Since $H<1/2$,  $\Lambda^0_3(\widehat\eta)$ can only be  $(2\pi+ \vert\widehat\eta\vert)^2$ or $(\pi/H)^2+ \vert\widehat\eta\vert^2$ ($\widehat\eta\in I_3$). Therefore, we get
$$
\min\left\{(2\pi+\delta_3)^2,\frac{\pi^2}{H^2}+\delta^2_3\right\}\leq\Lambda^0_3(\widehat\eta)\leq  4\pi^2+K_3,
$$
and this contradicts the hypothesis on the chosen $K_3$.

Finally, we proceed by contradiction, denying \eqref{EstL4}. Thus, as in the proof of \eqref{EstL2}, there exists $\widehat\eta\in [-\pi, \pi]$ such that
$\Lambda^0_4(\widehat\eta) \leq  4\pi^2+K_4.
$
Since $H<1/2$, $\Lambda^0_4(\widehat\eta)$ can only be  $(4\pi-\vert\widehat\eta\vert)^2$ or $(\pi/H)^2+ \vert\widehat\eta\vert^2$ for $\widehat\eta\in [-\pi,\pi]$. Therefore, we get
$$
\min\left\{9\pi^2,\frac{\pi^2}{H^2}\right\}\leq\Lambda^0_4(\widehat\eta) \leq  4\pi^2+K_4,
$$
and this contradicts the hypothesis on the chosen $K_4$.
\end{proof}

\section{The boundary layer phenomenon in the periodicity cell}\label{sec3}

The traditional results of the homogenization theory given in Corollary~\ref{Lemma_convergence} do not help to conclude on the splitting of band edges and in this
section we examine special solutions of a boundary-value problem in
the strip $\Pi$ with the only hole $\overline{\omega}$ of unit size.
Let us define
\begin{equation}\label{def_Xi}
\Xi:=\Pi\setminus\overline{\omega}\,.
\end{equation}
Although we apply these solutions under the symmetry assumption
\eqref{(symm)}, we only use it  in Section~\ref{subsec34}; see Figure~\ref{fig-mirror}.

\subsection{The problems in $\Xi$ and their solvability}\label{subsec31}

According to \cite{SOME}, near the perforation string
 \begin{equation}\begin{array}{c}\displaystyle
\omega^\varepsilon(0,0),\dots,\omega^\varepsilon(0,N-1)\subset
\varpi^0=\{x:\,{\color{black}\vert x_1 \vert<1/2},x_2\in(0,H)\},
\end{array}\label{(32)}\end{equation}
cf., \eqref{(2)}, there appears a boundary layer which is described in the stretched coordinates
 \begin{equation}\begin{array}{c}
\xi=(\xi_1,\xi_2)=\varepsilon^{-1}{\color{black}(x_1,x_2-\varepsilon k H)}
\end{array}\label{(33)}\end{equation}
by means of a family of special solutions to the Laplace equation
 \begin{equation}\begin{array}{c}
-\Delta_\xi W(\xi)=0,\,\,\xi\in {\color{black}\Xi,}
\end{array}\label{(34)}\end{equation}
or the Poisson equation
\begin{equation}\begin{array}{c}\displaystyle
 -\Delta_\xi W(\xi)=F(\xi),\,\,\xi\in{\color{black}\Xi},
  \end{array}\label{(4300)}\end{equation}
with the periodicity conditions
 \begin{equation}\begin{array}{c}\displaystyle
W(\xi_1,H)=W(\xi_1,0),\,\,\frac{\partial W}{\partial \xi_2}(\xi_1,H)=
\frac{\partial W}{\partial \xi_2}(\xi_1,0),\,\,\xi_1\in{\mathbb R},
\end{array}\label{(35)}\end{equation}
and the Neumann condition on the boundary of the hole
$\overline{\omega}$ inside the strip \eqref{(1)}, either homogeneous
 \begin{equation}\begin{array}{c}
 \partial_{\nu(\xi)}W(\xi)=0,\,\,\xi\in\partial\omega,
 \end{array}\label{(36)}\end{equation}
or inhomogeneous
\begin{equation}\begin{array}{c}\displaystyle
 \partial_{\nu(\xi)}W(\xi)=G(\xi),\,\,\xi\in\partial\omega,
  \end{array}\label{(43N)}\end{equation}
with particular {\color{black}functions $F(\xi)$ and $G(\xi)$}. Here,
$\nu(\xi)=(\nu_1(\xi),\nu_2(\xi))$ is the outward (with respect to
$\Xi$) normal vector on $\partial\omega$ and, therefore, the inward one with
respect to $\omega$.

\begin{Remark}
The boundary condition \eqref{(36)} is
directly inherited from the original condition \eqref{(5)} on the
boundary of the perforation string \eqref{(32)}. For any
$\Lambda^\varepsilon\leq C$, we have
 $$
 \Delta_x+\Lambda^\varepsilon=\varepsilon^{-2}(\Delta_\xi+\varepsilon^2\Lambda^\varepsilon)
 \,\,\mbox{\rm with}\,\,\varepsilon^2\Lambda^\varepsilon \leq C\varepsilon^2
 $$
and the main asymptotic part $\Delta_\xi$ of the {\color{black}above} differential operator is involved with the Laplace
equation \eqref{(34)}. The periodicity conditions \eqref{(35)} have no relation to the
original quasi-periodicity conditions \eqref{(12)}, \eqref{(13)} but are needed to support
the standard asymptotic ansatz
$$
 U^\varepsilon\approx  w(x_2)W (\varepsilon^{-1}x),
$$
where $w$ is a smooth function in $x_2\in[0,H]$ and $W$  is {\color{black}a function} $H$-periodic in $\xi_2=\varepsilon^{-1}x_2$.
\end{Remark}

We proceed with the variational formulation
 \begin{equation}\begin{array}{c}\displaystyle
 (\nabla_\xi W,\nabla_\xi V)_\Xi=(F,V)_\Xi+(G,V)_{\partial \color{black} \omega}\quad
 \forall V\in {\cal H}^1_{per}(\Xi)
\end{array}\label{(44po)}\end{equation}
of the Poisson equation \eqref{(4300)} with the periodicity \eqref{(35)} and boundary \eqref{(43N)}
conditions. Here, ${\cal H}^1_{per}(\Xi)$ is the completion of the
linear space $C^\infty_{per}(\overline{\Xi})$ (infinitely
differentiable $H$-periodic in $\xi_2$ functions with compact
supports) in the norm
  \begin{equation}\begin{array}{c}\displaystyle
 \|W;{\cal H}^1_{per}(\Xi)\|=\Big(\|\nabla_\xi W;L^2(\Xi)\|^2+\|W;L^2(\Xi(2R))\|^2\Big)^{1/2},
\end{array}\label{(norm)}\end{equation}
where $R$ is a fixed positive constant such that
 \begin{equation}\begin{array}{c}\displaystyle
 \overline{\omega}\subset\Xi(R):=\{\xi\in\Xi:\,\vert \xi_1 \vert<R\}.
\end{array}\label{(ksi)}\end{equation}
Also, for convenience, we introduce here the  cut-off
functions, $\chi_\pm\in C^\infty({\mathbb R})$,
 \begin{equation}\begin{array}{c}\displaystyle
\chi_\pm(t)=1\,\,\mbox{\rm for}\,\,\pm t>2R\,\,\mbox{\rm and}\,\, \chi_\pm(t)=0\,\,\mbox{\rm for}\,\,\pm t<R,
\end{array}\label{(chi)}\end{equation}
with a fixed   $R>0$ satisfying \eqref{(ksi)}.

\begin{Proposition}\label{Proposition_Fredholm}
Let the functions on the right-hand sides of the problem
\eqref{(4300)}, \eqref{(35)}, \eqref{(43N)} meet the inclusions
  \begin{equation}\begin{array}{c}\displaystyle
(1+\xi^2_1)^{1/2}F\in L^2(\Xi),\quad G\in L^2(\partial\omega)
 \end{array}\label{(FG)}\end{equation}
and the orthogonality condition
 \begin{equation}\begin{array}{c}\displaystyle
\int\limits_\Xi F(\xi)d\xi+\int\limits_{\partial\omega}
G(\xi){\color{black}ds_\xi}=0.
\end{array}\label{(comp)}\end{equation}
Then the problem has a solution $W\in{\cal H}_{per}^1(\Xi)$ which is
defined up to an additive constant.
\end{Proposition}

\begin{proof}
We consider the perturbed equation
\begin{equation}\begin{array}{c}\displaystyle
 -\Delta_\xi W(\xi)+\mu X_R(\xi)W(\xi)=F(\xi),\,\,\xi\in{\color{black}\Xi,}
  \end{array}\label{(43)}\end{equation}
with the {\color{black}boundary conditions \eqref{(35)} and \eqref{(43N)};}
here, $\mu>0$ is a parameter and $X_R$ is the {\color{black} characteristic}
function of the truncated domain $\Xi(2R)$, i.e. $X_R(\xi)=1$ for
$\vert \xi_1 \vert<2R$ and $X_R(\xi)=0$ for $\vert \xi_1 \vert>2R$.
The variational formulation of the problem \eqref{(43)},
\eqref{(43N)}, \eqref{(35)} reads:
 \begin{equation}\begin{array}{c}\displaystyle
 (\nabla_\xi W,\nabla_\xi V)_\Xi+\mu(W,V)_{\Xi(2R)}=(F,V)_\Xi+(G,V)_{\partial{\color{black} \omega}}\quad
 \forall V\in {\cal H}^1_{per}(\Xi).
\end{array}\label{(44)}\end{equation}
In view of the one-dimensional Hardy inequality
$$
\int\limits_R^{+\infty}\Big\vert\frac{dZ}{d\xi_1}(\xi_1)\Big\vert^2d\xi_1\geq
\frac{1}{4}
\int\limits_R^{+\infty}\vert Z(\xi_1)\vert^2\,\frac{d\xi_1}{\xi_1^2}\quad
\forall Z\in C^\infty_c(R,+\infty),
$$
applied to the functions
$Z(\pm\xi_1,\xi_2)=\chi_\pm(\xi_1)W(\xi_1,\xi_2)$ {\color{black} with  $\chi_{\pm}$ defined by \eqref{(chi)},} and integrated in
$\xi_2\in(0,H)$, the norm \eqref{(norm)} is equivalent to the norm
 \begin{equation}\begin{array}{c}\displaystyle
 \Big(\|\nabla_\xi W;L^2(\Xi)\|^2+\|(1+\xi_1^2)^{-1/2}W;L^2(\Xi)\|^2\Big)^{1/2}.
\end{array}\label{(weight)}\end{equation}
Notice that the last Lebesgue norm in \eqref{(norm)} is computed
over a compact set while the weighted Lebesgue norm in
\eqref{(weight)} involves the whole infinite domain $\Xi$.

It is self evident that the left-hand side of the integral identity
\eqref{(44)} with $\mu>0$ can be taken as a scalar product in the
Hilbert space ${\cal H}^1_{per}(\Xi)$.  Hence,   according to the equivalency of the norms \eqref{(norm)} and
\eqref{(weight)},  and, owing to \eqref{(FG)},
the right-hand side of \eqref{(44)} defines a continuous functional
in ${\cal H}^1_{per}(\Xi)$. Thus, the Riesz representation theorem
assures that the problem \eqref{(44)} with $\mu>0$ has a unique
solution $W\in{\cal H}^1_{per}(\Xi)$ in the case \eqref{(FG)}.

According to the above mentioned equivalence of norms, considering the space ${\cal H}_{per}^1(\Xi)$ with the norm \eqref{(weight)},  and the fact that the  embedding ${\cal H}_{per}^1(\Xi) \subset
L^2(\Xi(2R))$ is compact, for any fixed $\mu$,  the spectral problem associated to \eqref{(44)}
\begin{equation}\begin{array}{c}\displaystyle
(\nabla_\xi W,\nabla_\xi V)_\Xi+\mu(W,V)_{\Xi(2R)}=\nu (W,V)_{\Xi(2R)}
\quad  \forall V\in {\cal H}^1_{per}(\Xi),
\end{array}\label{(44s)}\end{equation}
has a discrete spectrum with the corresponding  eigenfunctions being orthogonal both in $ {\cal H}_{per}^1(\Xi) $  and
$L^2(\Xi(2R))$. In addition,     $\nu=\mu$ is an eigenvalue  of \eqref{(44s)} with the  associated   eigenspace  of the constant functions $\mathfrak{C}$. Thus, considering the decomposition
${\cal H}^1_{per}(\Xi)= \mathfrak{C} \bigoplus  \mathfrak{C}^\bot$ with
$\mathfrak{C}^\bot$ the subspace formed by the elements of
${\cal H}^1_{per}(\Xi)$ which are orthogonal to the constants, by  the Fredholm alternative,
problem \eqref{(44po)} has a unique solution in  $\mathfrak{C}^\bot$ provided that the functional on the right hand side of \eqref{(44po)} is in the dual space  $(\mathfrak{C}^\bot)^*$, namely, provided that it satisfies the orthogonality condition \eqref{(comp)}.
This concludes the proof of the proposition.
\end{proof}

\subsection{Integral characteristics}\label{subsec32}

First of all, we recall that,
according to the general theory of elliptic problems in domains with
cylindrical outlets to infinity, see \cite[Ch. 5]{NaPl} and \cite[Section~3, 5]{na262}, the homogeneous problem \eqref{(34)}, \eqref{(35)}, \eqref{(36)} has
just two{\footnote{$2=\frac{1}{2}(2\times2)$ where the last $2$ is
the number of outlets to infinity in the domain $\Xi$ and the next
to the last $2$ is the number of linearly independent, $H$-periodic
in $\xi_2$ and polynomial in $\xi_1$, harmonics in the intact strip
$\Pi$, namely $1$ and $\xi_1$ in our case. This mnemonic rule works
for many other problems in domains with cylindrical and periodic
outlets to infinity, see the review paper \cite[Section~3, 5]{na262}.}}
{\color{black} linearly independent} solutions with the polynomial behavior at infinity. {\color{black} It is evident that the first solution is}  a constant, and we set
\begin{equation}\label{W0}
W^0(\xi)=1.
\end{equation}

Let us seek the second solution to the problem
\eqref{(34)}, \eqref{(35)}, \eqref{(36)} in the form
 \begin{equation}\begin{array}{c}
W^1(\xi)=\xi_1+W^1_0(\xi){\color{black}+C}
\end{array}\label{(39)}\end{equation}
where $C$ is a  certain constant, cf. \eqref{(W101)}, and $W^1_0\in{\cal H}^1_{per}(\Xi)$ satisfies the Laplace equation \eqref{(34)}, the
periodicity conditions \eqref{(35)} and the inhomogeneous Neumann
condition
 \begin{equation}\begin{array}{c}
 \partial_{\nu(\xi)} W^1_0(\xi)=-\partial_{\nu(\xi)}\xi_1=-\nu_1(\xi),\,\,\xi\in\partial\omega.
 \end{array}\label{(40)}\end{equation}

\begin{Proposition}\label{Proposition_W1}
 There is a unique   solution of problem
\eqref{(34)}, \eqref{(35)}, \eqref{(36)} with the decomposition
 \begin{equation}\begin{array}{c}\displaystyle
W^1(\xi)=\sum\limits_\pm\chi_\pm(\xi_1)(\xi_1\pm m_1(\Xi))+\widetilde{W}^1(\xi)
 \end{array}\label{(41)}\end{equation}
where  $\chi_\pm$  is defined by \eqref{(ksi)}--\eqref{(chi)},  $m_1(\Xi)$ is a
constant, and the remainder {\color{black} $\widetilde{W}^1(\xi)$ } and its
derivatives get the exponential decay $O(e^{-2\pi\vert \xi_1 \vert/H})$ as
$\xi_1\to\pm\infty$. The quantity $m_1(\Xi)$ in \eqref{(41)} is given by
 \begin{equation}\begin{array}{c}\displaystyle
 m_1(\Xi)=\frac{1}{2H}\left(\|\nabla_\xi W^1_0;L^2(\Xi)\|^2+\vert\omega\vert\right)>0,
 \end{array}\label{(42)}\end{equation}
where $\vert\omega\vert=\mbox{\rm mes}_2\,\omega$ and $W_0^1$ is a  solution of \eqref{(34)}, \eqref{(35)} and \eqref{(40)} in the space ${\cal H}^1_{per}(\Xi)$.
In addition, any solution of the problem \eqref{(34)}, \eqref{(35)}, \eqref{(36)} with polynomial growth at infinity is a linear combination
$c_0W^0+c_1W^1$ with some coefficients $c_0, c_1$.
\end{Proposition}

\begin{proof}
In the case $F=0$ and $G(\xi)=-\partial_\nu \xi_1$ the
equality \eqref{(comp)} is evidently fulfilled and, thus, the
problem \eqref{(34)}, \eqref{(35)}, \eqref{(40)} in its variational
form \eqref{(44po)} has a solution $W^1_0\in{\cal H}^1_{per}(\Xi)$ which is {\color{black} uniquely} defined
up to an additive constant. Since the boundary $\partial\omega$ is
smooth, this solution is infinitely differentiable in
$\overline{\Xi}$ and the Fourier method, in particular, gives the
decomposition
   \begin{equation}\begin{array}{c}
 W^1_0(\xi)=\sum\limits_\pm\chi_\pm(\xi_1)C_\pm+\widetilde{W}^1_0(\xi)
   \end{array}\label{(W100)}\end{equation}
with the exponentially decaying remainder $\color{black} \widetilde{W}^1_0$, and some constants
$ C_\pm$ {\color{black} which  can also depend on $R$, {\color{black}cf.~\eqref{(ksi)}}.} Setting
\begin{equation}\label{(W101)} C=-\frac{1}{2}(C_++C_-),\end{equation}
the function $W^1(\xi)=\xi_1+W^1_0(\xi)+C $
becomes the desired solution \eqref{(39)} of the problem
\eqref{(34)}, \eqref{(35)}, \eqref{(36)} admitting the representation \eqref{(41)}
with $m_1(\Xi)=\frac{1}{2}(C_+-C_-)$.

To prove the relation \eqref{(42)}, we apply  the Green formula twice and write
\begin{equation*}
\begin{split}
 &-\int\limits_{\partial\omega}(\xi_1+W^1_0(\xi))\partial_{\nu(\xi)}\xi_1ds_\xi=
 -\int\limits_{\partial\omega}\xi_1\partial_{\nu(\xi)}\xi_1ds_\xi+
 \int\limits_{\partial\omega} W^1_0(\xi)\partial_{\nu(\xi)} W^1_0(\xi)ds_\xi
\\
&\quad =\int\limits_\omega\vert\nabla_\xi\xi_1\vert^2d\xi+\|\nabla_\xi
W^1_0;L^2(\Xi)\|^2= \vert\omega\vert+\|\nabla_\xi W^1_0;L^2(\Xi)\|^2,
\end{split}
\end{equation*}
where we have used equalities \eqref{(40)} and \eqref{(44po)}. Similarly, we write
\begin{eqnarray*}
&&\displaystyle \int\limits_{\partial\omega}(\xi_1+W^1_0(\xi))\partial_{\nu(\xi)}\xi_1ds_\xi=
 \int\limits_{\partial\omega}\big((\xi_1+W^1_0(\xi))\partial_{\nu(\xi)}\xi_1
 -\xi_1\partial_{\nu(\xi)}(\xi_1+W^1_0(\xi))\big)ds_\xi
\\
 && \displaystyle \quad
 =\lim\limits_{T\to+\infty}\sum\limits_\pm \mp\int\limits_0^H
 \Big((\xi_1+W^1_0(\xi))\frac{\partial\xi_1}{\partial\xi_1}-\xi_1
 \frac{\partial }{\partial\xi_1}(\xi_1+W^1_0(\xi))\Big)\bigg\vert_{\xi_1=\pm T}d\xi_2
 \\
 && \displaystyle \quad
 =\lim\limits_{T\to+\infty}\sum\limits_\pm \mp\int\limits_0^H
 W^1_0(\pm T,\xi_2)d\xi_2=-2Hm_1(\Xi),
\end{eqnarray*}
and the relationship \eqref{(42)} follows immediately.
\end{proof}

\begin{Remark}\label{Remark_m1}
 The quantity \eqref{(42)} is an integral
characteristics of the Neumann hole $\overline{\omega}$ in the strip
$\Pi$ of width $H$ with the periodicity conditions at the lateral
sides. This characteristics looks quite similar to the classical
virtual mass tensor in the exterior Neumann problem, although {\color{black} it is a}
scalar, cf., \cite[Appendix G]{PoSe}. For any set
$\overline{\omega}$ of the positive area $\mbox{\rm mes}_2\,\omega$,
we have $m_1(\Xi)>0$. At the same time, in the case of a crack
$\Upsilon=\overline{\omega}$ along the $\xi_1$-axis we observe that
$\partial_{\nu(\xi)}\xi_1=0$ on $\partial\omega$, $W^1_0(\xi)=C^1_0$
and $\mbox{\rm mes}_2\,\Upsilon=0$, therefore,
$m_1(\Pi\setminus\Upsilon)=0$. However, the smoothness assumption on
the boundary in Section~\ref{subsec11} excludes cracks from our present
consideration.
\end{Remark}

\subsection{Other special solutions}\label{subsec33}

It proves necessary to introduce here two solutions of boundary value problems in $\Xi.$ First, let us introduce $W^2$ a solution of
the problem \eqref{(34)}, \eqref{(35)} and  the inhomogeneous Neumann condition \eqref{(40)} with the replacement $1\mapsto2$, namely
 \begin{equation}\begin{array}{c}
 \partial_{\nu(\xi)} W^2(\xi)=-\partial_{\nu(\xi)}\xi_2=-\nu_2(\xi),\,\,\xi\in\partial\omega.
 \end{array}\label{(409)}\end{equation}
The compatibility condition \eqref{(comp)} is again fulfilled so
that the problem \eqref{(34)}, \eqref{(35)}, \eqref{(409)} has a
bounded solution which is {\color{black} uniquely} defined up to an additive constant and,
therefore, is fixed uniquely in the form
 \begin{equation}\begin{array}{c}\displaystyle
W^2(\xi)=\sum\limits_\pm\pm\chi_\pm(\xi_1)m_2(\Xi)+\widetilde{W}^2(\xi)
 \end{array}\label{(419)}\end{equation}
where $m_2(\Xi)$ is  a constant, and the remainder $\widetilde{W}^2(\xi)$ and its
derivatives get the exponential decay as $\xi_1\to\pm\infty$.

In contrast to the quantity \eqref{(42)} the coefficient $m_2(\Xi)$
in \eqref{(419)} can get arbitrary sign (see Section~\ref{subsec34}). Notice
that the following integral representation is valid,  {\color{black} cf.  \eqref{(39)} and \eqref{(40)}:}
\begin{multline}\label{(4199)}
\int\limits_{\partial\omega}W^1(\xi)\partial_{\nu(\xi)}\xi_2\,ds_\xi=
\int\limits_{\partial\omega}\big(W^2(\xi)\partial_{\nu(\xi)}W^1(\xi)
-W^1(\xi)\partial_{\nu(\xi)}W^2(\xi)\big)\,ds_\xi\\
=\lim\limits_{T\to+\infty}\sum\limits_{\mp} \! \mp\int\limits_0^H \!\Big(
W^2(\pm T,\xi_2)\frac{\partial W^1}{\partial\xi_1}(\pm T,\xi_2)
-W^1(\pm T,\xi_2)\frac{\partial W^2}{\partial\xi_1}(\pm
T,\xi_2)\Big)\,d\xi_2=- 2Hm_2(\Xi).
\end{multline}

Finally, we introduce a solution $W^3$ to the Poisson equation
 \begin{equation}\begin{array}{c}\displaystyle
-\Delta_\xi W^3(\xi)=1,\,\,\xi\in\Xi,
 \end{array}\label{(S1)}\end{equation}
with the boundary conditions \eqref{(35)}, \eqref{(36)} which can be found in
the form
 \begin{equation}\begin{array}{c}\displaystyle
W^3(\xi)=-\frac{1}{2}\xi^2_1+\sum\limits_\pm\pm\chi_\pm
(\xi_1)\bigg(\frac{\vert\omega\vert}{2H}\xi_1+ m_3(\Xi)\bigg)+ \textcolor{black}{\widetilde{W}^{\,3}}(\xi)
 \end{array}\label{(SS2)}\end{equation}
where $m_3(\Xi)$ is a constant, and the remainder $\widetilde{W}^3(\xi)$ and its
derivatives get the exponential decay as $\xi_1\to\pm\infty$.
To show this, we accept the representation $W^3(\xi)= W^3_0(\xi)-\xi^2_1/2$ and
observe that $ W^3_0$ is a solution of the problem \eqref{(34)},
\eqref{(35)} with the Neumann condition
$$
\partial_{\nu(\xi)}
 W^3_0(\xi)=\nu_1(\xi)\xi_1,\,\,\xi\in\partial\omega.
$$
Thus, the argument in the proof of Proposition~\ref{Proposition_W1} to get $W_0^1$ and  $W^1$ (cf. \eqref{(39)}) gives us a
solution with the linear growth as $\xi_1\rightarrow\pm\infty$,
and we can provide the
decomposition
$$
 W^3_0(\xi)=\sum\limits_\pm\pm {\color{black}\chi_\pm(\xi_1)} \Big(C^3_1\xi_1+C^3_0({\color{black} \omega })\Big)+
\textcolor{black}{\widetilde{W}^{\,3}}(\xi),
$$
for certain coefficients $C_1^3$ and $C^3_0(\omega)$.

To compute the coefficient $C_1^3$, we apply the Green formula twice  as
follows:
{
$$
-\vert\omega\vert =\int\limits_{\partial\omega}\frac{1}{2}
\frac{\partial}{\partial\nu(\xi)}\xi_1^2ds_\xi
 =\int\limits_{\partial\omega}\partial_{\nu(\xi)} W^3_0(\xi)ds_\xi =
\lim\limits_{T\rightarrow+\infty}\sum\limits_\pm\mp\int\limits_0^H
\frac{\partial}{\partial\xi_1} W^3_0(\pm T,\xi_2)d\xi_2=-2HC^3_1.
$$
In contrast to $C^3_1$, the coefficient $C^3_0$ depends on the shape of
$\omega$ but we will not use it in the sequel, and we avoid introducing here its computation.

\subsection{The symmetry assumption and its consequences}\label{subsec34}

As pointed out in Section~\ref{subsec11} we can describe the band-gap structure of the low-frequency range of
the spectrum \eqref{(8)} only in the case of the mirror symmetry of the  hole. Therefore, we
will justify the derived asymptotics under the supposition
\eqref{(symm)}, cf. Section~\ref{sec4}.
First of all,  we realize that
\begin{equation}\begin{array}{c}\displaystyle
m_2(\Xi)=0,
\end{array}\label{(m20)}\end{equation}
so that all asymptotic expansions will simplify.
This is a consequence of  the fact that the boundary layer terms have the following important
properties.

\begin{Lemma}\label{Lemma_W}
Under the assumption \eqref{(symm)}, the
functions $W^1$, $W^3$ and $W^2$, respectively, are even and odd in
the variable $\xi_2-H/2$ and, hence,
\begin{equation}\begin{array}{c}\displaystyle
\frac{\partial W^j}{\partial \xi_2}(\xi_1,0)=\frac{\partial
W^j}{\partial \xi_2}(\xi_1,H)=0, {\color{black}\,\,\, \xi_1\in\mathbb{R},}\,\,j=1,3,\vspace{0.2cm}\\ W^2(\xi_1,0)=
W^2(\xi_1,H)=0{\color{black}, \,\,\, \xi_1\in\mathbb{R}}.
\end{array}\label{(J8)}\end{equation}
\end{Lemma}

The results in Lemma~\ref{Lemma_W} are a consequence of the definition of functions $W^i$, $i=1,2,3$,  and the uniqueness of the solutions of the problems that they satisfy in the way stated throughout the section. Equation \eqref{(m20)} follows from the evenness of $W^1$ and formula \eqref{(4199)}.

\section{Formal asymptotic analysis of simple eigenvalues}\label{sec4}

In this section, by means of matched asymptotic expansions, we construct a corrector improving the first approximation \eqref{(27)}. In particular, we provide the complete analysis of the first correction term  of the eigenpairs of \eqref{(11)}--\eqref{(14)}  in the case where the limit    eigenvalue is simple
(see Remark~\ref{remarknueva} for multiple eigenvalues).
The asymptotic structures   here constructed  will  give us a reason to introduce the symmetry assumption \eqref{(symm)}, see Section~\ref{subsec45} and Remark~\ref{Remark_N}.

\subsection{Asymptotic ans\"{a}tze}\label{subsec41}

Let us fix the Floquet parameter $\eta\in[-\pi,\pi]$ such that the eigenvalue $\Lambda^0_p(\eta)$
of the problem \eqref{(29)}--\eqref{(25)} is simple. In other words,
only one dispersion curve crosses the point
$(\eta,\Lambda^0_p(\eta))$.
Let us fix $U^0_p(\cdot;\eta)$ a corresponding eigenfunction (see \eqref{(W12)}).
We employ the method of matched asymptotic expansions, see e.g. \cite{NgSP,MaNaPl}, in the interpretation \cite{na457,na489},
to obtain  corrector terms for $\Lambda^0_p(\eta)$ and $U^0_p(\cdot;\eta)$.

Let us accept the simplest asymptotic ans\"{a}tze
\begin{equation}\begin{array}{c}\displaystyle
\Lambda^\varepsilon_p(\eta)=\Lambda^0_p(\eta)+\varepsilon\Lambda^\prime_p(\eta)+
\varepsilon^2\Lambda^{\prime\prime}_p(\eta)+\dots,
\end{array}\label{(W0)}\end{equation}
\begin{equation}\begin{array}{c}\displaystyle
U^\varepsilon_p(x;\eta)=U^0_p(x;\eta)+\varepsilon U^\prime_p(x;\eta)
+\varepsilon^2 U^{\prime\prime}_p(x;\eta)+\dots, \vspace{0.2cm}\\
\end{array}\label{(W1)}\end{equation}
We regard \eqref{(W1)} as the outer expansion, which fits in $\varpi^0\setminus \varsigma $  at a distance from  the   vertical mid-line $\varsigma=\{x:\, x_1=0,x_2\in(0,H)\}$.  We have excluded the line segment $\varsigma$ in the equation
\eqref{(W1)} because of the perforation string \eqref{(32)} which provokes the boundary layer phenomenon.
Here, and in what follows, dots stand for higher-order  terms which are inessential in our formal asymptotic analysis.

Note that, although  we will not determine the second order terms $\Lambda^{\prime\prime}_p(\eta)$ and
$U^{\prime\prime}_p(x;\eta)$,   they are involved with the
asymptotic procedure. Also,  we emphasize that the main term $U^0_p$ in
\eqref{(W1)} is a smooth function in $\varpi^0$ but the correction
terms may present jumps through   $\varsigma$.

\medskip

Inserting these
ans\"{a}tze into the equations \eqref{(11)}--\eqref{(14)} and
extracting terms of order $\varepsilon$ readily yield the following
restrictions for {\color{black}  the first order
terms $\Lambda^{ \prime}_p(\eta)$ and
$U^{ \prime}_p(x;\eta)$}:
the differential equation
\begin{equation}\begin{array}{c}
-\Delta_x
U_p^\prime(x;\eta)-\Lambda_p^0(\eta)U^\prime_p(x;\eta)=
\Lambda_p^\prime(\eta)U_p^0(x;\eta), \,\,
x\in\varpi^0\setminus\varsigma,
\end{array}\label{(W2)}\end{equation}
 the quasi-periodicity
conditions \eqref{(25)} at the vertical sides,   the Neumann
conditions on the punctured horizontal sides
\begin{equation}\begin{array}{c}\displaystyle
\frac{\partial U^\prime_p}{\partial
x_2}(x_1,0;\eta)=\frac{\partial U^\prime_p}{\partial
x_2}(x_1,H;\eta)=0,\,\,
x_1\in\Big(-\frac{1}{2},0\Big)\cup\Big(0,\frac{1}{2}\Big),
\end{array}\label{(W3)}\end{equation}
and some transmission conditions on $\varsigma$  that we determine by the matching procedure (cf. \eqref{(W11)} and \eqref{(W11N)}).
This is the aim of Section~\ref{subsec42} and \ref{subsec43} below, while $\Lambda_p^\prime(\eta)$ is determined in Section~\ref{subsec44}.

In order to do this, we introduce the inner expansion
\begin{equation}\begin{array}{c}\displaystyle
U^\varepsilon_p(x;\eta)=w^0_p(x_2;\eta)+\varepsilon w^\prime_p(\xi,x_2;\eta)+
\varepsilon^2 w^{\prime\prime}_p(\xi,x_2;\eta)+\dots,
\end{array}\label{(W4)}\end{equation}
 where we have assumed that the main term $w^0_p $ is constant in $\xi$, cf. \eqref{(33)}, while the functions arising in further terms, $ w^\prime_p $ and $w^{\prime\prime}_p$, depend on $\xi$ and satisfy a  periodicity condition in the $\xi_2$-direction, namely, conditions  \eqref{(35)}.

\subsection{The first transmission condition}\label{subsec42}

The Taylor formula implies
\begin{multline}
U^0_p(x;\eta)+\varepsilon U^\prime_p(x;\eta)+\varepsilon^2U^{\prime\prime}_p(x;\eta)=
U^0_p(0,x_2;\eta)+\varepsilon\Big(U^\prime_p(\pm0,x_2;\eta)+\xi_1\frac{\partial
U^0_p}{\partial x_1} (0,x_2;\eta)\Big)\\
+\varepsilon^2\Big(U^{\prime\prime}_p(\pm0,x_2;\eta) +\xi_1\frac{\partial U^\prime_p}{\partial x_1}
(\pm 0,x_2;\eta)+ \frac{\xi_1^2}{2}\frac{\partial^2 U^0_p}{\partial x^2_1}
(0,x_2;\eta)\Big)+\dots.\label{(W5)}
\end{multline}
Hence, comparing terms of order $1$ in \eqref{(W4)}
and \eqref{(W5)} leads us to the formula
\begin{equation}\begin{array}{c}\displaystyle
w^0_p(x_2;\eta)=U^0_p(0,x_2;\eta).
\end{array}\label{(W6)}\end{equation}

In addition, taking derivatives with respect to $\xi$ in  equations \eqref{(11)} and \eqref{(14)},  inserting \eqref{(W4)} in \eqref{(11)} and \eqref{(14)}, and extracting the terms of order $\varepsilon$, we obtain that  the first order term in the inner expansion \eqref{(W4)}
satisfies the  equation \eqref{(34)},  with periodicity conditions  \eqref{(35)} and  the
inhomogeneous Neumann condition
\begin{equation}\begin{array}{c}\displaystyle
\partial_{\nu(\xi)}w^\prime_p(\xi,x_2;\eta)=-\frac{\partial U^0_p}{\partial x_2}(0,x_2;\eta)
\partial_{\nu(\xi)}\xi_2,\,\,\xi\in\partial\omega,
\end{array}\label{(W7)}\end{equation}
which   takes into account  the
discrepancy in \eqref{(36)} of the main term \eqref{(W6)} due to its
dependence on the slow variable $x_2$. Indeed, we have used the formula for the directional derivative:
$$
\frac{\partial V}{\partial
\nu(x)}(\xi,x_2)=\frac{1}{\varepsilon}\frac{\partial V}{\partial
\nu(\xi)}(\xi,x_2)+\nu_2(\xi)\frac{\partial V}{\partial
x_2}(\xi,x_2).
$$
Furthermore, the matching of the outer and inner expansions at the first order  prescribes the following behavior at infinity for $w^\prime_p$
\begin{equation}\begin{array}{c}\displaystyle
w^\prime_p(\xi,x_2;\eta) \sim \xi_1 \frac{\partial U^0_p}{\partial
x_1}(0,x_2;\eta)+ U^\prime_p(\pm0,x_2;\eta) \quad \,\, \mbox{ as } \quad
\xi_1\to\pm\infty,
\end{array}\label{(W8)}\end{equation}
cf. \eqref{(W6)} and the factor of $\varepsilon$ on the right-hand side of
\eqref{(W5)}.

The solution of the problem \eqref{(34)}, \eqref{(35)},
\eqref{(W7)}, \eqref{(W8)} is nothing but {\color{black} a}  linear combination
\begin{equation}\label{(W9)}
w^\prime_p(\xi,x_2;\eta)=\frac{\partial U^0_p}{\partial
x_1}(0,x_2;\eta)W^1(\xi)+ \frac{\partial U^0_p}{\partial
x_2}(0,x_2;\eta)W^2(\xi)+C^\prime_p(x_2;\eta)W^0
\end{equation}
of the   solutions of the problems on $\Xi$ introduced   in Section~\ref{subsec32} and \ref{subsec33}, see \eqref{def_Xi}, where the factor   $C^\prime_p(x_2;\eta)$ is related to \eqref{(W8)}, not determined yet. However, it does not influence our further analysis.

\medskip

Using the decompositions \textcolor{black}{\eqref{(41)}}, \eqref{(419)} and recalling \eqref{W0}, we find the following expressions in \eqref{(W8)}:
\begin{equation}\label{(W10)}
U^\prime_p(\pm0,x_2;\eta)=\pm \Big(\frac{\partial U^0_p}{\partial
x_1}(0,x_2;\eta)m_1(\Xi)+ \frac{\partial U^0_p}{\partial
x_2}(0,x_2;\eta)m_2(\Xi)\Big)\!+\!C^\prime_p(x_2;\eta).
\end{equation}
Although the traces \eqref{(W10)} are not yet fixed, we compute the
jump of $U^\prime_p$ through
$\varsigma$,
$$[U^\prime_p]_0(x_2; \eta)=U^\prime_p(+0,x_2;\eta)-U^\prime_p(-0,x_2;\eta),$$
that is,
\begin{equation}\begin{array}{c}\displaystyle
[U^\prime_p]_0(x_2;\eta)=2\frac{\partial U^0_p}{\partial
x_1}(0,x_2;\eta)m_1(\Xi)+ 2\frac{\partial U^0_p}{\partial
x_2}(0,x_2;\eta)m_2(\Xi),\quad  x_2\in(0,H).
\end{array}\label{(W11)}\end{equation}

\subsection{The second transmission condition}\label{subsec43}

To proceed, we have to
deal with the third term of the inner expansion \eqref{(W4)} which
after inserting into the problem \eqref{(11)}--\eqref{(14)} and
extracting terms of order $\varepsilon^2$ leads to the problem
\begin{equation}\begin{array}{c}\displaystyle
-\Delta_\xi
w_p^{\prime\prime}(\xi,x_2;\eta)=f_p^{\prime\prime}(\xi,x_2;\eta),\,\,\xi\in\Xi,
\\\\\displaystyle
\partial_{\nu(\xi)}w_p^{\prime\prime}(\xi,x_2;\eta)=g_p^{\prime\prime}(\xi,x_2;\eta)
,\,\,\xi\in\partial\omega,
\end{array}\label{(SSS)}\end{equation}
with the periodicity conditions \eqref{(35)}. According to \eqref{(W6)}, \eqref{(29)} and \eqref{(W9)}, the right-hand sides of \eqref{(SSS)}
are given by
\begin{equation}\label{(S12)}
\begin{split}
f_p^{\prime\prime}(\xi,x_2;\eta)&=\Delta_xw_p^0(\xi,x_2;\eta)+\Lambda^0_p (\eta)
w_p^0(\xi,x_2;\eta)+2\frac{\partial^2w_p^\prime}{\partial
x_2\partial\xi_2}(\xi,x_2;\eta)\\
&=-\frac{\partial^2 U^0_p}{\partial x^2_1}(0,x_2;\eta)+2
\frac{\partial^2U^0_p}{\partial x_1\partial
x_2}(0,x_2;\eta)\frac{\partial W^1}{\partial\xi_2}(\xi) +2
\frac{\partial^2U^0_p}{\partial x_2^2}(0,x_2;\eta)\frac{\partial
W^2}{\partial\xi_2}(\xi),
\end{split}
\end{equation}
\begin{equation*}
\begin{split}
g_p^{\prime\prime}(\xi,x_2;\eta)&=
-\nu_2(\xi)\frac{\partial w_p^\prime}{\partial x_2}(\xi,x_2;\eta)\\
&=-\nu_2(\xi)\frac{\partial^2 U^0_p}{\partial x_1\partial x_2}(0,x_2;\eta)W^1(\xi)
-\nu_2(\xi)\frac{\partial^2 U^0_p}{\partial x^2_2}(0,x_2;\eta)W^2(\xi)
-\nu_2(\xi)\frac{\partial
C^\prime_p}{\partial x_2}(x_2;\eta).
\end{split}\end{equation*}
Furthermore, the matching procedure and the Taylor formula
\eqref{(W5)}, up to the order $\varepsilon^2$,  establish the following behavior at infinity:
\begin{equation}\label{(S14)}
\displaystyle
w_p^{\prime\prime}(\xi,x_2;\!\eta)  \sim
\frac{\xi_1^2}{2}\frac{\partial^2 U^0_p}{\partial x^2_1}
(0,x_2;\!\eta) +\xi_1\frac{\partial
U^\prime_p}{\partial x_1}
(\pm 0,x_2;\!\eta)+
U^{\prime\prime}_p(\pm0,x_2;\!\eta)  \quad \hbox{ as }   \quad
\xi_1\rightarrow\pm\infty.
\end{equation}

We observe that, owing to \eqref{(41)} and \eqref{(419)}, the
derivatives $\partial W^q/\partial\xi_2$ decay exponentially at
infinity while the first term on the right-hand side \eqref{(S12)}
is constant in $\xi$. Thus, a solution of the problem \eqref{(SSS)},
\eqref{(35)}\textcolor{black}{,} \eqref{(S14)} admits the quadratic growth as
$\xi_1\rightarrow\pm\infty$, and we set
\begin{equation}\begin{array}{c}\displaystyle
w_p^{\prime\prime}(\xi,x_2;\eta)=-\frac{\partial^2 U^0_p}{\partial
x^2_1}
(0,x_2;\eta)W^3(\xi)+\widehat{w}^{\,\prime\prime}_p(\xi,x_2;\eta),
\end{array}\label{(S15)}\end{equation}
where $W^3$ is given by \eqref{(SS2)}. The remaining part
$\widehat{w}^{\,\prime\prime}_p$ verifies the problem
\begin{equation}\begin{array}{c}\displaystyle
-\Delta_\xi \widehat{w}^{\,\prime\prime}_p(\xi,x_2;\eta)=
\widehat{f}^{\,\prime\prime}_p(\xi,x_2;\eta),\,\,\xi\in\Xi,
\\\\\displaystyle
\partial_{\nu(\xi)}\widehat{w}^{\,\prime\prime}_p(\xi,x_2;\eta)=g_p^{\prime\prime}(\xi,x_2;\eta)
,\,\,\xi\in\partial\omega,
\end{array}\label{(SSpr)}\end{equation}
with the periodicity condition \eqref{(35)}, where
$$\widehat{f}^{\,\prime\prime}_p(\xi,x_2;\eta)=
f^{\prime\prime}_p(\xi,x_2;\eta)+\frac{\partial^2 U^0_p}{\partial x^2_1} (0,x_2;\eta)\in L^2(\Xi),$$
and gets an exponential decay at infinity.
A solution of such a problem  exists in the form
\begin{equation}\label{(S16)}
\!\!\widehat{w}^{\,\prime\prime}_p(\xi,x_2;\eta)=\sum\limits_\pm
\pm\chi_\pm(\xi_1)(\widehat{C}_1(x_2;\eta)\xi_1+\widehat{C}_0(x_2;\eta))+
\widetilde{w}^{\,\prime\prime}_p(\xi,x_2;\eta)
\end{equation}
for certain coefficients  $\widehat{C}_0$ and $\widehat{C}_1$, and a remainder $\widetilde{w}^{\,\prime\prime}_p$  which gets the
exponential decay {\color{black}as $\xi_1\to\pm\infty.$} To derive the second transmission condition for $U^\prime_p$ arising in \eqref{(W2)}, it suffices to compute the coefficient
$\widehat{C}_1$ because the other coefficient $\widehat{C}_0$  proves to be
of no further use.

Indeed,  by applying the Green formula in    \eqref{(SSpr)}, we  readily obtain
\begin{equation}\label{(SSC)}
\int\limits_\Xi\widehat{f}^{\,\prime\prime}_p(\xi,x_2;\eta)d\xi+
\int\limits_{\partial\omega}g^{\prime\prime}_p(\xi,x_2;\eta)ds_\xi
=
-\lim\limits_{T\rightarrow+\infty}\sum\limits_\pm\pm
\int\limits_0^H\frac{\partial\widehat{w}^{\,\prime\prime}_p}{\partial
\xi_1}(\pm T,\xi_2,x_2;\eta)d\xi_2 =\textcolor{black}{-} 2H\widehat{C}_1(x_2;\eta).
\end{equation}
Let us to  process the left-hand side. First, we take $V=\xi_2$ and $W=W^q$ with
$q=1,2$ in formula \eqref{(44po)}, and we get
$$
\int\limits_\Xi\frac{\partial W^q}{\partial \xi_2}(
\xi,x_2)d\xi-\int\limits_{\partial\omega}\nu_2(\xi) W^q(
\xi,x_2)ds_\xi=0,\quad q=1,2.
$$
Using these formulas in the  definitions of $\widehat{f}^{\,\prime\prime}_p$ and $g^{\prime\prime}_p$, we have
\begin{equation*}
\begin{split}
\int\limits_\Xi\widehat{f}^{\,\prime\prime}_p(\xi,x_2;\eta)d\xi+
\int\limits_{\partial\omega}g^{\prime\prime}_p(\xi,x_2;\eta)ds_\xi
= &\frac{\partial^2 U^0_p}{\partial x_1\partial x_2}(0,x_2;\eta) \!
\int\limits_{\partial\omega}\nu_2(\xi)W^1(\xi) ds_\xi\\
&+\frac{\partial^2 U^0_p}{\partial x^2_2}(0,x_2;\eta)\!\int\limits_{\partial\omega}\nu_2(\xi)W^2(\xi)ds_\xi-\frac{\partial C^\prime_p}{\partial x_2}(x_2;\eta)\!\int\limits_{\partial\omega} \nu_2(\xi)ds_\xi.
\end{split}\end{equation*}
Now, let us note that by the Green formula it follows
$$
\int\limits_{\partial\omega} \nu_2(\xi)ds_\xi=\int\limits_{\partial\omega} \partial_{\nu(\xi)}\xi_2 ds_\xi=0,
$$
{\color{black}which cancels the term containing the derivative of $C^\prime_p(x_2;\eta).$}
Besides, from \eqref{(409)} and \eqref{(44po)}, we obtain
\begin{equation}\label{(Mxi)}
-\int\limits_{\partial\omega}\nu_2(\xi)W^2(\xi){\color{black}ds_\xi}=
\int\limits_{\partial\omega}W^2(\xi)\partial_{\nu(\xi)}W^2(\xi){\color{black}ds_\xi} =\|\nabla_\xi W^2;L^2(\Xi)\|^2=:M(\Xi)>0.
\end{equation}
Finally, considering  \eqref{(4199)} and using \eqref{(SSC)}--\eqref{(Mxi)}, we get
\begin{equation}\begin{array}{c}\displaystyle
-2H\widehat{C}_1(x_2;\eta)= \textcolor{black}{2H}\frac{\partial^2 U^0_p}{\partial
x_1\partial x_2}(0,x_2;\eta)m_2(\Xi)-\frac{\partial^2
U^0_p}{\partial x^2_2}(0,x_2;\eta)M(\Xi).
\end{array}\label{(S17)}\end{equation}
Gathering \eqref{(S14)}, \eqref{(S15)}, \eqref{(S16)}, \eqref{(S17)}
and \eqref{(SS2)} we conclude that
\begin{equation}\label{(S18)}
\frac{\partial U^\prime_p}{\partial x_1}
(\pm 0,x_2;\eta)=\textcolor{black}{\mp} m_2(\Xi)\frac{\partial^2 U^0_p}{\partial
x_1\partial x_2}(0,x_2;\eta) \mp \frac{|\omega|}{2H}\frac{\partial^2 U^0_p}{\partial
x^2_1}(0,x_2;\eta) \textcolor{black}{\pm\frac{M(\Xi)}{2H}}\frac{\partial^2 U^0_p}{\partial
x^2_2}(0,x_2;\eta).
\end{equation}
Thus, we obtain the jump through $\varsigma$ for the normal derivative of  $U^\prime_p$
\begin{equation}\label{(W11N)}
\Big[\frac{\partial U^\prime_p}{\partial
x_1}\Big]_0(x_2;\eta)=\textcolor{black}{-}2m_2(\Xi)\frac{\partial^2 U^0_p}{\partial
x_1\partial x_2}(0,x_2;\eta)-\frac{|\omega|}{H}\frac{\partial^2
U^0_p}{\partial x^2_1}(0,x_2;\eta)\textcolor{black}{+\frac{M(\Xi)}{H}} \frac{\partial^2 U^0_p}{\partial x^2_2}(0,x_2;\eta).
\end{equation}
This completes  the problem for correction terms $\Lambda^\prime_p(\eta)$
and $U^\prime_p(\cdot;\eta)$ in the ans\"{a}tze \eqref{(W0)} and
\eqref{(W1)}. Namely, they are the unknowns of the problem \eqref{(W2)},
\eqref{(W3)}, \eqref{(25)}, \eqref{(W11)} and \eqref{(W11N)}. The existence and uniqueness of both terms is provided below.

\subsection{Computing the correction term in the eigenvalue asymptotics}\label{subsec44}

Since, by our assumption, the eigenvalue
$\Lambda_p^0(\eta)$ is simple, the  solution of   problem      \eqref{(W2)},
\eqref{(W3)}, \eqref{(25)}, \eqref{(W11)}, {\color{black}\eqref{(W11N)}} has  only
one  compatibility condition. Indeed,  it must satisfy  the orthogonality condition, in the sense of
the Green formula, of the  right-hand side of \eqref{(W2)} to the eigenfunction
\begin{equation}\begin{array}{c}\displaystyle
U^0_p(x;\eta):=U^0_{jk}(x;\eta)=e^{i(\eta+2\pi
j)x_1}\cos\Big(\pi k\frac{x_2}{H}\Big),
\end{array}\label{(W12)}\end{equation}
see \eqref{(31)}.  This determines completely $ \Lambda_p'(\eta)$ as we show below (cf. \eqref{(W14)}).

First, we observe that,  by \eqref{(W12)},
$$
\|U^0_p;L^2(\varpi^0)\|^2=\|U^0_p;L^2(\varsigma)\|^2=\frac{1}{2}(1+\delta_{k,0})H,
$$
where $\delta_{k,l}$ denotes the Kronecker symbol.
Then, we multiply \eqref{(W2)} by  the conjugate of $U^0_p$ and integrate over $ \varpi^0 \setminus \varsigma $
to  get
$$
\frac{1}{2}(1+\delta_{k,0})H\Lambda^\prime_p(\eta)
=-\int\limits_{\varpi^0}
\left(\Delta_x U^\prime_p(x;\eta)+\Lambda^0_p{\color{black}(\eta)} U^\prime_p(x;\eta)\right)\overline{U^0_p(x;\eta)}dx.
$$
Because of \eqref{(W3)}, \eqref{(25)} and  \eqref{(W12)}, the Green formula yields
\begin{equation*}
\begin{split}
\frac{1}{2}(1+\delta_{k,0})H\Lambda^\prime_p(\eta) &=\int_0^H \Big(\overline{U^0_p(x;\eta)} \frac{\partial
U_p^\prime}{\partial x_1}(x;\eta)-U^\prime_p(x;\eta)\overline{
\frac{\partial U_p^0}{\partial
x_1}(x;\eta)}\Big)\bigg\vert_{x_1=-0}^{+0} \!\!dx_2\\
&=\int_0^H \Big(\overline{U^0_p(0,x_2;\eta)}
\Big[\frac{\partial U_p^\prime}{\partial
x_1}\Big]_0(x_2;\eta)-[U^\prime_p]_0(x_2;\eta)\overline{
\frac{\partial U_p^0}{\partial x_1}(0,x_2;\eta)}\Big)dx_2.
\end{split}\end{equation*}
Now, taking into account the jump
conditions  \eqref{(W11)} and \eqref{(W11N)}, and using {\color{black}\eqref{(W12)},}
$$
\int_0^H\cos^2\Big(\pi
k\frac{x_2}{H}\Big)dx_2=\!\frac{H}{2}(1+\delta_{k,0}){\color{black}\quad \mbox{and} \quad }\int_0^H\cos\Big(\pi
k\frac{x_2}{H}\Big)\sin\Big(\pi k\frac{x_2}{H}\Big)dx_2=0,
$$
we have
$$
\frac{1}{2}(1+\delta_{k,0})H\Lambda^\prime_p(\eta)
=\!\frac{H}{2}(1+\delta_{k,0})\bigg(\!\frac{\vert\omega\vert}{H}(\eta+2\pi j)^2\textcolor{black}{-\frac{M(\Xi)}{H}}\Big(\!\frac{\pi
k }{H}\!\Big)^{\!2}\!-2(\eta+2\pi j)^2m_1(\Xi)\!\bigg).
$$
As a result, we obtain the relationship
\begin{equation}\label{(W14)}
\Lambda^\prime_p(\eta) =\textcolor{black}{-}
2\bigg(\frac{\pi^2k^2}{H^2}\textcolor{black}{\frac{M(\Xi)}{2H}}+(\eta+2\pi
j)^2\Big(m_1(\Xi)-\frac{\vert\omega\vert}{2H}\Big)\bigg).
\end{equation}

Notice that, according to \eqref{(Mxi)} and \eqref{(42)}, the right-hand side of \eqref{(W14)} is negative.
Also, the process determines uniquely  the terms $\Lambda_p^{\prime}$ and  $U_p^{\prime}$ in the asymptotic series \eqref{(W0)} and \eqref{(W1)}.

\begin{Remark} \label{remarknueva}
{\rm
Assuming that $(\eta, \Lambda_p^0(\eta))$ is a crossing point of two dispersion curves does not affect the formal computations in Sections~\ref{subsec41}--\ref{subsec43}, $U_p^0(\cdot,\eta)$  being any of the corresponding  eigenfunctions in \eqref{(W12)} with $k=0$.
Also, in Section~\ref{subsec44},  when determining the second term of the asymptotic expansions  $\Lambda_p^\prime(\eta)$ and $U_p^\prime(\cdot,\eta)$, there is no contradiction  since the corresponding eigenfunctions only depend on $x_1$, namely, rewriting computations we obtain
$$
\Lambda^\prime_p(\eta) =-2(\eta+2\pi j)^2\Big(m_1(\Xi)-\frac{\vert\omega\vert}{2H}\Big)
$$
while there are two associated solutions    $U_p^\prime(\cdot,\eta)$ one for each eigenfunction of $\Lambda_p^0(\eta)$.  }
\end{Remark}

\subsection{On the symmetry assumption}\label{subsec45}
The first term \eqref{(W6)}
of the inner expansion \eqref{(W4)} meets the Neumann condition
\eqref{(14)} at the sides {\color{black}$x_2=0$ and $x_2=H$} of the periodicity cell
$\varpi^\varepsilon$ because $U^0_p$ does. Let us examine the second
term \eqref{(W9)} which satisfies
\begin{equation}
\begin{split}
\!\!\!\frac{\partial w^\prime_p}{\partial
x_2}\Big(\frac{x}{\varepsilon},x_2;\eta\Big)\bigg|_{x_2=0}=&\frac{1}{\varepsilon}\bigg(
\frac{\partial W^1}{\partial
\xi_2}\Big(\frac{x_1}{\varepsilon},0\Big)\frac{\partial
U^0_p}{\partial x_1} ({\color{black}0},0;\eta )+\frac{\partial W^2}{\partial
\xi_2}\Big(\frac{x_1}{\varepsilon},0\Big)\frac{\partial
U^0_p}{\partial x_2} ({\color{black}0},0;\eta )\bigg)\\
&+\frac{\partial^2 U^0_p}{\partial x_1\partial x_2} ({\color{black}0},0;\eta
)W^1\Big(\frac{x_1}{\varepsilon},0\Big)+\frac{\partial^2
U^0_p}{\partial x^2_2} ({\color{black}0},0;\eta
)W^2\Big(\frac{x_1}{\varepsilon},0\Big) \textcolor{black}{+\frac{\partial C^\prime_p}{\partial x_2}(0;\eta)}.
\end{split}
\label{(W99)}\end{equation}
A similar formula is valid at $x_2=H$. \textcolor{black}{Using \eqref{(30)} the second and third terms on
the right-hand side of \eqref{(W99)} vanish}. {\color{black} Similarly, the last term vanishes because,  by construction (cf. \eqref{(W9)}, \eqref{(W8)} and \eqref{(W3)}), it satisfies}
$$
\frac{\partial C^\prime_p}{\partial x_2}(0;\eta)=\frac{\partial C^\prime_p}{\partial x_2}(H;\eta)=0,
$$
but the other addends  do so only  if
\begin{equation}\begin{array}{c}\displaystyle
\frac{\partial W^1}{\partial \xi_2}(\xi_1,0)=0,\,\, W^2(\xi_1,0)=0,
\quad \xi_1\in{\mathbb R},
\end{array}\label{(W98)}\end{equation}
{\color{black}cf. also   \eqref{(W12)}.}
{\color{black}There} is no reason for \eqref{(W98)} to be fulfilled for any
asymmetric hole $\omega$ but, owing to Lemma~\ref{Lemma_W}, the assumption
\eqref{(symm)} gives us the relations \eqref{(J8)} and, therefore,
\eqref{(W98)}. Furthermore{\color{black},} all terms on the right-hand side of
\eqref{(W99)} vanish{\color{black}.}

\begin{Remark}\label{Remark_N}
{\rm If the relation \eqref{(W98)} is denied, the inner expansion
\eqref{(W9)} leaves in the Neumann condition \eqref{(14)}
discrepancies of order $1$ which are localized in the vicinity of the points $(0,0)$, $(0,H)$ and decay
exponentially at a distant from them. To compensate, a new
boundary layer is needed involving solutions to the Neumann
problems for the Laplace operator in the half-planes with
semi-infinite families of holes, that is, in
$$
{\mathbb R}^2_+\setminus\bigcup\limits_{k=0}^\infty
\overline{\omega(1,k)}\,\,\mbox{\rm and}\,\,{\mathbb
R}^2_-\setminus\bigcup\limits_{k=0}^\infty \overline{\omega(1,-k)},
$$
cf.~\eqref{(3)}. Asymptotics at infinity of solutions to elliptic
boundary-value problems in angular domains with periodic boundaries
have been investigated in \cite{na143,na395}. However, such a
two-dimensional boundary layer seriously complicates the asymptotic
procedure and we postpone the research in the case of more general perforation for
another paper.}
\end{Remark}

\section{Some bounds for convergence rates}\label{sec5}

In this section, we obtain some important complementary results on the approximation \eqref{(27)}.
In particular, we get some estimates which establish the closeness of  eigenvalues  $\Lambda^\varepsilon(\eta)$ of problem \eqref{(11)}--\eqref{(14)} and  the first three
dispersion curves of the homogenized problem (see Theorem~\ref{Th5.1}). As a consequence, we can identify the first eigenvalue $\Lambda_1^\varepsilon(\eta)$ at a certain distance from the nodes
$(\eta_{\mbox{\tiny$\square$}}, \Lambda_{\mbox{\tiny$\square$}})= (\pm\pi,\pi^2)$
where the question of their splitting does not appear at all (see Corollary~\ref{Th5.2}). For this first eigenvalue, we get a uniform bound for the  convergence rate. The analysis of this section does not take into account the multiplicity of  the eigenvalues of the limit problem.

Let us summarize the results of the section:

\begin{Theorem}\label{Th5.1}
There exist $\Lambda_\star^\varepsilon(\eta)$ and  $\Lambda_\pm^\varepsilon(\eta)$ eigenvalues of  the problem \eqref{(11)}--\eqref{(14)} which satisfy
\begin{eqnarray}
  \vert\Lambda_\star^\varepsilon(\eta)-\Lambda_1^0(\eta)\vert\leq C_0\varepsilon && \quad\mbox{for } \eta\in[-\pi,\pi], \, 0<\varepsilon<\varepsilon_0, \label{VisikL1} \\
  \vert\Lambda_\pm^\varepsilon(\eta)-\Lambda_\pm^0(\eta)\vert\leq C_0\varepsilon && \quad\mbox{for } \eta\in[-\pi,\pi], \, 0<\varepsilon<\varepsilon_0, \label{VisikLpm}
\end{eqnarray}
where $\Lambda_1^0(\eta)=\eta^2$ and $\Lambda_\pm^0(\eta)=(\eta\pm 2\pi)^2$ are eigenvalues of  problem \eqref{(29)}--\eqref{(25)}, and the positive constants  $\varepsilon_0$ and $C_0$ are independent of $\varepsilon$ and $\eta$.
\end{Theorem}

\begin{Corollary}\label{Th5.2}
Let $H\in(0,1).$ Fixed $\delta\in(0,\pi)$, there exists $\varepsilon_0=\varepsilon_0(H)$ such that
the eigenvalue $\Lambda_1^\varepsilon(\eta)$ of  problem \eqref{(11)}--\eqref{(14)}   in the sequence \eqref{(17)},
and the   eigenvalue $\Lambda_1^0(\eta)$  of problem  \eqref{(29)}--\eqref{(25)} in the sequence \eqref{(S2)} satisfy
$$\vert\Lambda_1^\varepsilon(\eta)-\Lambda_1^0(\eta)\vert\leq C_0 \varepsilon\qquad \mbox{for }\eta\in[-\pi+\delta,\pi-\delta] \mbox{ and } 0<\varepsilon<\varepsilon_0,$$
where the positive constant $C_0$ is independent of the parameters $\varepsilon$ and $\eta$.
\end{Corollary}

The proofs of these results are in Section~\ref{subsec54} and use the lemma on almost eigenvalues  which we introduce in Section~\ref{subsec51}. They rely on the construction of approximations to eigenvalues and eigenfunctions which is done in Sections~\ref{subsec52} and \ref{subsec53}.

\subsection{The abstract setting}\label{subsec51}

We first reformulate the spectral problem \eqref{(11)}--\eqref{(14)} in terms of operators on Hilbert spaces, cf. \eqref{(J3)}. In the space
$H^{1,\eta}_{per}(\varpi^\varepsilon)$ we consider the scalar product
\begin{equation}\begin{array}{c}\displaystyle
\langle U^\varepsilon,V^\varepsilon\rangle_{\varepsilon\eta}=(\nabla_xU^\varepsilon,\nabla_xV^\varepsilon)_{\varpi^\varepsilon}+
( U^\varepsilon,V^\varepsilon)_{\varpi^\varepsilon}
\end{array}\label{(J1)}\end{equation}
and the positive,
compact and symmetric
operator ${\cal B}^\varepsilon(\eta)$,
\begin{equation}\begin{array}{c}\displaystyle
\langle {\cal B}^\varepsilon(\eta)U^\varepsilon,V^\varepsilon\rangle_{\varepsilon\eta}=
( U^\varepsilon,V^\varepsilon)_{\varpi^\varepsilon}\quad\forall U^\varepsilon,V^\varepsilon
\in H^{1,\eta}_{per}(\varpi^\varepsilon).
\end{array}\label{(J2)}\end{equation}
The space $H^{1,\eta}_{per}(\varpi^\varepsilon)$ equipped with the
scalar product \eqref{(J1)} is denoted by ${\cal
H}^\varepsilon(\eta)$ and $\|U^\varepsilon;{\cal
H}^\varepsilon(\eta)\|$ {\color{black}denotes the norm} generated by
\eqref{(J1)}.

Comparing \eqref{(J1)}, \eqref{(J2)} with \eqref{(16)}, we see that
the variational formulation of the problem
\eqref{(11)}--\eqref{(14)} is equivalent to the equation
\begin{equation}\begin{array}{c}\displaystyle
{\cal B}^\varepsilon(\eta)U^\varepsilon(\eta)=M^\varepsilon(\eta)
U^\varepsilon(\eta)\,\, \mbox{\rm in}\,\,{\cal H}^\varepsilon(\eta)
\end{array}\label{(J3)}\end{equation}
with the new spectral parameter
\begin{equation}\begin{array}{c}\displaystyle
M^\varepsilon(\eta)=(1+\Lambda^\varepsilon(\eta))^{-1}.
\end{array}\label{(J4)}\end{equation}

The following result  (a lemma on almost eigenvalues,  cf. \cite{ViLu}) is a  consequence of the
spectral decomposition of resolvent, cf.  \cite[Ch. 6]{BiSo}.

\begin{Lemma}\label{Lemma_Visik}
Let $M^\varepsilon_{as}(\eta)\in{\mathbb R}$
and $U^\varepsilon_{as}(\eta) \in{\cal
H}^\varepsilon(\eta)\setminus\{0\}$ verify the relationship
\begin{equation}\label{(J5)}
\|{\cal
B}^\varepsilon(\eta)U^\varepsilon_{as}(\eta)-M^\varepsilon_{as}(\eta)
U^\varepsilon_{as}(\eta);{\cal H}^\varepsilon(\eta)\|=\delta_\varepsilon\|
U^\varepsilon_{as}(\eta);{\cal H}^\varepsilon(\eta)\|.
\end{equation}
Then, there exists an eigenvalue $M^\varepsilon(\eta)$ of the
operator ${\cal B}^\varepsilon(\eta)$ such that
$$
\vert M^\varepsilon(\eta)-M^\varepsilon_{as}(\eta)\vert\leq\delta_\varepsilon.
$$
\end{Lemma}

In Sections~\ref{subsec52} and \ref{subsec53} below, we provide $M^\varepsilon_{as}(\eta)$ and
$U^\varepsilon_{as}(\eta)$ and obtain a bound for  the rest $\delta_\varepsilon$ in \eqref{(J5)}.

\subsection{Approximate eigenvalue and eigenfunction}\label{subsec52}

Let $\Lambda^0_\pm(\eta)=({\color{black}\eta}\pm2\pi)^2$ be  \textcolor{black}{eigenvalues} in
\eqref{(31)} corresponding to a fixed Floquet parameter
$\eta\in[-\pi,\pi]$. According to \eqref{(J4)} we take
\begin{equation}\begin{array}{c}\displaystyle
M_\pm^0(\eta)=(1+\Lambda^0_\pm(\eta))^{-1}
\end{array}\label{(un1)}\end{equation}
as an approximate eigenvalue ($\pm$ respectively), and
\begin{equation}\label{(un2)}
U^\varepsilon_\pm(x;\eta)=X^\varepsilon(x_1)
U^0_\pm(x_1;\eta)+ (1-X^\varepsilon(x_1))\Big(U^0_\pm(0;\eta)+x_1
\frac{\partial U^0_\pm}{\partial x_1}(0;\eta)\Big)+
\varepsilon\chi_0(x_1)\frac{\partial U^0_\pm}{\partial x_1}(0;\eta)
W^1_0\Big(\frac{x}{\varepsilon}\Big),
\end{equation}
as an approximate  eigenfunction constructed from the asymptotic expansions in Section~\ref{sec4} (cf. \eqref{(W1)}, \eqref{(W4)}, \eqref{(W6)} and \eqref{(W9)} which holds for $\eta\in[-\pi,\pi]$).
$W^1_0$ is the bounded harmonics in $\Xi$, see \eqref{(39)} and  \eqref{(W100)},
\begin{equation}\begin{array}{c}\displaystyle
U^0_\pm(x_1;\eta)=e^{i(\eta\pm\textcolor{black}{2}\pi)x_1},
\end{array}\label{(un3)}\end{equation}
\begin{eqnarray}\!\!\displaystyle
\hspace{-1.4cm}&&X^\varepsilon(x_1)=1\!-\!\chi_+(x_1/\varepsilon)\!-\!\chi_-(x_1/\varepsilon), \mbox{ i.e. }
X^\varepsilon(x_1)=1\,\,\mbox{\rm for }\vert x_1\vert\geq 2R\varepsilon,\,\,
X^\varepsilon(x_1)=0\,\,\mbox{\rm for }\vert x_1\vert\leq R\varepsilon, \nonumber
\vspace{0.15cm}\\
\hspace{-1.4cm}&&\displaystyle
\chi_0\in C^\infty({\mathbb R}),\quad \chi_0(x_1)=1\,\,\mbox{\rm
for }\vert x_1\vert\leq 1/6,\,\, \chi_0(x_1)=0\,\,\mbox{\rm
for }\vert x_1\vert\geq 1/3, \label{(ra3)}
\end{eqnarray}
where the even smooth cut-off functions $\chi_{\pm}$ are defined by \eqref{(chi)}.
Notice that,  {\color{black} for $0<\vert\eta\vert<\pi$}, $\Lambda^0_\pm(\eta)$ is a simple eigenvalue so that it
corresponds to the only eigenfunction \eqref{(un3)}, see
\eqref{(31)} with $j=\pm1$ and $k=0$, so that the sign plus or minus
is fixed in these formulas.

\subsection{Estimating the discrepancy}\label{subsec53}

The function \textcolor{black}{\eqref{(un2)}}
satisfies the Neumann condition \eqref{(14)} as well as the
quasi-periodicity conditions \eqref{(12)}, \eqref{(13)}. To conclude
these  {\color{black} assertions}, we recall \eqref{(J8)}  {\color{black} and \eqref{(40)}}, and observe that
$X^\varepsilon(x_1)=1$ and $\chi_0(x_1)=0$ near the points
$x_1=\pm1/2$.

In order to apply Lemma~\ref{Lemma_Visik},  we multiply \eqref{(J5)}  by $ \|
U^\varepsilon_{as}(\eta);{\cal H}^\varepsilon(\eta)\|^{-1}$ and  obtain the relation
\begin{equation}\label{(un4)}
\begin{array}{rl}
\delta^\varepsilon_\pm(\eta):=\!&\!\|U_\pm^\varepsilon;{\cal
H}^\varepsilon(\eta)\|^{-1}\|{\cal
B}^\varepsilon(\eta)U_\pm^\varepsilon-M^0_\pm(\eta)
U_\pm^\varepsilon;{\cal H}^\varepsilon(\eta)\|\vspace{0.1cm}\\
=\!&\!\|U_\pm^\varepsilon;{\cal H}^\varepsilon(\eta)\|^{-1} M^0_\pm(\eta) \,\sup\vert(\nabla_x
U^\varepsilon_\pm,\nabla_x V^\varepsilon)_{\varpi^\varepsilon}-\Lambda^0_\pm(\eta)(
U^\varepsilon_\pm,V^\varepsilon)_{\varpi^\varepsilon}\vert \vspace{0.1cm}\\
 =\!&\!\|U_\pm^\varepsilon;{\cal
H}^\varepsilon(\eta)\|^{-1} M^0_\pm(\eta) \, \sup\vert(\Delta_x
U^\varepsilon_\pm+\Lambda^0_\pm(\eta))
U^\varepsilon_\pm,V^\varepsilon)_{\varpi^\varepsilon}\vert.
\end{array}
\end{equation}
Here,   the supreme is computed over all function $V^\varepsilon\in{\cal
H}^\varepsilon(\eta)$ with unit norm and this calculation takes into
account definitions \eqref{(J1)}, \eqref{(J2)}, {\color{black}\eqref{(un1)}} and the Green formula
together with the Neumann and quasi-periodicity conditions for
$U^\varepsilon_\pm$ and the Neumann and periodicity conditions for $W_0^1$, \eqref{(35)} and \eqref{(40)}, respectively.
Let us show the estimate
\begin{equation}\label{deltacota} \delta^\varepsilon_\pm(\eta)\leq c\varepsilon \quad \mbox{for } \varepsilon \leq \varepsilon_0, \end{equation} with some constants $c$ and $\varepsilon_0$
independent of $\eta$.

Indeed,  by  \eqref{(un2)}, we write
\begin{equation}\label{(un5)}
\Delta_x U_\pm^\varepsilon(x;\eta)+\Lambda^0_\pm(\eta)
{\color{black}U_\pm^\varepsilon}(x;\eta)=:\sum\limits_{j=1}^{{\color{black}6}}S^\varepsilon_{j,\pm}(x;\eta)
\end{equation}
with
\begin{eqnarray*}
&&S^\varepsilon_{1,\pm}(x;\eta)= X^\varepsilon(x_1)(\Delta_xU^0_\pm(x_1;\psi)+\Lambda^0_\pm(\eta)U^0_\pm(x_1;\psi)),\\
&&S^\varepsilon_{2,\pm}(x;\eta)= [\Delta_x,X^\varepsilon(x_1)]\Big(U^0_\pm(x_1;\eta)-U^0_\pm(0;\eta)-x_1
\frac{\partial U^0_\pm}{\partial x_1}(0;\eta)\Big),\\
&&S^\varepsilon_{3,\pm}(x;\eta)= \frac{1}{\varepsilon}\frac{\partial U^0_\pm}{\partial
x_1}(0;\eta)\chi_0(x_1)\Delta_\xi W^1_0(\xi), \\
&&S^\varepsilon_{4,\pm}(x;\eta)= \varepsilon\Lambda^0_\pm(\eta)\chi_0(x_1)\frac{\partial
U^0_\pm}{\partial x_1}(0;\eta)W^1_0(\xi),
\\&&  S^\varepsilon_{5,\pm}(x;\eta)=\varepsilon\frac{\partial U^0_\pm}{\partial
x_1}(0;\eta)[\Delta_x,\chi_0(x_1)]W^1_0(\xi),
\\
&& S^\varepsilon_{6,\pm}(x;\eta)= \Lambda^0_\pm(\eta) (1-X^\varepsilon(x_1))\Big(U^0_\pm(0;\eta)+x_1
\frac{\partial U^0_\pm}{\partial x_1}(0;\eta)\Big);
\end{eqnarray*}
here, $[\Delta_x, \chi]$ stands for the commutator, i.e.
$[\Delta_x, \chi]U:=\Delta_x (\chi U)- \chi \Delta_x U$.
Note that $[\Delta_x, \chi]U=2 \nabla_x \chi \cdot \nabla_x U+ U \Delta_x \chi $ and $[\Delta_x,1-X^\varepsilon(x_1)]=-[\Delta_x,X^\varepsilon(x_1)]$.

We have $S^\varepsilon_{1,\pm}(x;\eta)=0$ and
$S^\varepsilon_{3,\pm}(x;\eta)=0$ because of the equations \eqref{(11)}
for $U^0_\pm$ and \eqref{(34)} for $W^1_0$. Since $W^1_0$ admits the
representation \eqref{(W100)} and, therefore, is bounded together
with its derivative, we conclude that
\begin{eqnarray}
\!\!\!\!\vert(S^\varepsilon_{4,\pm},V^\varepsilon)_{\varpi^\varepsilon}\vert\leq
c\,\varepsilon\sup\limits_{\xi\in\Xi}\vert W^1_0(\xi)\vert\,\|V^\varepsilon;L^2(\varpi^\varepsilon)\|\leq
C\,\varepsilon \textcolor{black}{\,\|V^\varepsilon;L^2(\varpi^\varepsilon)\|}. \quad
\label{(un61)}
\end{eqnarray}
Moreover, since the coefficients in the commutator $[\Delta_x,\chi_0({\color{black}x_1})]$ do not depend on $\varepsilon$
and have  their supports in the union of the rectangles $\Upsilon^0_\pm=\{x:\,\pm x_1\in[1/6,1/3],x_2\in[0,H]\}$, and $\nabla_\xi W_0^1$ has an exponential decay, we have
\begin{equation}\label{(un62)}
\vert(S^\varepsilon_{5,\pm},V^\varepsilon)_{\varpi^\varepsilon}\vert
\leq \,  c\,\varepsilon\sup\limits_{\xi\in\Xi}(\vert W^1_0(\xi)\vert+ \vert\xi_1\vert \vert\nabla_{\xi_1}
W^1_0(\xi)\vert)\,\|V^\varepsilon;L^2(\varpi^\varepsilon)\|
\leq \, C\,
\varepsilon \,\|V^\varepsilon;L^2(\varpi^\varepsilon)\| .
\end{equation}

On the other hand, owing to \eqref{(ra3)}, the support of {\color{black}$S^\varepsilon_{2,\pm}$ belongs to
the union of the thin rectangles ${\color{black}\Upsilon^\varepsilon_\pm}=\{x:\pm
x_1\in[R\varepsilon,2R\varepsilon],x_2\in[0,H]\}$ and the
coefficient of the derivative and the free coefficient in the
commutator
$$
[\Delta_x,X^\varepsilon(x_1)](\cdot)=2\frac{\partial
X^\varepsilon}{\partial x_1}(x_1)\frac{\partial (\cdot)}{\partial x_1}
+\Delta_x X^\varepsilon(x_1)(\cdot)
$$
are of order $\varepsilon^{-1}$ and $\varepsilon^{-2}$ respectively.
{\color{black}Besides, the} inequality
$$
\|V^\varepsilon;L^2(\Upsilon^\varepsilon_\pm)\|\leq
c\varepsilon^{1/2}\|V^\varepsilon;\textcolor{black}{\Hh^\varepsilon(\eta)}\|
$$
is valid, see for example the proof of \eqref{(KK3)} and \eqref{(KK4)}. Thus, based on the Taylor
formula for $U^0_\pm$, we see that
\begin{equation}\label{(un6)}
\vert(S^\varepsilon_{2,\pm},V^\varepsilon)_{\varpi^\varepsilon}\vert\leq \,
c\sum\limits_\pm\vert\Upsilon^\varepsilon_\pm\vert^{1/2}
\max\limits_{x\in\Upsilon^\varepsilon_\pm}\bigg\vert \frac{\partial^2
U^{0}_\pm}{\partial
x_1^2}(x_1;{\color{black}\eta})\bigg\vert\,\|V^\varepsilon;L^2(\Upsilon^\varepsilon_\pm)\|
\leq \,  c\varepsilon \, \textcolor{black}{\|V^\varepsilon;\Hh^\varepsilon(\eta)\|}.
\end{equation}
Similarly, since the support of $S_{6,\pm}^\varepsilon$ is included in $\Theta^\varepsilon=[-2R\varepsilon ,2R\varepsilon]\times [0,H]$, we have
\begin{equation}\begin{array}{c}\displaystyle
\vert(S^\varepsilon_{6,\pm},V^\varepsilon)_{\varpi^\varepsilon}\vert\leq C \vert\Theta^\varepsilon\vert^{1/2} \|V^\varepsilon;\!L^2(\Theta^\varepsilon)\| \leq C\,
\varepsilon \,\|V^\varepsilon;\Hh^\varepsilon(\eta)\| .
\end{array}\label{(un63)}\end{equation}

Finally, by definition of $U_\pm^\varepsilon$ (see \eqref{(un2)} and \eqref{(un3)}), it can be proved that
\begin{equation}\label{(un64)}
\|U_\pm^\varepsilon;{\cal H}^\varepsilon(\eta)\|^2 \lie \|U_\pm^0;L^2(\varpi^0)\|^2+ \|\nabla_x U_\pm^0;L^2(\varpi^0)\|^2=(1+\Lambda_\pm^0(\eta))H.
\end{equation}

Based on the representation \eqref{(un4)}, {\color{black} \eqref{(un5)},} the estimates
\eqref{(un61)}, \eqref{(un62)}, {\color{black} \eqref{(un6)} and \eqref{(un63)},  and the convergence \eqref{(un64)}}, we arrive at   \eqref{deltacota}.

\subsection{Asymptotics of the eigenvalues}\label{subsec54}

Considering {\color{black} the estimate \eqref{deltacota}},  Lemma~\ref{Lemma_Visik} gives us an eigenvalue
$M^\varepsilon_\pm(\eta)$ of the operator ${\cal
B}^\varepsilon(\eta)$ such that
\begin{equation}\begin{array}{c}\displaystyle
\vert M^\varepsilon_\pm(\eta)-M^0_\pm(\eta)\vert\leq c\varepsilon
\end{array}\label{(un7)}\end{equation}
where the factor $c$ is independent of $\eta$. Recalling
\eqref{(J4)}, we derive from \eqref{(un7)} that
\begin{equation} \displaystyle
\vert\Lambda^\varepsilon_\pm(\eta)-\Lambda^0_\pm(\eta)\vert\leq
c\varepsilon(1+
\Lambda^0_\pm(\eta))(1+\Lambda^\varepsilon_\pm(\eta)),
\label{(un8)}\end{equation}
and, hence
$$
(1+\Lambda^\varepsilon_\pm(\eta))(1-c\varepsilon(1+
\Lambda^0_\pm(\eta))\leq 1+\Lambda^0_\pm(\eta).$$
Let us set
$$
\varepsilon_0:=\frac{1}{2c(1+4\pi^2)}.
$$
Then, for $\varepsilon<\varepsilon_0$ and $\eta\in [-\pi,\pi]$, we have $(1-c\varepsilon(1+\Lambda^0_\pm(\eta))\!>\!1/2$ and therefore
\begin{equation}\begin{array}{c}\displaystyle
\vert\Lambda^\varepsilon_\pm(\eta)-\Lambda^0_\pm(\eta)\vert\leq
2c\varepsilon(1+ \Lambda^0_\pm(\eta))^2 \leq 2c\varepsilon (1+9\pi^2)^2=:C_0\varepsilon.
\end{array}\label{(un10)}\end{equation}
This ends the proof of \eqref{VisikLpm}.

In a similar way, replacing $\Lambda_\pm^0(\eta)$ and $U_\pm^0(x_1;\eta)$ by  $\Lambda_1^0(\eta)=\eta^2$ and $U_1^0(x_1;\eta)=e^{i\eta x_1}$ respectively  in \eqref{(un1)} and \eqref{(un2)}, we obtain some constants $\varepsilon_0, C_0 >0$ and certain eigenvalues $\Lambda_\star^\varepsilon(\eta)$ of \eqref{(11)}--\eqref{(14)} which satisfy \eqref{VisikL1}.
Thus, the proof of Theorem~\ref{Th5.1} is completed.

\medskip

Finally, on account of \eqref{EstL2}, there cannot be more than one eigenvalue $\Lambda_p^\varepsilon(\eta)$ in the box $[-\pi+\delta, \pi-\delta]\times [0,\pi^2+K_1]$ for any $\delta>0$ and $K_1$ defined by \eqref{def_K1yK2}, and hence we can identify the eigenvalue  $\Lambda_\star^\varepsilon(\eta)$ given in Proposition~\ref{Th5.1} with the first eigenvalue $\Lambda_1^\varepsilon(\eta)$ at a  distance $\delta$  from $\eta_{\mbox{\tiny$\square$}}= \pm\pi$ and  Corollary~\ref{Th5.2} holds.

\section{Asymptotic analysis near nodes}\label{sec6}

\medskip

The main difference between the asymptotic analysis in the previous and the
next sections is that in what follows the limit eigenvalue under consideration is
always multiple and gives rise to a node
of the dispersion curves in {\color{black}Figure~\ref{fig2}} a)--b).
Furthermore, examining the splitting of the band edges and the opening
of spectral gaps requires much more precise  asymptotic formulas for the eigenvalues in
\eqref{(17)} which are valid in a neighborhood of a certain value of the Floquet
parameter $\eta$. This seriously complicates the asymptotic
analysis as well as the justification procedure.
In fact, the asymptotic analysis is somehow double, since it takes into account the small parameter and the small neighborhood of the nodes
$(\eta_{\mbox{\small{$\circ$}}},\Lambda_{\mbox{\small{$\circ$}}})=(0,4\pi^2)$ and $(\eta_{\mbox{\tiny$\square$}}, \Lambda_{\mbox{\tiny$\square$}})= (\pm\pi,\pi^2).$
In  Sections~\ref{subsec61}--\ref{subsec63}, we perform all the computations for the node $(0,4\pi^2)$ while, for the sake of brevity, we sketch the main changes for the nodes  $(\pm\pi,\pi^2),$ cf. Section~\ref{subsec64}. Section~\ref{subsec61} contains the asymptotic analysis based on asymptotic expansions while Sections~\ref{subsec62}--\ref{subsec63} contain a justification scheme for the abstract formulation in Section~\ref{subsec51}.

\subsection{The node $(\eta_{\mbox{\small{$\circ$}}},\Lambda_{\mbox{\small{$\circ$}}})=(0,4\pi^2)$   for $H\in(0,1/2)$}\label{subsec61}

This node marked with $\circ$ occurs in Figure~\ref{fig2}} a) (cf.~also Figure~\ref{fig_new}) under the assumption $H\in(0,1/2)$ as the intersection point of the two (plus and
minus) limit dispersion curves
\begin{equation}\begin{array}{c}\displaystyle
\Lambda_\pm^0(\eta)=(\eta\pm2\pi)^2,\,\,\eta\in[-\pi,\pi].
\end{array}\label{(K1)}\end{equation}
The problem \eqref{(29)}--\eqref{(25)} with $\eta=0$ has the
eigenvalue $\Lambda^0:=\Lambda_2^0(0)=\Lambda_3^0(0)=4\pi^2$ of multiplicity
$2$ with the eigenfunctions
\begin{equation}\begin{array}{c}\displaystyle
U^0_\pm(x)=e^{\pm2\pi ix_1}.
\end{array}\label{(K2)}\end{equation}

\medskip

To investigate the perturbed dispersion curves \eqref{(18)} with $p=2,3$ near
the point $(\eta_{\mbox{\small{$\circ$}}},\Lambda_{\mbox{\small{$\circ$}}})=(0,4\pi^2)$,  we use the idea in \cite{na453} by introducing the rapid Floquet variable
\begin{equation}\begin{array}{c}\displaystyle
\psi=\varepsilon^{-1}\eta
\end{array}\label{(KK0)}\end{equation}
in a neighborhood of $\eta=0$, and perform the asymptotic ansatz for the eigenvalues
\begin{equation}\begin{array}{c}\displaystyle
\Lambda^\varepsilon_p(\eta)=\Lambda^0+\varepsilon\Lambda^\prime(\psi)+\varepsilon^2
\Lambda^{\prime\prime}(\psi)+\dots
\end{array}\label{(K3)}\end{equation}
with $p=2,3$ as in Figure~\ref{fig3} a). To shorten the notation, we do not
display the index $p$ in the terms of the anz\"{a}tze.

We assume the outer expansion for the corresponding eigenfunction

\begin{equation}\begin{array}{c}\displaystyle
U^\varepsilon(x;\!\eta)=U^0(x;\!\psi)+\varepsilon
U^\prime(x;\!\psi)+\varepsilon^2 U^{\prime\prime}(x;\!\psi)+\dots
\end{array}\label{(K4)}\end{equation}
to be valid in $\varpi^0\setminus\varsigma$, where
\begin{equation}\begin{array}{c}\displaystyle
U^0(x;\psi)=a_+(\psi)e^{+2\pi ix_1}+a_-(\psi)e^{-2\pi ix_1},
\end{array}\label{(K5)}\end{equation}
$\psi$ is a parameter, $\psi=O(1)$, and $a(\psi)=(a_+(\psi),a_-(\psi))$ is a column vector in ${\mathbb C}^2$
to be determined together with the correction terms
$\Lambda^\prime(\psi)$ and $U^\prime(\cdot;\psi)$ in the ans\"{a}tze
\eqref{(K3)} and  \eqref{(K4)}, respectively.
We follow the technique developed in Sections~\ref{subsec41}--\ref{subsec44} and we only outline the main differences.
As in Section~\ref{sec4}, the terms
$\Lambda^{\prime\prime}(\psi)$ in \eqref{(K3)} and
$U^{\prime\prime}(x;\psi)$ in \eqref{(K4)} are not of further use.

We look for an inner expansion in the vicinity of the transversal perforation string \eqref{(32)}
\begin{equation}\begin{array}{c}\displaystyle
U^\varepsilon(x;\eta)=w^0(x_2;\psi)+\varepsilon
w^\prime(\xi;\psi) +\varepsilon^2
w^{\prime\prime}(\xi;\psi)+\dots,
\end{array}\label{(K6)}\end{equation}
where we have assumed that the main term $w^0$ does not depend on $\xi=\varepsilon^{-1}x$ while the functions arising in further terms, $ w^\prime $ and $w^{\prime\prime}$ satisfy a  periodicity condition in the $\xi_2$-direction.
Following the scheme in Section~\ref{sec4}, the immediate result of the
matching procedure at the first order is
\begin{equation}\begin{array}{c}\displaystyle
w^0(x_2;\psi)=U^0(0;\psi)=(a_+(\psi)+a_-(\psi))W^0=a_+(\psi)+a_-(\psi),
\end{array}\label{(K7)}\end{equation}
cf.~\eqref{(W6)} and  \eqref{W0}.
Since the main term \eqref{(K6)} is independent of the transversal
variable, the dependence on $x_2$ (not on $\xi_2$!) disappear in all
terms and we will write the argument $x_1$ instead of $x$ on the
right-hand side of \eqref{(K4)} and omit $x_2$  on the right-hand
side of \eqref{(K6)}.

We continue with the matching procedure at the second order taking into account the Taylor expansion for \eqref{(K4)}, cf.~\eqref{(W5)}.
The Taylor formula applied to \eqref{(K5)} gives
\begin{equation}\label{(K8)}
\displaystyle
U^0(x;\psi)= a_+(\psi)+a_-(\psi)+2\pi
ix_1(a_+(\psi)-a_-(\psi)) -\textcolor{black}{2}\pi^2
 x_1^2(a_+(\psi)+a_-(\psi))+O(\vert x_1\vert^3),
\end{equation}
where $O(\vert x_1\vert^3)$ depends on $\psi$, and, recalling the solution \eqref{(39)} of the problem \eqref{(34)}, \eqref{(35)}, \eqref{(36)}, cf.~\eqref{(W9)}, we set
\begin{equation}\begin{array}{c}\displaystyle
w^\prime(\xi;\psi)=2\pi i
(a_+(\psi)-a_-(\psi))W^1(\xi)+a^\prime(\psi)W^0
\end{array}\label{(K9)}\end{equation}
with some factor $a^\prime(\psi)$ which can be fixed arbitrarily at
the present stage of our analysis. In contrast to \textcolor{black}{\eqref{(W9)}} the
solution $W^2$ is absent in \eqref{(K9)}.
Thus, the first jump condition for the correction term in \eqref{(K4)} is (cf.~\eqref{(W5)}, \eqref{(W11)} and \eqref{(K8)}):
 \begin{equation}\begin{array}{c}\displaystyle
[U^\prime]_0(\psi)=2\frac{\partial U^0}{\partial
x_1}(0;\!\psi)m_1(\Xi)=4\pi i(a_+(\psi)-a_-(\psi)) m_1(\Xi).
\end{array}\label{(K12)}\end{equation}
This formula coincides with \eqref{(W11)} because $\partial_{x_2}
U^0\!=\!0,$ and it is independent of $x_2$.

The matching procedure at level $\varepsilon^2$, in the same way as in Section~\ref{subsec43}, gives
$$
w^{\prime\prime}(\xi;\psi)=4\pi^2(a_+(\psi)+a_-(\psi))W^3(\xi)+a^{\prime\prime}(\psi)W^0+ \widetilde{w}^{\,\prime\prime}(\xi;\psi)
$$
where $W^3$ is the solution \eqref{(SS2)} of the problem \eqref{(S1)}, \eqref{(35)}, \eqref{(36)}, $a^{\prime\prime}(\psi)$ is some factor which can be fixed arbitrarily at
the present stage of our analysis and the remainder $\widetilde{w}^{\,\prime\prime}$ gets the
exponential decay as $\xi_1\to\pm\infty$ (cf.~\eqref{(S15)}, \eqref{(S16)} and \eqref{(S17)}). Besides,
the second jump condition \eqref{(W11N)} now takes the simplified
form (cf.~\eqref{(W11N)} and \eqref{(K8)})
 \begin{equation}\begin{array}{c}\displaystyle
\Big[\frac{\partial U^\prime}{\partial
x_1}\Big]_0\textcolor{black}{(\psi)}=\textcolor{black}{-}\frac{\vert\omega\vert}{H}\frac{\partial^2
U^0}{\partial x^2_1}({\color{black}0};\psi)=\textcolor{black}{4}\pi^2
(a_+(\psi)+a_-(\psi))\frac{\vert\omega\vert}{H}.
 \end{array}\label{(K12N)}\end{equation}

Other restrictions on $U^\prime$ are readily inherited from \eqref{(11)}, \eqref{(14)}
and \eqref{(K3)}, cf. Section~\ref{subsec41}:
 \begin{equation}\begin{array}{c}
-\Delta_x U^\prime(x;\psi)-\Lambda^0 U^\prime(x;\psi)=\Lambda^\prime(\psi)U^0(x;\psi),
\,\, x\in\varpi^0\setminus\varsigma,
\\\\\displaystyle
\frac{\partial U^\prime}{\partial x_2}(x_1,0;\psi)=
\frac{\partial U^\prime}{\partial x_2}(x_1,H;\psi)=0,\,\,
x_1\in\Big(-\frac{1}{2},0\Big)\cup\Big(0,\frac{1}{2}\Big).
\end{array}\label{(K13)}\end{equation}
In the quasi-periodicity conditions it is also necessary to take
into account the fast Floquet parameter \eqref{(KK0)} and the Taylor
formula
 \begin{equation}\begin{array}{c}
e^{i\eta}=e^{i\varepsilon\psi}=1+i \varepsilon\psi+O(\varepsilon^2).
\end{array}\label{(K99)}\end{equation}
In this way, inserting \eqref{(K4)} into \eqref{(12)}, \eqref{(13)}, collecting terms of order $\varepsilon$ and using \eqref{(K5)} yield
\begin{equation}\label{(K14)}\begin{split}\displaystyle
&U^\prime\Big(\frac{1}{2},x_2;\psi\Big)\!-\!U^\prime\Big(\!\!-\frac{1}{2},x_2;\psi\Big)\!=\!i\psi
U^0\Big(\!\!-\frac{1}{2},x_2;\psi\Big)\!=\!-i\psi(a_+(\psi)+a_-(\psi)), \vspace{0.3cm}\\
&\displaystyle
\frac{\partial U^\prime}{\partial
x_1}\Big(\frac{1}{2},x_2;\psi\Big)\!-\! \frac{\partial
U^\prime}{\partial x_1}\Big(\!\!-\frac{1}{2},x_2;\psi\Big)\! =\! i\psi
\frac{\partial U^0}{\partial
x_1}\Big(\!\!-\frac{1}{2},x_2;\psi\Big)\!=\!2\pi\psi(a_+(\psi)\!-\!a_-(\psi)).
\end{split}\end{equation}

The problem \eqref{(K12)},  \eqref{(K12N)}, \eqref{(K13)},
\eqref{(K14)} has two compatibility conditions which can be derived by
multiplying the partial differential equations in \eqref{(K13)} by  the eigenfunctions
\eqref{(K2)} and applying the Green formula on $\varpi^0\setminus\varsigma$.
Thus, we have
\begin{equation}\label{(K15)}
\begin{split}
\Lambda^\prime(\psi)Ha_\pm(\psi)&=-\int\limits_{\varpi^0}
\overline{e^{\pm2\pi ix_1}}\left(\Delta_x U^\prime(x;\psi)+\Lambda^0 U^\prime(x;\psi)\right)dx\\
&=\int_0^H \!\!\Big(
\frac{\partial U^\prime}{\partial x_1}(x;\psi)\pm 2\pi
iU^\prime(x;\psi)\Big)\bigg\vert_{x_1=-\frac{1}{2}}^{\frac{1}{2}} \! dx_2+ \!
\int_0^H \!\! \Big[
\frac{\partial U^\prime}{\partial x_1}(x;\psi)\pm 2\pi
iU^\prime(x;\psi)\Big]_0 \!dx_2.
\end{split}
\end{equation}
Notice that the factor $H$ on the left-hand side is due to the formula
 \begin{equation}\nonumber\begin{array}{c}\displaystyle
 \|e^{\pm2 i\pi x_1};L^2(\varpi^0)\|=H^{1/2}.
\end{array}\label{(luc000)}\end{equation}
Using the inhomogeneous data in  \eqref{(K12)}, \eqref{(K12N)} and \eqref{(K14)},
we observe that the integrands are constants and reduce \eqref{(K15)} to
the system of two  linear algebraic equations with
the spectral parameter $\Lambda^\prime(\psi)$:
\begin{equation}\begin{array}{l}\displaystyle
\Lambda^\prime(\psi)a_+(\psi) =   \left( 4\pi^2\frac{\vert\omega\vert}{H}  - 8\pi^2m_1(\Xi) + 4\pi\psi \right) a_+(\psi)
+  \left( 4\pi^2\frac{\vert\omega\vert}{H} +  8\pi^2m_1(\Xi)  \right)  a_-(\psi),\vspace{0.3cm}\\
\displaystyle
\Lambda^\prime(\psi)a_-(\psi) = \left(\! 4\pi^2\frac{\vert\omega\vert}{H} +  8\pi^2m_1(\Xi) \right)  a_+(\psi) +  \left( 4\pi^2\frac{\vert\omega\vert}{H}  -  8\pi^2m_1(\Xi)-4\pi\psi\! \right) a_-(\psi).
\end{array}\label{(K16)}\end{equation}
The two eigenvalues of this system are
 \begin{equation}\begin{array}{c}\displaystyle
\textcolor{black}{\Lambda^{\prime}_\pm}(\psi)=4\pi\left(2\pi\textcolor{black}{
\Big(\frac{\vert\omega\vert}{2H}-m_1(\Xi)\Big)}\pm \sqrt{4\pi^2
\Big(m_1(\Xi)\textcolor{black}{+}\frac{\vert\omega\vert}{2H}\Big)^2+\psi^2}\right).
 \end{array}\label{(K17)}\end{equation}
In particular, we have $ {\color{black}\Lambda_-^{\prime}}(\psi)<0$ and ${\color{black} \Lambda_+^{\prime}}(\psi)>0$ because
 \begin{equation}
\Lambda^{\prime}_{-}(\psi)\leq -16\pi^2m_1(\Xi)\quad \hbox{ and }\quad  \Lambda^{\prime}_+(\psi)\geq 8\pi^2\frac{\vert\omega\vert}{H},
\label{(K17n)}\end{equation}
and $m_1(\Xi)>0$, see \eqref{(42)} and Remark~\ref{remarknueva} to compare.

In addition, the corresponding eigenvectors
$a^\pm(\psi)=(a^\pm_+(\psi),a^\pm_-(\psi))$ can also be easily computed. Finally, the compatibility conditions in the problem
\eqref{(K13)}, \eqref{(K12)}, \eqref{(K12N)}, \eqref{(K14)} are
satisfied and it has a solution $U^\prime(x;\psi)$ which is defined
up to a linear combination of the eigenfunctions \eqref{(K2)} but, in the sequel, it can be fixed orthogonal to them  and therefore become unique.
This condition determines all the terms in the asymptotic ans\"{a}tze
\eqref{(K3)}, \eqref{(K4)} and \eqref{(K6)}.

According to \eqref{(K17n)} we have $\Lambda^\prime_+(\psi)>
\Lambda^\prime_-(\psi)$, so that the eigenpair
$\{\Lambda^\prime_-(\psi),a^-(\psi)\}$ can be related to the eigenpair
$\{\Lambda^\varepsilon_2(\eta), U^\varepsilon_2(x;\psi)\}$ of the
problem \eqref{(11)}--\eqref{(14)} while
$\{\Lambda^\prime_+(\psi),a^+(\psi)\}$ does to
$\{\Lambda^\varepsilon_3(\eta),U^\varepsilon_3(x;\psi)\}$.

Now, we formulate our result on the splitting edges of the second and third limit spectral bands
giving rise to the open gap $\gamma_2^\varepsilon$ (cf.~Figure~\ref{fig3} a)). Its proof is in Sections~\ref{subsec62}--\ref{subsec63}.

\begin{Theorem}\label{Theorem_multiple1}
Let $H\in(0,1/2)$ and $\psi_0>0$. Then,  there
exist positive $\varepsilon_0=\varepsilon_0(H{\color{black},\psi_0})$ and $C=C(H{\color{black},\psi_0})$ such
that, for $\varepsilon\in(0,\varepsilon_0]$, the entries
$\Lambda^\varepsilon_2(\eta)$ and $\Lambda^\varepsilon_3(\eta)$ of
the eigenvalue sequence \eqref{(17)} with $\eta=\varepsilon\psi$, {\color{black}$\vert\psi \vert\leq \psi_0$,}
meet the estimates
\begin{equation}\nonumber\begin{array}{c}\displaystyle
\vert\Lambda^\varepsilon_3(\varepsilon\psi)-4\pi^2-\varepsilon{\color{black}\Lambda_+^{\prime}}(\psi)\vert\leq C\varepsilon^2,\\\\
\vert\Lambda^\varepsilon_2(\varepsilon\psi)-4\pi^2-\varepsilon{\color{black}\Lambda_-^{\prime}}(\psi)\vert\leq C\varepsilon^2,
\end{array}\label{(de12)}\end{equation}
where the quantities ${\color{black}\Lambda_\pm^{\prime}}(\psi)$ are given by \eqref{(K17)}.
\end{Theorem}

\subsection{Approximate eigenvalues and eigenfunctions}\label{subsec62}

Recalling Section~\ref{subsec51}, based on calculations in Section~\ref{subsec61},
we set
\textcolor{black}{\begin{equation}\begin{array}{c}\displaystyle
M^{1,\varepsilon}_\pm(\psi)=(1+4\pi^2+\varepsilon\Lambda^{\prime}_\pm(\psi))^{-1}
\end{array}\label{(ra1)}\end{equation}
where $\Lambda^{\prime}_\pm(\psi)$} are taken from
\eqref{(K17)}. Similarly to \eqref{(un2)}, based on the asymptotic formulas in Section~\ref{subsec61}, we define the approximate eigenfunction
\begin{multline}\label{(ra2)}
\displaystyle U^{\varepsilon}_\pm(x;\psi)=  \displaystyle X^\varepsilon(x_1)
\big(U^{0}_\pm(x_1;\psi)+\varepsilon
 U^{\prime}_\pm(x_1;\psi)\big) +
\varepsilon\chi_0(x_1)\Big(\widetilde{w}^{\prime}_\pm(\frac{x}{\varepsilon};\psi)+
\varepsilon\widehat{w}^{\,\prime\prime}_\pm(\frac{x}{\varepsilon};\psi)\Big)+\varepsilon^2 R^{\varepsilon}_\pm(x;\psi)
\vspace{0.1cm}\\
 \displaystyle +(1-X^\varepsilon(x_1))\bigg(U^{0}_\pm(0;\psi)+x_1
\frac{\partial U^{0}_\pm}{\partial x_1}(0;\psi)+\frac{x_1^2}{2}
\frac{\partial^2 U^{0}_\pm}{\partial x_1^2}(0;\psi)
+\frac{\varepsilon}{2}\sum\limits_{\tau=\pm} \Big(U^{\prime}_\pm(\tau
0;\psi)+x_1 \frac{\partial U^{\prime}_\pm}{\partial
x_1}(\tau0;\psi)\Big)\bigg).
\end{multline}
Let us describe the terms arising in \eqref{(ra2)}.

The cut-off functions are defined in \eqref{(ra3)}. The main
term \textcolor{black}{$U^{0}_\pm$} is the linear combination
$$U_\pm^0(x;\psi)=a_+^\pm(\psi)e^{+2\pi ix_1}+a_-^\pm(\psi)e^{-2\pi ix_1},$$
where, for each sign $\pm$,  the coefficient column vector \textcolor{black}{$a^\pm(\psi)=(a^\pm_+(\psi),a^\pm_-(\psi))$} is the eigenvector of the
system \eqref{(K16)} with
$\Lambda^\prime(\psi)=\textcolor{black}{\Lambda^{\prime}_\pm(\psi)}$ and
\textcolor{black}{$U^{\prime}_\pm$} is a solution of the problem \eqref{(K13)},
\eqref{(K14)}, \eqref{(K12)}, \eqref{(K12N)}, the compatibility
conditions of which are fulfilled due to \eqref{(K16)}. Both,
\textcolor{black}{$U^{0}_\pm$} and \textcolor{black}{$U^{\prime}_\pm$}, depend on the variable $x_1$ only.
\textcolor{black}{We fix the main term by prescribing the normalization condition
\begin{equation}\begin{array}{c}\displaystyle
\vert a^\pm(\psi)\vert=1\,\,\mbox{\rm
which implies}\,\,\|U^{0}_\pm;H^2(\varpi^0)\|=C_0>0.
\end{array}\label{(ra6)}\end{equation}
Then, the solution of the problem \eqref{(K13)}, \eqref{(K14)},
\eqref{(K12)}, \eqref{(K12N)} meets the estimate
\begin{equation}\begin{array}{c}\displaystyle
\|U^{\prime}_\pm;C^2(\overline{\varpi^0\cap \{x_1>0}\}\|+
\|U^{\prime}_\pm;C^2(\overline{\varpi^0\cap \{x_1<0}\}\|\leq C_1(1+\vert\psi\vert)
\end{array}\label{(ra7)}\end{equation}
due to the factor $\psi$ on the right-hand sides of} \textcolor{black}{\eqref{(K14)} and \eqref{(K17)}.}

The boundary layer terms $\widetilde{w}^{\prime}_\pm$ and $\widehat{w}^{\,\prime\prime}_\pm$  take the
form
\begin{equation}\begin{array}{c}\displaystyle
\textcolor{black}{\widetilde{w}^{\prime}_\pm(\xi;\psi)=\frac{\partial
U^{0}_\pm}{\partial x_1}(0;\psi)\widetilde{W}^{\,1}_0(\xi)=2\pi
i(a^\pm_+(\psi)- a^\pm_-(\psi))\widetilde{W}^{\,1}_0(\xi)}
 \end{array}\label{(ra4)}\end{equation}
and
\begin{equation}\begin{array}{c}\displaystyle
\widehat{w}^{\,\prime\prime}_\pm(\xi;\psi)=\textcolor{black}{-}\frac{\partial^2
U^{0}_\pm}{\partial
x^2_1}(0;\psi)\widehat{W}^{\,3}(\xi)+\frac{1}{2}\sum\limits_{\tau=\pm}\frac{\partial
U^{\prime}_\pm}{\partial x_1}(\tau 0;\psi)\widetilde{W}^{\,1}_0(\xi)
 \end{array}\label{(ra44)}\end{equation}
where $\widetilde{W}_0^{\,1}$ is the exponentially decaying remainder
in the decomposition \eqref{(W100)} while
 $\widehat{W}^{\,3}$ is a bounded part of the solution
\eqref{(SS2)} of the problem \eqref{(S1)}, \eqref{(35)},
\eqref{(36)}, that is,
 \begin{equation}\label{(ra45)}
\widehat{W}^{\,3}(\xi)=W^3(\xi)+\frac{\xi_1^2}{2}-\sum\limits_{\tau=\pm}\tau
\chi_\tau(\xi_1)\frac{\vert\omega\vert}{2H}\xi_1.
\end{equation}
Recalling Proposition~\ref{Proposition_W1} and the relation
\eqref{(ra6)}, we write
\begin{equation}\label{(ra9)}
\|e^{\sigma\vert\xi_1\vert}\textcolor{black}{\widetilde{w}^{\,\prime}_\pm};H^2(\Xi)\|\leq
C_3\,\,\mbox{\rm with
any}\,\,\sigma\in\Big(0,\frac{2\pi}{H}\Big).
\end{equation}

Finally in \eqref{(ra2)}, we fix $R^{\varepsilon}_\pm$ to get $U^{\varepsilon}_\pm\in{\cal H}^\varepsilon(\eta)\setminus\{0\}$.
First, we  take functions $R^{\varepsilon}_\pm\in H^2(\varpi^0)$ such that they have support
in $[1/4,1/2]\times[0,H]$ and satisfy the boundary conditions
\begin{equation}\label{(ra5)}
\begin{split}
& \frac{\partial R^{\varepsilon}_\pm}{\partial
x_2}(x_1,0;\psi)=\frac{\partial R^{\varepsilon}_\pm}{\partial
x_2}(x_1,H;\psi)=0,\,\, x_1\in\Big(-\frac{1}{2},\frac{1}{2}\Big), \\
& R^{\varepsilon}_\pm\Big(\frac{1}{2},x_2;\psi\Big)=\varepsilon^{-2}
\big(e^{i\varepsilon\psi}-1-i\varepsilon\psi\big)
U^{0}_\pm\Big(-\frac{1}{2};\psi\Big)  -\varepsilon^{-1}
\big(e^{i\varepsilon\psi}-1\big)
U^{\prime}_\pm\Big(-\frac{1}{2};\psi\Big), \,\, x_2\in(0,H), \\
& \frac{\partial R^{\varepsilon}_\pm}{\partial
x_1}\Big(\frac{1}{2},x_2;\psi\Big)=\varepsilon^{-2}
\big(e^{i\varepsilon\psi}-1-i\varepsilon\psi\big) \frac{\partial
U^{0}_\pm}{\partial x_1}\Big(-\frac{1}{2};\psi\Big)  -
\varepsilon^{-1} \big(e^{i\varepsilon\psi}-1\big) \frac{\partial
U^{\prime}_\pm}{\partial x_1}\Big(-\frac{1}{2};\psi\Big),  \,\, x_2\in(0,H).
\end{split}
\end{equation}
Applying the Taylor formula to $e^{i\varepsilon\psi}$, and taking into account \eqref{(ra6)} and \eqref{(ra7)}, we
find a function $R_\pm^\varepsilon$ such that, in addition to \eqref{(ra5)},  satisfies \begin{equation}\begin{array}{c}\displaystyle
\|R^{\varepsilon}_\pm;H^2(\varpi^0)\|\leq
C_2\vert\psi\vert(1+\vert\psi\vert).
\end{array}\label{(ra8)}\end{equation}
Owing to the relations \eqref{(ra5)}, the approximate eigenfunction
\eqref{(ra2)} meets the quasi-periodicity conditions \eqref{(12)}, \eqref{(13)} with
$\eta=\varepsilon\psi$. Thus,  $U_\pm^{\varepsilon}$ falls into ${\cal H}^\varepsilon(\eta)\setminus\{0\}$.

Note that the function \eqref{(ra2)} satisfies the Neumann boundary condition
on the lateral sides  $x_2=0$ and $x_2=H$ of the periodicity cell because of
\eqref{(ra5)} for $R^{\varepsilon}_\pm$ and \eqref{(J8)} for
\eqref{(ra4)} and \eqref{(ra44)}.
At the boundaries of the holes
$\varpi^\varepsilon(0,k)$ \textcolor{black}{with $k=0,\dots,N-1$}, \textcolor{black}{by  definition \eqref{(ra3)}
of the cut-off functions, \eqref{(ra4)}, \eqref{(ra44)} and \eqref{(33)}, we obtain for $ x\in \partial \varpi^\varepsilon(0,k)$:}
\begin{equation*}
\begin{split}
\partial_{\nu(x)}U^{\varepsilon}_\pm(x;\psi)=&\bigg(\frac{\partial
U^{0}_\pm}{\partial
x_1}(0;\psi)+\frac{\varepsilon}{2}\sum\limits_\pm\frac{\partial
U^{\prime}_\pm}{\partial x_1}(\tau 0;\psi)\!\bigg)\Big(\nu_1(\xi)+
\partial_{\nu(\xi)}\widetilde{W}^{1}_0(\xi)\Big)
\\&+ \varepsilon\frac{\partial^2 U^{0}_\pm}{\partial
x^2_1}(0;\psi)\Big(\nu_1(\xi)\xi_1-
\partial_{\nu(\xi)}\widehat{W}^{3}(\xi)\!\Big).
\end{split}
\end{equation*}
Now, by \eqref{(40)} and \eqref{(W100)}, we have that $
\partial_{\nu(\xi)}\widetilde{W}^{1}_0(\xi)=-\nu_1(\xi)$.
Moreover, since $\widehat{W}^{3}$ is defined by  \eqref{(ra45)} with $W^3$ satisfying \eqref{(36)}, we get
$\partial_{\nu(\xi)}\widehat{W}^{3}(\xi)=\xi_1\nu_1(\xi)$. Therefore, for
$ x\in \partial \varpi^\varepsilon(0,k)$,
 we get $\partial_{\nu(x)}U^{\varepsilon}_\pm(x;\psi)=0$.
 }

\subsection{Estimating the discrepancy}\label{subsec63}

We take the value \eqref{(ra1)} and the function \eqref{(ra2)} to be the almost eigenvalue and eigenfunction respectively and follow the analysis of Section~\ref{subsec53}. To make it easier the analysis, we also keep the same notations.

Let us proceed to apply Lemma~\ref{Lemma_Visik}. Considering \eqref{(ra1)}, we have
\begin{equation}\label{(de2)}
\begin{split}
\delta_\pm^\varepsilon(\psi):\!&\!=\|U^{\varepsilon}_\pm;{\Hh}^\varepsilon(\varepsilon\psi)\|^{-1}
\|{\Bb}^\varepsilon(\varepsilon\psi)U^{\varepsilon}_\pm-M^{1,\varepsilon}_\pm(\psi)
U^{\varepsilon}_\pm;{\Hh}^\varepsilon(\varepsilon\psi)\|\\
\!&\!=\|U^{\varepsilon}_\pm;{\Hh}^\varepsilon(\varepsilon\psi)\|^{-1}
M^{1,\varepsilon}_\pm(\psi)\, \sup \vert(\Delta_x
U^{\varepsilon}_\pm\!+\!(4\pi^2\!+\!\varepsilon\Lambda^\prime_\pm(\psi))
U^{\varepsilon}_\pm,V^\varepsilon)_{\varpi^\varepsilon}\vert.
\end{split}\end{equation}
The supreme is computed over all function $V^\varepsilon\in{\cal
H}^\varepsilon(\varepsilon \psi)$ with unit norm and this calculation takes into
account definitions \eqref{(J1)}, \eqref{(J2)} and the Green formula
together with the Neumann and quasi-periodicity conditions for
$U^\varepsilon_\pm$.
For any fixed $\psi_0>0$,  let us show the estimate
\begin{equation}\label{deltacota_bis}
\delta_{\pm}^\varepsilon(\psi)\leq c(\psi_0) \varepsilon^2 \qquad \mbox{for } \vert\psi\vert<\psi_0, \, \varepsilon\leq \varepsilon_0
\end{equation}
with $c(\psi_0)$ and $\varepsilon_0=\varepsilon_0(\psi_0)$ some constants independent of $\psi$ but they depend on $\psi_0.$

Indeed, we write
\begin{equation}\label{(un5bis)}
\Delta_x
U^{\varepsilon}_\pm(x;\psi)+\big(4\pi^2+\varepsilon\Lambda^\prime_\pm(\psi)\big)
U^{\varepsilon}_\pm(x;\psi)=:\sum\limits_{j=1}^{10}S^\varepsilon_{j,\pm}(x;\psi),
\end{equation}
where
\begin{eqnarray*}
&&S^\varepsilon_{1,\pm}(x;\psi)=
X^\varepsilon(x_1)(\Delta_xU^{0}_\pm(x_1;\psi)+4\pi^2U^{0}_\pm(x_1;\psi)), \\
&&S^\varepsilon_{2,\pm}(x;\psi)=
\varepsilon X^\varepsilon (x_1)\big(\Delta_xU^{\prime}_\pm(x_1;\psi)+4\pi^2U^{\prime}_\pm(x_1;\psi)+
\Lambda^{\prime}_\pm(\psi)U^{0}_\pm(x_1;\psi)\big), \\
&& S^\varepsilon_{3,\pm}(x;\psi)=
\varepsilon^2\left(X^\varepsilon(x_1) \Lambda^{\prime}_\pm(\psi)U^{\prime}_\pm(x_1;\psi)+
\big(\Delta_x+4\pi^2+\varepsilon\Lambda^\prime_\pm(\psi)\big)
R^{\varepsilon}_\pm(x;\psi)\right),\\
&& S^\varepsilon_{4,\pm}(x;\psi)=
\big(4\pi^2+\varepsilon\Lambda^\prime_\pm(\psi)\big)
(1-X^\varepsilon(x_1))\left(U^{0}_\pm(0;\psi)+x_1
\frac{\partial U^{0}_\pm}{\partial x_1}(0;\psi)+\frac{x_1^2}{2}
\frac{\partial^2 U^{0}_\pm}{\partial x_1^2}(0;\psi)\right)\\
&&\hspace{2.5cm}+(1-X^\varepsilon(x_1))\frac{\partial^2 U^{0}_\pm}{\partial x_1^2}(0;\psi), \\
&& S^\varepsilon_{5,\pm}(x;\psi)=
\big(4\pi^2+\varepsilon\Lambda^\prime_\pm(\psi)\big)
(1-X^\varepsilon(x_1))\frac{\varepsilon}{2}\sum\limits_{\tau=\pm} \Big(U^{\prime}_\pm(\tau
0;\psi)+x_1 \frac{\partial U^{\prime}_\pm}{\partial
x_1}(\tau0;\psi)\Big),\\
&& S^\varepsilon_{6,\pm}(x;\psi) =
 [\Delta_x,X^\varepsilon(x_1)]\bigg(U^{0}_\pm(x_1;\psi)-U^{0}_\pm(0;\psi)-x_1
\frac{\partial U^{0}_\pm}{\partial x_1}(0;\psi)-\frac{x_1^2}{2}
\frac{\partial^2 U^{0}_\pm}{\partial x_1^2}(0;\psi)\bigg),\\
&&S^\varepsilon_{7,\pm}(x;\psi)=
 \varepsilon[\Delta_x,X^\varepsilon(x_1)]\bigg(U^{\prime}_\pm(x_1;\psi)
- \frac{1}{2}\sum\limits_{\tau=\pm}\bigg(U^{\prime}_\pm(\tau 0;\psi)
+x_1 \frac{\partial U^{\prime}_\pm}{\partial x_1}(\tau 0;\psi)\bigg)\bigg),\\
&&S^\varepsilon_{8,\pm}(x;\psi)=
 \chi_0(x_1) \Big(\ee^{-1}\Delta_\xi\widetilde{w}^{\prime}_\pm(\xi;\psi)+
\Delta_\xi\widehat{w}^{\,\prime\prime}_\pm(\xi;\psi)\Big), \\
&&S^\varepsilon_{9,\pm}(x;\psi)=
\varepsilon[\Delta_x,\chi_0(x_1)]\big(\widetilde{w}^{\,\prime}_\pm(\xi;\psi)+
\varepsilon\widehat{w}^{\,\prime\prime}_\pm(\xi;\psi)\big),\\
&&S^\varepsilon_{10,\pm}(x;\psi)= \varepsilon (4\pi^2+\varepsilon \Lambda^\prime_\pm(\psi))\chi_0(x_1)
\big(\widetilde{w}^{\prime}_\pm(\xi;\psi) + \varepsilon\widehat{w}^{\,\prime\prime}_\pm(\xi;\psi)\big).
\end{eqnarray*}
Let us estimate the scalar products
 $$
I^\varepsilon_j(V^\varepsilon;\psi)=(S^\varepsilon_{j,\pm},V^\varepsilon)_{\varpi^\varepsilon} \quad \mbox{for } j=1,2,\ldots,10 \mbox{ and } V^\varepsilon\in{\cal H}^\varepsilon(\varepsilon\psi).
 $$

First of all, according the definitions of
$U_\pm^0$ and $U_\pm^\prime$ (cf. \eqref{(29)}  and \eqref{(K13)}) there holds $S^\varepsilon_1=0$ and
$S^\varepsilon_2=0$ so that
\begin{equation}\begin{array}{c}\displaystyle
I^\varepsilon_1(V^\varepsilon;\psi)=0\,\,\mbox{\rm
and}\,\,I^\varepsilon_2(V^\varepsilon;\psi)=0.
\end{array}\label{(de4)}\end{equation}
Furthermore,
by \eqref{(K17)}, \eqref{(ra7)} and \eqref{(ra8)}  we readily derive the estimate
$$
\vert I^\varepsilon_3(V^\varepsilon;\psi)\vert\leq
c_1\varepsilon^2(1+\vert\psi\vert)^3\|V^\varepsilon;L^2(\varpi^\varepsilon)\|.
$$

Now, using the  definition of $U^0_\pm$, we write
\begin{equation*}
\begin{split}
S^\varepsilon_{4,\pm}(x;\psi)=  &\varepsilon\Lambda^\prime_\pm(\psi)(1-X^\varepsilon(x_1)) U^{0}_\pm(0;\psi)\\
&+\big(4\pi^2+\varepsilon\Lambda^\prime_\pm(\psi)\big)
(1-X^\varepsilon(x_1))\left( \!x_1
\frac{\partial U^{0}_\pm}{\partial x_1}(0;\psi)+\frac{x_1^2}{2}
\frac{\partial^2 U^{0}_\pm}{\partial x_1^2}(0;\psi)\right).
\end{split}
\end{equation*}
Thus, by construction of the test
function $X^\varepsilon$, the support of $S_{4,\pm}^\varepsilon$ is included in $\Theta^\varepsilon=[-2\varepsilon R, 2\varepsilon R ] \times [0,H]$ and   we easily obtain the estimate
\begin{equation}\nonumber
\vert I^\varepsilon_4(V^\varepsilon;\psi)\vert\leq
{\color{black}c_2 \varepsilon (1+\vert\psi\vert) \vert\Theta^\varepsilon\vert^{1/2} \|V^\varepsilon;L^2(\Theta^\varepsilon)\| \leq
c_3 \varepsilon^2 (1+\vert\psi\vert) \|V^\varepsilon;{\cal H}^\varepsilon(\varepsilon\psi)\|. }
\label{(de6.1)}
\end{equation}
Similarly, we obtain
\begin{equation}\nonumber
\vert I^\varepsilon_5(V^\varepsilon;\psi)\vert\leq
c_4\varepsilon^2 (1+\vert\psi\vert)^2 \|V^\varepsilon;{\cal H}^\varepsilon(\varepsilon\psi)\|.
\label{(de6.2)}
\end{equation}

In a similar way to \eqref{(un6)}, using the Taylor formula for $U^{0}_\pm$ yields the inequality
\begin{equation}\label{(de8)}
\vert I^\varepsilon_6(V^\varepsilon;\psi)\vert  \leq
c_5\sum\limits_\pm\vert\Upsilon^\varepsilon_\pm\vert^{1/2}
\varepsilon \max\limits_{x\in\Upsilon^\varepsilon_\pm}\bigg\vert \frac{\partial^3 U_\pm^0}{\partial
x_1^3}(x_1;\psi)\bigg\vert\,\|V^\varepsilon;L^2(\Upsilon^\varepsilon_\pm)\|
\leq c_6\varepsilon^2\|V^\varepsilon;{\cal
H}^\varepsilon(\varepsilon\psi)\|.
\end{equation}

As regards $S^\varepsilon_{7,\pm}$ and $S^\varepsilon_{8,\pm}$ or equivalently, first we note that
\begin{multline*}
\textcolor{black}{\frac{1}{2}\sum\limits_{\tau=\pm}
\Big(U^{\prime}_\pm(\tau 0;\psi)+ x_1\frac{\partial U^{\prime}_\pm}{\partial x_1}(\tau 0;\psi)\Big)
= U^{\prime}_\pm(\sigma 0;\psi)+x_1\frac{\partial U^{\prime}_\pm}{\partial x_1}(\sigma 0;\psi)}
-\frac{\sigma}{2}[U^{\prime}_\pm]_0(\psi)
-\frac{\sigma}{2} x_1\Big[\frac{\partial U^{\prime}_\pm}{\partial x_1}\Big]_0(\psi), \\
\quad\mbox{ for }\,\sigma x_1>0,\quad \sigma\in\{-1,+1\}.
\end{multline*}
Besides,
$[\Delta_x,X^\varepsilon(x_1)]=
\textcolor{black}{-[\Delta_x,\chi_+(x_1/\varepsilon)]-[\Delta_x,\chi_-(x_1/\varepsilon)]}$, and  we get
\begin{equation*}
\begin{split}
\!S^\varepsilon_{7,\pm}(x;\psi) =& -\varepsilon\!\sum\limits_{\sigma=\pm}
[\Delta_x,\chi_\sigma(x_1/\varepsilon)]
\bigg(U^{\prime}_\pm(x_1;\psi)- U^{\prime}_\pm(\sigma 0;\psi)
-x_1 \frac{\partial U^{\prime}_\pm}{\partial x_1}(\sigma 0;\psi)\bigg)\\
&+\frac{1}{2} [U^{\prime}_\pm]_0(\psi)\varepsilon^{-1}\!\sum\limits_{\tau=\pm} \!
\tau\Delta_\xi\chi_\tau(\xi_1)
+\frac{1}{2}\Big[\frac{\partial U^{\prime}_\pm}{\partial x_1}\Big]_0(\psi)\!\sum\limits_{\tau=\pm} \!
\tau\Delta_\xi(\xi_1\chi_\tau(\xi_1)).
\end{split}\end{equation*}
The first term can be estimated and the others will be when they are joined into $S^\varepsilon_{8,\pm}$. Indeed, let us write
\begin{equation}\label{8-1}
S^\varepsilon_{7,\pm}(x;\psi)+ S^\varepsilon_{8,\pm}(x;\psi) =:\sum\limits_{j=1}^{4}T^\varepsilon_{j,\pm}(x;\psi),
\end{equation}
where
\begin{eqnarray*}
&&T^\varepsilon_{1,\pm}(x;\psi)= -\varepsilon\!\sum\limits_{\sigma=\pm}
[\Delta_x,\chi_\sigma(x_1/\varepsilon)]
\Big(\!U^{\prime}_\pm(x_1;\!\psi)- U^{\prime}_\pm(\sigma 0;\!\psi)
-x_1 \frac{\partial U^{\prime}_\pm}{\partial x_1}(\sigma 0;\!\psi)\!\Big),\\
&&T^\varepsilon_{2,\pm}(x;\psi)= \varepsilon^{-1}\chi_0(x_1)\Big(\Delta_\xi\widetilde{w}^{\,\prime}_\pm(\xi;\psi)+
\frac{1}{2}[U^{\prime}_\pm]_0(\psi)\sum\limits_{\tau=\pm}\tau
\Delta_\xi\chi_\tau(\xi_1)\Big),\\
&&T^\varepsilon_{3,\pm}(x;\psi)= \chi_0(x_1)\Big( \Delta_\xi\widehat{w}^{\,\prime\prime}_\pm(\xi;\psi)+
\frac{1}{2}\Big[\frac{\partial U^{\prime}_\pm}{\partial
x_1}\Big]_0(\psi)\sum\limits_{\tau=\pm} \tau\Delta_\xi
(\xi_1\chi_\tau(\xi_1))\Big),\\
&&T^\varepsilon_{4,\pm}(x;\psi)= \big(1-\chi_0(x_1)\big)
\Big(
\frac{1}{2\varepsilon}[U^{\prime}_\pm]_0(\psi)\sum\limits_{\tau=\pm}\tau
\Delta_\xi\chi_\tau(\xi_1)+ \frac{1}{2}\Big[\frac{\partial U^{\prime}_\pm}{\partial
x_1}\Big]_0(\psi)\sum\limits_{\tau=\pm} \tau\Delta_\xi
(\xi_1\chi_\tau(\xi_1))
\Big).
\end{eqnarray*}
Similarly to \eqref{(un6)} and \eqref{(de8)}, using the Taylor formula for {\color{black}$U_\pm^\prime$}
yields the inequality
\begin{equation}
\vert(T^\varepsilon_{1,\pm},V^\varepsilon)_{\varpi^\varepsilon}\vert\leq
c_7 \varepsilon\sum\limits_\pm\vert\Upsilon^\varepsilon_\pm\vert^{1/2}
\,\max\limits_{x\in\Upsilon^\varepsilon_\pm}\bigg\vert \frac{\partial^2 U_\pm^\prime}{\partial
x_1^2}(x_1;\psi)\bigg\vert\,\|V^\varepsilon;L^2(\Upsilon^\varepsilon_\pm)\|
\leq  c_7\varepsilon^2(1+\vert\psi\vert)\|V^\varepsilon;{\cal H}^\varepsilon(\varepsilon\psi)\|.  \label{8-2}
\end{equation}
Now, by formulas \eqref{(ra4)},  \eqref{(W100)}, \eqref{(K12)},  \eqref{(ra44)}, \eqref{(ra45)} and  \eqref{(K12N)},
and the fact that $\Delta_\xi W_0^{\, 1}=0$, $-\Delta_\xi W^3=1$ and $\frac{\partial U_\pm^\prime}{\partial x_1}(+0;\psi)=-\frac{\partial U_\pm^\prime}{\partial x_1}(-0;\psi)$ (cf.~\eqref{(S18)}), it follows that
$$
T^\varepsilon_{2,\pm}(x;\psi)= T^\varepsilon_{3,\pm}(x;\psi) =0.
$$
On the other hand, since the support of $\big(1-\chi_0(x_1)\big)$ is  contained in
$\{\vert x_1\vert{\color{black}\geq 1/6}\}$ (see \eqref{(ra3)}) and the support of the derivatives of $\chi_\pm$ is {\color{black}in} $\{\pm \xi_1\in[R,2R]\}$ (see \eqref{(chi)}), under the condition
$
\varepsilon<\frac{1}{12R},
$
we have that $T^\varepsilon_{4,\pm}(x;\psi)=0$. Thus,  by \eqref{8-1} and \eqref{8-2},
we get
$$
\vert I^\varepsilon_7(V^\varepsilon;\psi) + I^\varepsilon_8(V^\varepsilon;\psi)\vert\leq
c_7\varepsilon^2 {\color{black}(1+\vert\psi\vert)}\|V^\varepsilon;{\cal
H}^\varepsilon(\varepsilon\psi)\|.
$$

Now, we consider $S^\varepsilon_{9,\pm}$. In a similar way to \eqref{(un62)}, since the coefficients of the commutator $[\Delta_x,\chi_0]$ do not
depend on $\varepsilon$ and have their supports in the union of the rectangles  $\Upsilon^0_\pm$,
while $\widetilde{w}^{\,\prime}_\pm(\xi;\psi)$, $\nabla_\xi\widetilde{w}^{\,\prime}_\pm(\xi;\psi)$ and $\nabla_\xi\widetilde{w}^{\,\prime\prime}_\pm(\xi;\psi)$ are exponentially decaying functions and $\widetilde{w}^{\,\prime\prime}_\pm(\xi;\psi)$ is a bounded function (see \eqref{(ra4)}--\eqref{(ra45)}), we have
\begin{equation}\begin{split}
\vert I^\varepsilon_9(V^\varepsilon;\psi)\vert
&\leq  c\,\varepsilon^2\sup\limits_{\xi\in\Xi}(\vert\xi_1\vert\vert\widetilde{w}^{\,\prime}_\pm(\xi)\vert +
\vert\xi_1\vert^2 \vert\nabla_{\xi_1} \widetilde{w}^{\,\prime}_\pm(\xi)\vert \!+\! \vert\widehat{w}^{\,\prime\prime}_\pm(\xi)\vert \!+\! \vert\xi_1\vert \vert\nabla_{\xi_1} \widehat{w}^{\,\prime\prime}_\pm(\xi)\vert) \|V^\varepsilon;L^2(\varpi^\varepsilon)\| \vspace{0.1cm}\\
&\leq C\,
\varepsilon^2 (1+\vert\psi\vert) \textcolor{black}{\,\|V^\varepsilon;L^2(\varpi^\varepsilon)\|} . \nonumber
 \end{split}\label{(de9)}\end{equation}}

Finally, to estimate $I_{10}^\varepsilon(V^\varepsilon;\psi)$, we introduce the following lemma:
\begin{Lemma}
Let $\chi_1\in C^\infty({\mathbb R}), \,  \chi_1(x_1)=1\,\,\mbox{\rm
for}\,\,\vert x_1\vert\leq 1/3,\, \chi_1(x_1)=0\,\,\mbox{\rm for}\,\,\vert x_1\vert\geq 2/3$.
There is $\varepsilon_0>0$ such that, for $\varepsilon<\varepsilon_0$, the inequality
 \begin{equation}\begin{array}{c}\displaystyle
\|e^{-\sigma
\vert x_1\vert/\varepsilon}{\color{black}\chi_1}V^\varepsilon;L^2(\varpi^\varepsilon)\|\leq
c_\sigma\varepsilon^{1/2}\|V^\varepsilon;H^1(\varpi^\varepsilon)\|
 \end{array}\label{(MEP1)}\end{equation}
is valid for all $V^\varepsilon\in H^1(\varpi^\varepsilon)$ with any $\sigma>0$ and a factor $c_\sigma$ independent of
$\varepsilon$.
\end{Lemma}
\begin{proof}
Without loss of generality we assume that $V^\varepsilon$
is a real function. We {\color{black}consider} the extended function
$\widehat{V}^{\,\varepsilon}$ constructed in such a way that satisfies \eqref{(KK3)}. We
have
\begin{eqnarray*}
&&\!\!\int\limits_{-1/2}^{1/2}  \!\!e^{-2\sigma
\vert x_1\vert/\varepsilon}\vert\chi_1(x_1)\widehat{V}^\varepsilon(x_1,x_2)\vert^2dx_1
=\!\!\int\limits_{-1/2}^{1/2} e^{-2\sigma
\vert x_1\vert/\varepsilon}
\bigg\vert\int\limits_{x_1}^{1/2}\!\!\frac{\partial}{\partial
t}(\chi_1(t)\widehat{V}^\varepsilon(t,x_2))^2\,dt\bigg\vert dx_1\\
&&\qquad \leq C \!\!\int\limits_{-1/2}^{1/2} \!\!e^{-2\sigma
\vert x_1\vert/\varepsilon}dx_1
\int\limits_{x_1}^{1/2} \! \bigg(\Big\vert\frac{\partial
\widehat{V}^\varepsilon}{\partial
t}(t,x_2)\Big\vert^2+\vert\widehat{V}^\varepsilon(t,x_2)\vert^2\bigg)\,dt
\\
&&\qquad \leq \varepsilon C_\sigma  \!\!\int\limits_{x_1}^{1/2} \!\bigg(\Big\vert\frac{\partial
\widehat{V}^\varepsilon}{\partial
t}(t,x_2)\Big\vert^2+\vert\widehat{V}^\varepsilon(t,x_2)\vert^2\bigg)\,dt.
\end{eqnarray*}
Integrating the above formula in $x_2\in(0,H)$, cf.  \eqref{(KK3)}, we get \eqref{(MEP1)}.
\end{proof}

In addition, using the periodicity of  $\widetilde{w}^{\,\prime}_\pm(\xi;\psi)$ in $\xi_2$, cf.~\eqref{(33)} and \eqref{(ra4)},
we have
\begin{equation}\label{(de10a)}
\!\!\! \|e^{\frac{\sigma}{\varepsilon}\vert x_1\vert}\widetilde{w}^{\,\prime}_{\pm}\Big(\frac{x}{\varepsilon}; \psi\Big);L^2(\varpi^\varepsilon)\| \!\leq \!
c \varepsilon^{\frac{1}{2}} \|e^{\sigma \vert\xi_1\vert}\widetilde{w}_\pm^{\prime}(\xi;\psi);L^2(\Xi)\|, \,\,\,\sigma\!\in\!\Big(0,\frac{2\pi}{H}\Big).
\end{equation}
Thus, gathering \eqref{(MEP1)}, \eqref{(de10a)}, \eqref{(ra9)} and the boundedness of  $\widetilde{w}^{\,\prime\prime}_\pm(\xi;\psi)$, we
conclude that
\begin{equation}\hspace{-0.02cm}\begin{array}{ll}\displaystyle
 \!\!\vert I^\varepsilon_{10}(V^\varepsilon;\psi)\vert
\!\!\!\!\!&\leq \! \varepsilon \vert4\pi^2\!+\!\varepsilon \Lambda^\prime_\pm(\psi)\vert
\Big(\|e^{\frac{\sigma}{\varepsilon}\vert x_1\vert}\widetilde{w}^{\,\prime}_{\pm};L^2(\varpi^\varepsilon)\|\,\|e^{\frac{-\sigma}{\varepsilon}\vert x_1\vert}
V^\varepsilon;L^2(\varpi^\varepsilon)\|
+   c\varepsilon \sup\limits_{\xi\in\Xi} \vert\widehat{w}^{\,\prime\prime}_{\pm}\vert  \|V^\varepsilon;L^2(\varpi^\varepsilon)\|\Big) \vspace{0.15cm}\\
&\leq\! c\varepsilon(1\!+\!\varepsilon(1+\vert\psi\vert))\left(\varepsilon^{1/2}\|e^{\sigma\vert\xi_1\vert}\widetilde{w}_\pm^{\prime};L^2(\Xi)\|\, \varepsilon^{1/2}\|V^\varepsilon;H^1(\varpi^\varepsilon)\| +
c\varepsilon (1\!+\!\vert\psi\vert)\|V^\varepsilon;L^2(\varpi^\varepsilon)\|\right)
\vspace{0.15cm}\\
&\leq \!
c_{10}\,\varepsilon^2(1\!+\!\vert\psi\vert)^2\,\|V^\varepsilon;H^1(\varpi^\varepsilon)\|.
\end{array}\label{(de10)}\end{equation}

Also, fixed $\psi_0>0$, by definition of $U_\pm^\varepsilon$ (see \eqref{(ra2)}--\eqref{(ra6)}), it can be proved that
\begin{equation}\label{(un64bis)}
\|U_\pm^\varepsilon;{\cal H}^\varepsilon(\varepsilon\psi)\|^2 \lie \|U_\pm^0;L^2(\varpi^0)\|^2+ \|\nabla_x U_\pm^0;L^2(\varpi^0)\|^2=(1+4\pi^2)H
\end{equation}
for }$\vert\psi\vert\leq \psi_0$. Finally, on account of  \eqref{(de2)},  \eqref{(un5bis)}, the estimates \eqref{(de4)}--\eqref{(de10)},  and the convergence \eqref{(un64bis)}, we arrive at   \eqref{deltacota_bis}.

\medskip

We are ready to apply Lemma~\ref{Lemma_Visik} ending the proof of Theorem~\ref{Theorem_multiple1}.

For any fixed $\psi_0>0$, we consider  \eqref{(de2)} and \eqref{deltacota_bis}.
Lemma~\ref{Lemma_Visik} gives eigenvalues
$M^\varepsilon_\pm(\varepsilon\psi)$ of the operator ${\cal
B}^\varepsilon(\eta)$ admitting the estimates
\begin{equation}\begin{array}{c}\displaystyle
\vert M^\varepsilon_\pm(\varepsilon\psi)-{\color{black}M^{1,\varepsilon}_\pm(\psi)}\vert\leq {\color{black} c(\psi_0)\varepsilon^2}
\end{array}\label{(de11)}\end{equation}
where {\color{black}$c(\psi_0)$ is independent of $\varepsilon$.}
Similarly to \eqref{(un8)}--\eqref{(un10)} we derive from \eqref{(de11)} that
under the restriction
$\varepsilon\leq\varepsilon(\psi_0),$
the corresponding eigenvalues
$\Lambda^\varepsilon_\pm(\varepsilon\psi)$   in the sequence
\eqref{(17)} satisfy the relations
\begin{equation}\label{EstEn0}
\begin{array}{c}\displaystyle
\vert\Lambda^\varepsilon_+(\varepsilon\psi)-4\pi^2-\varepsilon{\color{black}\Lambda_+^{\prime}}(\psi)\vert\leq
{\color{black}C(\psi_0)\varepsilon^2},\\\\
\vert\Lambda^\varepsilon_-(\varepsilon\psi)-4\pi^2-\varepsilon{\color{black}\Lambda_-^{\prime}}(\psi)\vert\leq
{\color{black}C(\psi_0)\varepsilon^2,}
\end{array}
\end{equation}
where $\Lambda_\pm^{\prime}(\psi)$ are given by \eqref{(K17)}.

Now, to identify $\Lambda^\varepsilon_-(\varepsilon\psi)$
and $\Lambda^\varepsilon_+(\varepsilon\psi)$  we use that
$$
\Lambda_+^{\prime}(\psi)-{\color{black}\Lambda_-^{\prime}}(\psi)=
8\pi\sqrt{4\pi^2{\color{black}\Big(m_1(\Xi)+\dfrac{\vert\omega\vert}{2H}\Big)^{\!2}}+\psi^2}
$$
and hence $\Lambda_-^\varepsilon (\varepsilon \psi) < \Lambda_+^\varepsilon (\varepsilon \psi)$ for  $\vert\psi\vert<\psi_0$.
Besides, from  \eqref{EstEn0} and \eqref{(K17)}, we can check that $\Lambda_+^\varepsilon (\varepsilon \psi)<4\pi^2+K_4$ for $\vert\psi\vert<\psi_0$ and $\varepsilon>0$ small enough, and consequently, by \eqref{EstL4}, $\Lambda_+^\varepsilon (\varepsilon \psi)\leq \Lambda_3^\varepsilon (\varepsilon \psi)$ under the assumption  $H\in(0,1/2).$ Finally, since $\Lambda_1^\varepsilon(0)=0\neq \Lambda_-^\varepsilon(0)$, we can identify $\Lambda^\varepsilon_-(\varepsilon\psi)=\Lambda^\varepsilon_2(\varepsilon\psi)$ and  $\Lambda^\varepsilon_+(\varepsilon\psi)=\Lambda^\varepsilon_3(\varepsilon\psi)$ for  $\vert\psi\vert\leq\psi_0$. This ends the proof of Theorem~\ref{Theorem_multiple1}.

\subsection{The node $(\eta_{\mbox{\tiny$\square$}}, \Lambda_{\mbox{\tiny$\square$}})= (\pm\pi,\pi^2)$ for $H\in(0,1)$}\label{subsec64}

Following the scheme in Sections~\ref{subsec61}--\ref{subsec63} for the node $(\eta_{\mbox{\normalsize$\circ$}}, \Lambda_{\mbox{\normalsize$\circ$}})= (0,4\pi^2)$, we consider the node $(\eta_{\mbox{\tiny$\square$}}, \Lambda_{\mbox{\tiny$\square$}})= (\pm\pi,\pi^2)$ under the assumption $H\in(0,1)$; cf.~Figure~\ref{fig2} a) and b).
For the sake of brevity, here we only outline the main changes.

Thanks to the $2\pi$-periodicity in $\eta$, we consider the node $(\eta_{\mbox{\tiny$\square$}}, \Lambda_{\mbox{\tiny$\square$}})= (\pm\pi,\pi^2)$ as the
intersection point of the dispersion curves
\begin{equation}\nonumber\begin{array}{c}\displaystyle
\Lambda=\eta^2\,\,\mbox{\rm and}\,\,\Lambda=(2\pi-\eta)^2\,\,\mbox{\rm with}\,\,\eta\in[0,2\pi].
\end{array}\label{KK1}\end{equation}
In other words, we extend by periodicity  the truss in Figure~\ref{fig2} a) as it is
depicted in  {\color{black}Figure~\ref{fig5}} a). Correspondingly, the dispersion curves in
{\color{black}Figure~\ref{fig3}} a) are extended periodically as well, cf.,  {\color{black}Figure~\ref{fig5}} b).

\begin{figure}
\begin{center}
\resizebox{!}{2.9cm} {\includegraphics{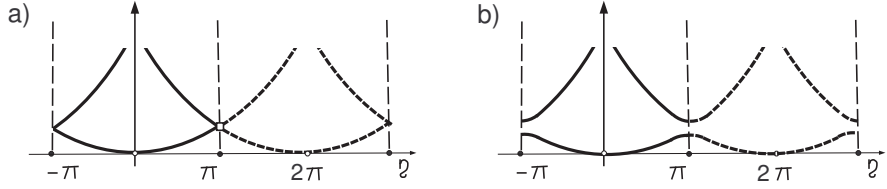}}
\caption{Duplication of the dispersion curves, limit {\bf a)} and
perturbed {\bf b)}.}
\label{fig5}
\end{center}
\end{figure}

\medskip

Let us list the changes with respect to Section~\ref{subsec61} which are necessary to support
the asymptotic ans\"{a}tze \eqref{(K3)} and \eqref{(K4)},
\eqref{(K6)} for the eigenpairs $\{\Lambda^\varepsilon_p(\eta),
U^\varepsilon_p(x;\psi)\}$, $p=1,2$, of the problem
\eqref{(11)}--\eqref{(14)} with the fast Floquet variable
\begin{equation}\begin{array}{c}\displaystyle
\psi=\varepsilon^{-1}(\eta-\pi)
\end{array}\label{(N4)}\end{equation}
instead of \eqref{(KK0)}.

To the eigenvalue $\Lambda^0:=\!\Lambda_1^0(\pi)\!=\!\Lambda_2^0(\pi)\!=\!\pi^2$ of the problem  \eqref{(29)}--\eqref{(25)}, there
corresponds the eigenfunctions
$
U^0_\pm(x)=e^{\pm i\pi x_1}.
$
Now, the main term in the outer expansion \eqref{(K4)} becomes the linear combination of these eigenfunctions
\begin{equation}\nonumber\begin{array}{c}\displaystyle
U^0(x_1;\psi)=a_+(\psi)e^{+\pi ix_1}+a_-(\psi)e^{-\pi ix_1}.
\end{array}\label{(N7)}\end{equation}
Notice that again no dependence on $x_2$ occurs.
The main term in the inner expansion \eqref{(K6)} keeps the form \eqref{(K7)} but the correction terms look as follows:
\begin{equation}\nonumber\begin{array}{c}\displaystyle
w^\prime(\xi;\psi)=\pi i
(a_+(\psi)-a_-(\psi))W^1(\xi)+a^\prime(\psi)W^0
\end{array}\label{(N12)}\end{equation}
and
\begin{equation}\nonumber\begin{array}{c}\displaystyle
w^{\prime\prime}(\xi;\psi)=\pi^2(a_+(\psi)+a_-(\psi))W^3(\xi)+a^{\prime\prime}(\psi)W^0+ \widetilde{w}^{\,\prime\prime}(\xi;\psi).
\end{array}\end{equation}
Similarly to \eqref{(K12)}, \eqref{(K12N)}, the jump conditions  now read
\begin{equation}\begin{array}{c}\displaystyle
[U^\prime]_0(\psi)=2\pi i(a_+(\psi)-a_-(\psi)) m_1(\Xi),\quad
x_2\in(0,H),
\\\\\displaystyle
\Big[\frac{\partial U^\prime}{\partial x_1}\Big]_0(\psi)= \pi^2
 (a_+(\psi)+a_-(\psi))\frac{\vert\omega\vert}{H},\quad x_2\in(0,H).
\end{array}\label{(N14)}\end{equation}
Moreover, instead of \eqref{(K99)}, we have
\begin{equation}\nonumber\begin{array}{c}
e^{i\eta}=e^{i(\pi+\varepsilon\psi)}=e^{i\pi}(1+i \varepsilon\psi+O(\varepsilon^2))=
-1-i\varepsilon\psi+O(\varepsilon^2),
\end{array}\label{(N99)}\end{equation}
so that the somehow quasi-periodicity conditions of the type \eqref{(K14)} turn into
\begin{equation}\begin{array}{l}\displaystyle
U^\prime\Big(\frac{1}{2},x_2;\psi\Big)+U^\prime
\Big(-\frac{1}{2},x_2;\psi\Big)=-i\psi U^0\Big(-\frac{1}{2},x_2;\psi\Big)=-\psi(a_+(\psi)-a_-(\psi)), \vspace{0.3cm}\\
\displaystyle
\frac{\partial U^\prime}{\partial
x_1}\Big(\frac{1}{2},x_2;\psi\Big)+ \frac{\partial
U^\prime}{\partial x_1}\!\Big(-\frac{1}{2},x_2;\psi\Big)
=-i\psi\frac{\partial U^0}{\partial
x_1} \Big(-\frac{1}{2},x_2;\psi\Big)
=-i\pi\psi\big(a_+(\psi)\!+\!a_-(\psi)\big).
\end{array}\label{(N16)}\end{equation}
It is worth mentioning that the relations \eqref{(K14)} are nothing
but inhomogeneous pure periodicity conditions while the
relations \eqref{(N16)} imply inhomogeneous anti-periodicity conditions of
the function $U^\prime$.

The problem  \eqref{(K13)}, \eqref{(N14)}, \eqref{(N16)} with $\Lambda^0=\pi^2$
has two compatibility conditions which can be obtained by inserting the data
of \eqref{(N16)} and \eqref{(N14)} into the Green formula as follows:
\begin{equation*}
\begin{split}
\Lambda^\prime(\psi)Ha_\pm(\psi)&=-\int\limits_{\varpi^0}
\overline{e^{\pm\pi ix_1}}\left(\Delta U^\prime(x;\psi)+\Lambda^0 U^\prime(x;\psi)\right)dx\\
&=-\int\limits_0^H e^{\mp \pi ix_1} \Big( \frac{\partial
U^\prime}{\partial x_1}(x;\psi)\pm \pi
iU^\prime(x;\psi)\Big)\bigg\vert_{-\frac{1}{2}}^{x_1=\frac{1}{2}} dx_2
+\int\limits_0^H\Big[ \frac{\partial
U^\prime}{\partial x_1}(x;\psi)\pm \pi
iU^\prime(x;\psi)\Big]_0 \!dx_2.
\end{split}\end{equation*}
They convert into the system of two algebraic equations
\begin{equation}\begin{array}{l}\displaystyle
\Lambda^\prime(\psi)a_+(\psi)= \left(\!\pi^2\frac{\vert\omega\vert}{H} - 2\pi^2m_1(\Xi)+2\pi\psi\right)a_+(\psi)
+ \left(\pi^2\frac{\vert\omega\vert}{H}+ 2\pi^2m_1(\Xi)\right)a_-(\psi),\vspace{0.3cm}\\
\displaystyle
\Lambda^\prime(\psi)a_-(\psi)=\left(\pi^2\frac{\vert\omega\vert}{H}+ 2\pi^2m_1(\Xi)\right)a_+(\psi)+
\left(\pi^2\frac{\vert\omega\vert}{H} - 2\pi^2m_1(\Xi)-2\pi\psi\right)a_-(\psi),
\end{array}\label{(N17)}\end{equation}
with the eigenvalues
\begin{equation}\begin{array}{c}\displaystyle
\textcolor{black}{
\Lambda^{\prime}_\pm}(\psi)=2\pi\bigg(\textcolor{black}{\pi\Big(
\frac{\vert\omega\vert}{2H}- m_1(\Xi)\Big)}\pm \sqrt{\pi^2 \Big(
m_1(\Xi)\textcolor{black}{+}\frac{\vert\omega\vert}{2H}\Big)^2+\psi^2}\bigg),
\end{array}\label{(N18)}\end{equation}
where
\begin{equation}
\Lambda^{\prime}_-(\psi)\leq -4\pi^2m_1(\Xi)
\,\,\hbox{ and }\,\, \Lambda^{\prime}_+(\psi)\geq 2\pi^2\frac{\vert\omega\vert}{H}.
\label{(N18n)}\end{equation}
The corresponding $a^\pm(\psi)=(a^\pm_+(\psi),a^\pm_-(\psi))$ can be easily computed from the algebraic equations \eqref{(N17)}.
We again have $\Lambda^\prime_+(\psi)>
\Lambda^\prime_-(\psi)$ and therefore, we establish the relation of the eigenpairs $\{\Lambda^\prime_-(\psi),a^-(\psi)\}$
and $\{\Lambda^\prime_+(\psi),a^+(\psi)\}$, respectively, with  the eigenpairs
$\{\Lambda^\varepsilon_1(\eta),U^\varepsilon_1(x;\psi)\}$ and
$\{\Lambda^\varepsilon_2(\eta),U^\varepsilon_2(x;\psi)\}$ of the problem
\eqref{(11)}--\eqref{(14)} with $\eta$ defined by \eqref{(N4)}.

Now, we formulate our result on splitting edges of the first and  second limit spectral bands
giving rise to the open gap $\gamma_1^\varepsilon$ (cf.~Figure~\ref{fig3} a) and b)); here, we take into account the $2\pi-$periodicity in $\eta$ of the functions $\Lambda_p^\varepsilon(\eta)$.

\begin{Theorem}\label{Theorem_multiple2}
{\color{black}Let $H\in(0,1)$ and $\psi_1>0$. Then, } there exist positive $\varepsilon_0=\varepsilon_0(H,\psi_1)$ and
$C=C(H,\psi_1)$ such that, for $\varepsilon\in(0,\varepsilon_0]$, the
entries $\Lambda^\varepsilon_1(\eta)$ and
$\Lambda^\varepsilon_2(\eta)$ of the eigenvalue sequence
\eqref{(17)} with $\eta=\pi+\varepsilon\psi$, {\color{black}$\vert\psi\vert\leq\psi_1$,} meet the estimates
\begin{equation}\nonumber\begin{array}{c}\displaystyle
\vert\Lambda^\varepsilon_2({\color{black}\pi+}\varepsilon\psi)-\pi^2-\varepsilon{\color{black}\Lambda_+^{\prime}}(\psi)\vert\leq
{\color{black}C\varepsilon^2}, \vspace{0.2cm}\\
\vert\Lambda^\varepsilon_1({\color{black}\pi+}\varepsilon\psi)-\pi^2-\varepsilon{\color{black}\Lambda_-^{\prime}}(\psi)\vert\leq
{\color{black}C\varepsilon^2},
\end{array}\label{(MEP12)}\end{equation}
where the quantities ${\color{black}\Lambda_\pm^{\prime}}(\psi)$ are given by
\eqref{(N18)}.
\end{Theorem}

\section{Opening the spectral gaps}\label{sec7}

In this section, we show that, under the mirror symmetry condition of the holes, cf. \eqref{(symm)},   there are open spectral gaps for the spectrum \eqref{(8)} of the original problem \eqref{(4)}--\eqref{(5)} in the perforated waveguide $\Pi^\ee$, cf. \eqref{(3)}; see also Figures~\ref{fig1} and \ref{fig-mirror}. Further specifying, for the values $H\in  (0,1)$  we show that there is at least one open gap while for $H\in (0,1/2)$ there are at least two open gaps. We provide asymptotic formulas for their localization and width, cf. Figure~\ref{fig3} b) and a) respectively and formulas \eqref{(L1)}--\eqref{(MEP94)},  \eqref{(LL1)} and \eqref{(MEP99)}. In Sections~\ref{subsec71} and \ref{subsec72}, respectively, we broach the cases where $H\in (0,1)$   and  $H\in (0,1/2)$.

\subsection{Opening spectral gap near the  node $(\eta_{\mbox{\tiny$\square$}}, \Lambda_{\mbox{\tiny$\square$}})$}\label{subsec71}

Recall $(\eta_{\mbox{\tiny$\square$}}, \Lambda_{\mbox{\tiny$\square$}})=(\pm\pi,\pi^2)$. Based on asymptotic formulas in Theorems~\ref{Th5.1} and \ref{Theorem_multiple2},
we prove in this section that
\begin{equation}\begin{array}{c}\displaystyle
\max\limits_{\eta\in[-\pi,\pi]}\Lambda^\varepsilon_1(\eta)\leq \pi^2-4\pi^2\varepsilon
m_1(\Xi) +O(\varepsilon^2),
\vspace{0.2cm}\\\displaystyle
\min_{\eta\in[-\pi,\pi]}\Lambda^\varepsilon_2(\eta)\geq \pi^2+2\pi^2\varepsilon
\frac{\vert\omega\vert}{H}+O(\varepsilon^2).
\end{array}\label{(L1)}\end{equation}
In this way, since $m_1(\Xi)\pm (2H)^{-1}\vert\omega\vert>0$ (see Proposition~\ref{Proposition_W1}), the spectral gap
\begin{equation}\begin{array}{c}\displaystyle
\gamma^\varepsilon_p=(\max_\eta\Lambda^\varepsilon_p(\eta),\min_\eta\Lambda^\varepsilon_{p+1}(\eta))
\end{array}\label{(L2)}\end{equation}
with $p=1$ stays open and has the width
\begin{equation}\begin{array}{c}\displaystyle
\vert\gamma^\varepsilon_1\vert\geq 4\pi^2\varepsilon
\Big(m_1(\Xi)+\frac{\vert\omega\vert}{2H}\Big)+O(\varepsilon^2).
\end{array}\label{(MEP94)}\end{equation}

Let us prove \eqref{(L1)} for $H\in(0,1)$.
We divide the proof in two parts depending on whether $\eta\in I_1$ or $\eta\in I_2$ where the sets $I_1=[-\pi+\delta_1, \pi-\delta_1]$  and $I_2=[-\pi,-\pi+\delta_1]\cup[\pi-\delta_1,\pi]$ for certain  $\delta_1\in(0,\pi)$, cf. Figure~\ref{fig_cajas}. For simplicity, we choose $\delta_1$ such that  $\Lambda_-^0(\pi-\delta_1)=(\pi+\delta_1)^2<\pi^2+K_2$ where $K_2$ is defined by \eqref{def_K1yK2}. Thus, by Proposition~\ref{PropoR12}, we have that there exists $\varepsilon_1=\varepsilon(H,\delta_1)>0$ such that
\begin{eqnarray}
\Lambda_2^\varepsilon(\eta)>\pi^2+K_1\qquad &&\mbox{for } \eta\in I_1, \, \varepsilon<\varepsilon_1,  \label{EstL2bis} \\
\Lambda_3^\varepsilon(\eta)>\pi^2+K_2\qquad &&\mbox{for } \eta\in I_2, \, \varepsilon<\varepsilon_1,   \label{EstL3bis}
\end{eqnarray}
where $K_1$ and  $K_2$ are defined by \eqref{def_K1yK2} and $K_1$ may depend on $\delta_1$. In addition, when $\eta\in I_2$, we separate again into two parts $\eta\in I_2\cap\{\eta\, : \, \pi-\vert\eta\vert\leq \varepsilon\psi_1\}$ and $\eta\in I_2\cap\{\eta\, : \, \pi-\vert\eta\vert\geq \varepsilon\psi_1\}$ for a certain constant $\psi_1>0$ that we will determine below.

\begin{figure}[ht]
\begin{center}
\resizebox{!}{5cm} {\includegraphics{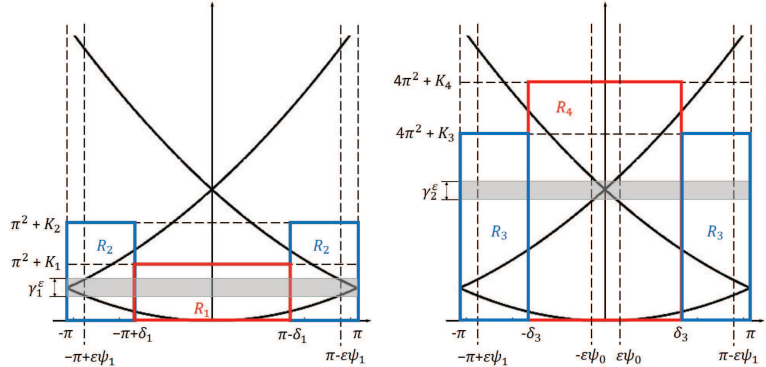}}
\caption{The different boxes $R_p$ for $p=1,2,3,4.$}
\label{fig_cajas}
\end{center}
\end{figure}

Firstly, we estimate $\Lambda_1^\varepsilon(\eta)$ and $\Lambda_2^\varepsilon(\eta)$ for $\eta\in I_1$ where \eqref{EstL2bis} holds, namely, the case where, for $\varepsilon$ small enough, there cannot be more than one eigenvalue $\Lambda_p^\varepsilon(\eta)$ in the box $R_1:=I_1\times [0,\pi^2+K_1].$ Thus, it is evident that
\begin{equation*}
\Lambda^\varepsilon_2(\eta)\geq \pi^2+2\pi^2\varepsilon \frac{\vert\omega\vert}{H}+O(\varepsilon^2)
\quad \mbox{for } \eta\in I_1.
\end{equation*}
Besides,  by Corollary~\ref{Th5.2}, we have
$$
\Lambda^\varepsilon_1(\eta) \leq \Lambda_1^0(\eta)+C_0\varepsilon\leq  (\pi-\delta_1)^2+ C_0\varepsilon \leq  \pi^2-4 \pi^2 m_1(\Xi)\varepsilon
\quad \mbox{for } \eta\in I_1
$$
and $\varepsilon$ small enough, which concludes the proof in $I_1.$

Secondly, we estimate $\Lambda_1^\varepsilon(\eta)$ and $\Lambda_2^\varepsilon(\eta)$ for $\eta\in I_2$ where \eqref{EstL3bis} holds, namely, the case where, for $\varepsilon$ small enough,  there cannot be more than two eigenvalues $\Lambda_p^\varepsilon(\eta)$ in the boxes $R_2:=I_2\times [0,\pi^2+K_2].$
Now, for any $\psi_1>0$, Theorem~\ref{Theorem_multiple2} and \eqref{(N18n)} allow us to obtain, for $\varepsilon$ small enough, the
extremum in \eqref{(L1)} restricted to $\eta\in I_2\cap\{\eta\, : \, \pi-\vert\eta\vert\leq \varepsilon\psi_1\}$.
Moreover, for $C_0$ the constant arising in \eqref{VisikL1} and \eqref{VisikLpm},  fixing $$\psi_1>C_0/2\pi,$$ we observe that the eigenvalues $\Lambda_\star^\varepsilon(\eta), \, \Lambda_-^\varepsilon(\eta)$ defined by Theorem~\ref{Th5.1} satisfy
$$\Lambda_-^\varepsilon(\eta)-\Lambda_\star^\varepsilon(\eta)\!\geq\!
\Lambda_-^0(\eta)-\Lambda_1^0(\eta)-2C_0\varepsilon\!=\!4\pi(\pi-\eta)-2C_0\varepsilon>0\,\,$$
for $\eta>0, \pi-\eta\geq \varepsilon \psi_1,$ and
$\Lambda_-^\varepsilon(\eta)\leq\Lambda_-^0(\eta)+C_0\varepsilon\leq\Lambda_-^0(\pi-\delta_1)+C_0\varepsilon\leq \pi^2+K_2$ for $\eta\in[\pi-\delta_1,\pi]$ and $\varepsilon$ small enough.
As a consequence, we can identify $\Lambda_1^\varepsilon(\eta)=\Lambda_\star^\varepsilon(\eta)$ and $\Lambda_2^\varepsilon(\eta)=\Lambda_-^\varepsilon(\eta)$ for $\eta\in[\pi-\delta_1,\pi-\varepsilon \psi_1]$, cf.~\eqref{EstL3bis}.
Thus, using Theorem~\ref{Th5.1} and taking
$$\psi_1=\max\left\{\frac{4\pi^2 m_1(\Xi)+C_0}{\pi}, \frac{C_0H+2\pi^2\vert\omega\vert}{2\pi H}  \right\},$$
for $\varepsilon$ small enough, we have
\begin{equation*}
  \begin{split}
    \Lambda^\varepsilon_1(\eta)&\leq \Lambda_1^0(\eta)+C_0\varepsilon = \pi^2-(\pi+\eta)(\pi-\eta)+C_0\varepsilon  \vspace{0.1cm}\\
       & \leq \pi^2 -\pi(\pi-\eta)+C_0\varepsilon  \leq \pi^2-4\pi^2 m_1(\Xi)\varepsilon \qquad \mbox{for } \eta\in [\pi-\delta_1, \pi-\varepsilon \psi_1], \vspace{0.2cm} \\
   \Lambda^\varepsilon_2(\eta) & \geq \Lambda_-^0(\eta)-C_0\varepsilon =  \pi^2 +(3\pi-\eta)(\pi-\eta)-C_0\varepsilon  \\
       &\geq \pi^2+2\pi(\pi-\eta) -C_0\varepsilon  \geq \pi^2+2\pi^2\frac{\vert\omega\vert}{H}\varepsilon \qquad \mbox{for } \eta\in [\pi-\delta_1, \pi-\varepsilon \psi_1].
   \end{split}
\end{equation*}
In a similar way, we can estimate $\Lambda_1^\varepsilon(\eta)$ and  $\Lambda_2^\varepsilon(\eta)$ for $\eta\in[-\pi+\varepsilon\psi_1, -\pi+\delta_1],$ where now $\Lambda_1^\varepsilon(\eta)=\Lambda_\star^\varepsilon(\eta)$ and $\Lambda_2^\varepsilon(\eta)=\Lambda_+^\varepsilon(\eta)$. This concludes the proof for $\eta\in I_2.$

\medskip

Now we formulate our result on opening spectral gap $\gamma^\varepsilon_1$ (see Figure~\ref{fig3} a)--b)):

\begin{Theorem}
\label{Corollary_dos}
Let $H\in(0,1)$.  Then,  there exists a positive constant $\varepsilon_0=\varepsilon_0(H)$ such that, for $\varepsilon\in(0,\varepsilon_0]$, the asymptotic formulas \eqref{(L1)}  are valid and the gap \eqref{(L2)} with $p=1$  has positive length \eqref{(MEP94)}.
\end{Theorem}

\subsection{Opening spectral gap near the node
$(\eta_{\mbox{\small{$\circ$}}},\Lambda_{\mbox{\small{$\circ$}}})$} \label{subsec72}

Recall $(\eta_{\mbox{\small{$\circ$}}},\Lambda_{\mbox{\small{$\circ$}}})=(0,4\pi^2)$. Similar computations on the base of Theorems~\ref{Th5.1}, \ref{Theorem_multiple1} and \ref{Theorem_multiple2}
prove that
\begin{equation}\begin{array}{c}\displaystyle
\max_{\eta\in[-\pi,\pi]}\Lambda^\varepsilon_2(\eta)\leq4\pi^2-16\pi^2\varepsilon
m_1(\Xi)+O(\varepsilon^2) ,
\\\displaystyle
\min_{\eta\in[-\pi,\pi]}\Lambda^\varepsilon_3(\eta)\geq4\pi^2+8\pi^2\varepsilon
\frac{\vert\omega\vert}{H}+O(\varepsilon^2),
\end{array}\label{(LL1)}\end{equation}
so that the gap \eqref{(L2)} with $p=2$ opens and gets  the width
\begin{equation}\begin{array}{c}\displaystyle
\vert\gamma^\varepsilon_2\vert\geq 16\pi^2\varepsilon
\Big(m_1(\Xi)+\frac{\vert\omega\vert}{2H}\Big)+O(\varepsilon^2).
\end{array}\label{(MEP99)}\end{equation}

Let us prove \eqref{(LL1)} for $H\in(0,1/2)$.
Now, we divide the proof in two parts depending on whether $\eta\in I_3$ or $\eta\in I_4$ where the sets $I_3=[-\pi,-\delta_3]\cup[\delta_3,\pi]$ and
$I_4=[-\delta_3, \delta_3]$ for certain $\delta_3\in (0,\pi)$, cf. Figure~\ref{fig_cajas}. For simplicity, we choose  $\delta_3$ such that $\Lambda_+^0(\delta_3)=(2\pi+\delta_3)^2<4\pi^2+K_4$ where $K_4$ is  defined by \eqref{def_K3yK4}. Thus, by Proposition~\ref{PropoR34} we have that there  exists $\varepsilon_1=\varepsilon(H,\delta_3)>0$ such that
\begin{eqnarray}
\Lambda_3^\varepsilon(\eta)>4\pi^2+K_3\qquad &&\mbox{for } \eta\in I_3, \, \varepsilon<\varepsilon_1,  \label{EstL3bbis}\\
\Lambda_4^\varepsilon(\eta)>4\pi^2+K_4\qquad &&\mbox{for } \eta\in  I_4, \, \varepsilon<\varepsilon_1,  \label{EstL4bis}
\end{eqnarray}
where $K_3$ and $K_4$ are  defined by \eqref{def_K3yK4} and $K_3$ may depend on $\delta_3$. In addition, when $\eta\in I_3$ or $\eta\in I_4$,  we separate again into two parts, namely,
we distinguish the four cases $\eta\in I_3\cap\{\eta\, : \, \pi-\vert\eta\vert\leq \varepsilon\psi_1\}$, $\eta\in I_3\cap\{\eta\, : \, \pi-\vert\eta\vert\geq \varepsilon\psi_1\}$,
$\eta\in [-\varepsilon \psi_0, \varepsilon \psi_0]\subset I_4$ and  $\eta\in I_4\cap\{\eta\, : \, \vert\eta\vert\geq \varepsilon\psi_0\}$ for a certain $\psi_0, \psi_1>0$.

Firstly, we estimate $\Lambda_2^\varepsilon(\eta)$ and $\Lambda_3^\varepsilon(\eta)$ for $\eta\in I_3$ where \eqref{EstL3bbis} holds, namely, the case where, for $\varepsilon$ small enough, there cannot be more than two eigenvalues $\Lambda_p^\varepsilon(\eta)$ in the boxes $R_3:=I_3\times [0,4\pi^2+K_3].$ Thus, it is evident that
\begin{equation*}
\Lambda^\varepsilon_3(\eta)\geq4\pi^2+8\pi^2\varepsilon \frac{\vert\omega\vert}{H}+O(\varepsilon^2)
\quad \mbox{for } \eta\in I_3=[-\pi,-\delta_3]\cup[\delta_3,\pi].
\end{equation*}
Besides, for any $\psi_1>0$, by virtue of Theorem~\ref{Theorem_multiple2} and \eqref{(N18)}, we get that
$$\Lambda_2^\varepsilon(\eta)\leq \pi^2+K(\psi_1)\varepsilon <2\pi^2
\qquad \mbox{for }\eta\in I_3\cap\{\eta\, : \, \pi-\vert\eta\vert\leq \varepsilon\psi_1\}$$
and $\varepsilon$ small enough.
Now, fixing $\psi_1>C_0/2$ and repeating the arguments  in the previous Section~\ref{subsec71} related with the set $I_2$, we can identify $\Lambda_2^\varepsilon(\eta)=\Lambda_-^\varepsilon(\eta)$ for $\eta\in[\delta_3,\pi-\varepsilon \psi_1]$ and $\Lambda_2^\varepsilon(\eta)=\Lambda_+^\varepsilon(\eta)$ for $\eta\in[-\pi+\varepsilon \psi_1, -\delta_3].$
Thus, by virtue of Theorem~\ref{Th5.1}, we can  check that
\begin{equation*}
\Lambda^\varepsilon_2(\eta)\leq 4\pi^2-16\pi^2\varepsilon m_1(\Xi)+O(\varepsilon^2)
\quad \mbox{for } \eta\in I_3\cap\{\eta\, : \, \pi-\vert\eta\vert\geq \varepsilon\psi_1\}.
\end{equation*}
and $\varepsilon$ small enough. This concludes the proof on the interval $I_3.$

Secondly, we estimate $\Lambda_2^\varepsilon(\eta)$ and $\Lambda_3^\varepsilon(\eta)$ when $\eta\in I_4$ where \eqref{EstL4bis} holds, namely, the case where, for $\varepsilon$ small enough,  there cannot be more than three eigenvalues $\Lambda_p^\varepsilon(\eta)$ in the box $R_4:=I_4\times [0,4\pi^2+K_4].$
Now, for any $\psi_0>0$, Theorem~\ref{Theorem_multiple1} and \eqref{(K17n)} allow us to obtain, for $\varepsilon$ small enough,
the extremum in \eqref{(LL1)} restricted to $\{\eta=\varepsilon\psi \, : \, \vert\psi\vert\leq \psi_0\}.$
Moreover, fixing $\psi_0>C_0/4\pi$, we observe that the eigenvalues $\Lambda_\pm^\varepsilon(\eta)$ defined by Theorem~\ref{Th5.1} satisfy
$$\Lambda_+^\varepsilon(\eta)-\Lambda_-^\varepsilon(\eta)\geq\Lambda_+^0(\eta)-\Lambda_-^0(\eta)-2C_0\varepsilon=8\pi\eta-2C_0\varepsilon>0\quad \mbox{for }\eta\geq\varepsilon \psi_0,$$
and
$\Lambda_+^\varepsilon(\eta)\leq\Lambda_+^0(\eta)+C_0\varepsilon\leq\Lambda_+^0(\delta_3)+C_0\varepsilon\leq 4\pi^2 +K_4$, for $\eta\in[0,\delta_3]$ and $\varepsilon$ small enough.
As a consequence, we can identify $\Lambda_2^\varepsilon(\eta)=\Lambda_-^\varepsilon(\eta)$ and $\Lambda_3^\varepsilon(\eta)=\Lambda_+^\varepsilon(\eta)$ for
$\eta\in[\varepsilon \psi_0,\delta_3].$ Note that, by Corollary~\ref{Th5.2}, $\Lambda_1^\varepsilon(\eta)=\Lambda_\star^\varepsilon(\eta)$ for $\eta\in [-\delta_3,\delta_3]$ and there cannot be more than three eigenvalues $\Lambda_p^\varepsilon(\eta)$ in the box $I_4\times [0,4\pi^2+K_4].$
Thus, using again Theorem~\ref{Th5.1} and taking
$$\psi_0=\max\left\{\frac{16\pi^2 m_1(\Xi)+C_0}{3\pi}, \frac{C_0H+8\pi^2\vert\omega\vert}{4\pi H}  \right\},$$
for $\varepsilon$ small enough, we have
\begin{equation*}
\begin{split}
    \Lambda^\varepsilon_2(\eta)&\leq \Lambda_-^0(\eta)+C_0\varepsilon =4\pi^2-(4\pi-\eta)\eta+C_0\varepsilon  \leq  4\pi^2-3\pi\eta +C_0\varepsilon \vspace{0.15cm}\\ & \leq 4\pi^2-16\pi^2 m_1(\Xi)\varepsilon  \qquad \mbox{for } \eta\in [\varepsilon \psi_0,\delta_3],  \vspace{0.3cm}\\
    \Lambda^\varepsilon_3(\eta)&\geq \Lambda_+^0(\eta)-C_0\varepsilon=4\pi^2+(4\pi+\eta)\eta-C_0\varepsilon  \geq  4\pi^2+4\pi\eta -C_0\varepsilon \\&  \geq 4\pi^2+8\pi^2\frac{\vert\omega\vert}{H}\varepsilon  \qquad \mbox{for } \eta\in [\varepsilon \psi_0,\delta_3].
     \end{split}
\end{equation*}
In a similar way, we can estimate $\Lambda_2^\varepsilon(\eta)$ and  $\Lambda_3^\varepsilon(\eta)$ for $\eta\in[-\delta_3,-\varepsilon\psi_0],$ where now $\Lambda_2^\varepsilon(\eta)=\Lambda_+^\varepsilon(\eta)$ and $\Lambda_3^\varepsilon(\eta)=\Lambda_-^\varepsilon(\eta)$. This concludes the proof for $\eta\in I_4.$

\medskip

Now we formulate our result on opening spectral gap $\gamma^\varepsilon_2$ (see Figure~\ref{fig3} a)):

\begin{Theorem}\label{Corollary_uno}
Let $H\in(0,1/2)$.  Then,  there exists a positive constant $\varepsilon_0=\varepsilon_0(H)$ such that, for $\varepsilon\in(0,\varepsilon_0]$, the asymptotic formulas \eqref{(LL1)} are valid and the gap \eqref{(L2)} with  \emph{}$p=2$ has positive length  \eqref{(MEP99)}.
\end{Theorem}

\section{Concluding remarks and open problems}\label{sec8}

We comment on other possible spectral gaps arising from other nodes of the limit dispersion curves which are not considered in previous sections.

\subsection{Closed and shaded gaps}\label{subsec81}

We note that the nodes marked with $\bullet$ and $\mbox{\tiny$\blacksquare$}$ in Figure~\ref{fig2} a)--c), can separate when dealing with the perturbed problem,
but do not give rise to spectral gaps because they are shaded by other dispersion curves in \textcolor{black}{Figure~\ref{fig3}} a)--c).
More precisely, the node $(0,4\pi^2)$ marked with $\circ$ in Figure~\ref{fig2}~a) gets the symbol $\bullet$ in {\color{black}Figure~\ref{fig2}} b) and c) because the spectral gap described in Theorem~\ref{Corollary_uno} is
shaded by a small perturbation, see Section~\ref{sec4}, of the limit
dispersion curves
 $$
 \Lambda=\frac{\pi^2}{H^2}+\eta^2,\,\,\eta\in[-\pi,\pi],
 \quad\mbox{\rm for}\,\,H\in\Big(\frac{1}{2},\frac{1}{\sqrt{3}}\Big),
 $$
 $$
 \Lambda=\frac{\pi^2}{H^2}+(\eta\mp2\pi)^2,\,\,\pm\eta\in[0,\pi],
 \quad\mbox{\rm for}\,\,H>\frac{1}{\sqrt{3}}.
 $$

In a similar way, after perturbation, the node $(\pm\pi,\pi^2)$, marked with
$\mbox{\tiny$\square$}$ in Figure~\ref{fig2} a) and b) provides an open
spectral gap when $H\in(0,1)$ but the same
node in Figure~\ref{fig2} c) is marked with $\mbox{\tiny$\blacksquare$}$
because the gap around it is shaded by a small perturbation of the
dispersion curve
$$
\Lambda=\frac{\pi^2}{H^2}+\eta^2,\,\,\eta\in[-\pi,\pi],\quad \mbox{\rm for}\,\,H>1.
$$

Other  nodes such as $(\eta_{\mbox{\small$\bullet$}},\Lambda_{\mbox{\small$\bullet$}})=(0,\pi^2H^{-2})$ and
$(\eta_{\mbox{\tiny$\blacksquare$}},\Lambda_{\mbox{\tiny$\blacksquare$}})= (\pm\pi,\pi^2(1+H^{-2}))$,
also detected in Figure~\ref{fig2} a)--c), do not give rise to open spectral gaps with some possible exceptions: $H=1$, $H=1/\sqrt{3}$, $H=1/2$,  $H=1/\sqrt{5}$, $H=1/\sqrt{8}$ and others (cf. Figures~\ref{fig_new} and \ref{fig6_new}).
To examine these nodes in these exceptional cases, important  modifications of our calculations in Section~\ref{subsec61} and \ref{subsec64} are needed and we postpone their study.
\begin{figure}[ht]
\begin{center}
\resizebox{!}{8cm} {\includegraphics{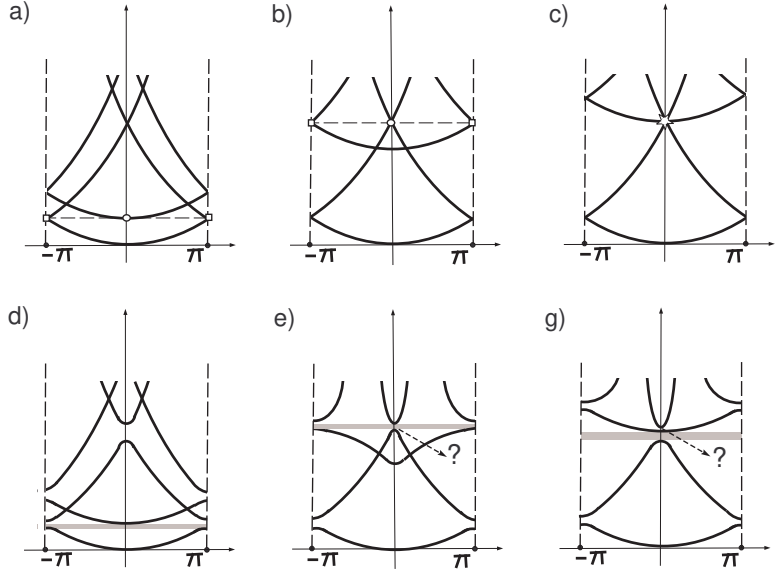}} \vspace{0.4cm}
\caption{The exceptional cases $H=1$, $H=1/\sqrt{3}$ and $H=1/2$.}  
\label{fig6_new}
\end{center}
\end{figure}

The nodes marked with $\mbox{\small{$\blacktriangleleft$}}$ and $\mbox{\small{$\blacktriangleright$}}$
in Figure~\ref{fig2} a), do not
give rise to open gaps due to another reason as depicted
schematically in Figure~\ref{fig_ascending}: both cases of perturbed curves do not
provide a gap. A rigorous justification of the absence of spectral gaps
around nodes generated by similar, either ascending, or descending,
dispersion curves can be found in \cite{na537}.
\begin{figure}[ht]
\begin{center}
\resizebox{!}{2.2cm} {\includegraphics{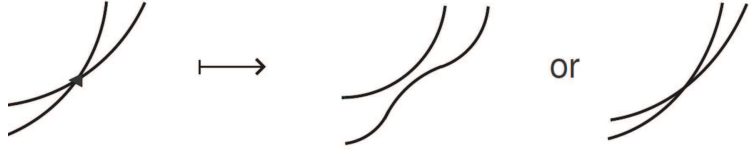}} \vspace{0.2cm}
\caption{The perturbation of ascending curves.}    
\label{fig_ascending}
\end{center}
\end{figure}

\subsection{On the symmetry assumption and possible generalizations} \label{subsec82}

Under the symmetry assumption \eqref{(symm)} we reduce the problem
\eqref{(11)}--\eqref{(14)} to the lower half
of the periodicity cell \eqref{(10)}
 \begin{equation}\begin{array}{c}
\displaystyle -\Delta
U^\varepsilon(x;\eta)=\Lambda^\varepsilon(\eta)U^\varepsilon(x;\eta),
\quad x\in\{x\in\varpi^\varepsilon:\,x_2<H/2\},   \vspace{0.2cm} \\
\displaystyle U^\varepsilon\Big(\frac{1}{2},x_2;\eta\Big)=e^{i\eta}U^\varepsilon\Big(-\frac{1}{2},x_2;\eta\Big), \quad
x_2\in\Big(0,\frac{H}{2}\Big), \vspace{0.3cm} \\
\displaystyle \frac{\partial U^\varepsilon}{\partial
x_1}\Big(\frac{1}{2},x_2;\eta\Big)=e^{i\eta} \frac{\partial
U^\varepsilon}{\partial x_1}\Big(-\frac{1}{2},x_2;\eta\Big), \quad
x_2\in\Big(0,\frac{H}{2}\Big), \vspace{0.3cm}\\
\displaystyle \partial_\nu U^\varepsilon(x)=0,\quad x\in \{x\in\partial\varpi^\varepsilon:\,\vert x_1\vert<1/2, \,\, x_2<H/2\}.
\end{array}\label{(L4)}\end{equation}
On the truncation line
$\Sigma^\varepsilon=\{x\in\varpi^\varepsilon:\,x_2=H/2\}$, we impose
an artificial boundary condition, either the Neumann condition
\begin{equation}\begin{array}{c}\displaystyle
\frac{\partial U^\varepsilon}{\partial
x_2}(x;\eta)=0,\,\,x\in\Sigma^\varepsilon,
\end{array}\label{(L5)}\end{equation}
or the Dirichlet one
\begin{equation}\begin{array}{c}
 U^\varepsilon(x;\eta)=0,\,\,x\in\Sigma^\varepsilon.
\end{array}\label{(L6)}\end{equation}
Clearly, in view of the geometrical symmetry the even (in the
variable $x_2-H/2$) extension above $\Sigma^\varepsilon$ of an
eigenfunction of the problem \eqref{(L4)}, \eqref{(L5)} becomes an
eigenfunction of the problem \eqref{(11)}--\eqref{(14)} with the
same eigenvalue while the odd extension does the same with an
eigenfunction of the problem \eqref{(L4)}, \eqref{(L6)}

A similar reduction of the limit problem \eqref{(29)}--\eqref{(25)}
divides the family \eqref{(31)} of eigenpairs into two groups
containing even ($q=2j$) and odd ($q=1+2j$) in the variable
$x_2-H/2$ eigenfunctions \eqref{(31)}. Hence, the eigenfunctions in the
first and second groups satisfy the Neumann and Dirichlet artificial
boundary conditions on the horizontal mid-line
$\{x\in\Pi^\varepsilon:\, x_2=H/2\}$ of the perforated strip
$\Pi^\varepsilon$. The limit dispersion curves are drawn in {\color{black}Figure~\ref{fig7}}
a) and b), respectively. The previous asymptotic analysis applied to
problems \eqref{(L4)}, \eqref{(L6)} and \eqref{(L4)}, \eqref{(L5)}
independently leads to the dispersion curves  in {\color{black}Figure~\ref{fig7}} c) and d),
respectively. Furthermore, the common graph in {\color{black}Figure~\ref{fig3}} b) is
obtained by uniting the latter graphs after perturbations so that the nodes
$\mbox{\tiny$\lozenge$}$ in Figure~\ref{fig2} b) do not
{\color{black}separate} in contrast to the nodes marked with $\mbox{\tiny$\square$}$ and
$\circ$ (see Figure~\ref{fig8_new}). We recognize this fact as the lack of interaction between
the intersecting curves \eqref{(K1)} with the index couples
$(j,k)=(\pm1,0)$ and $(j,k)=(0,1)$ in \eqref{(31)}.

\begin{figure}[th]
\begin{center}
\resizebox{!}{4.5cm} {\includegraphics{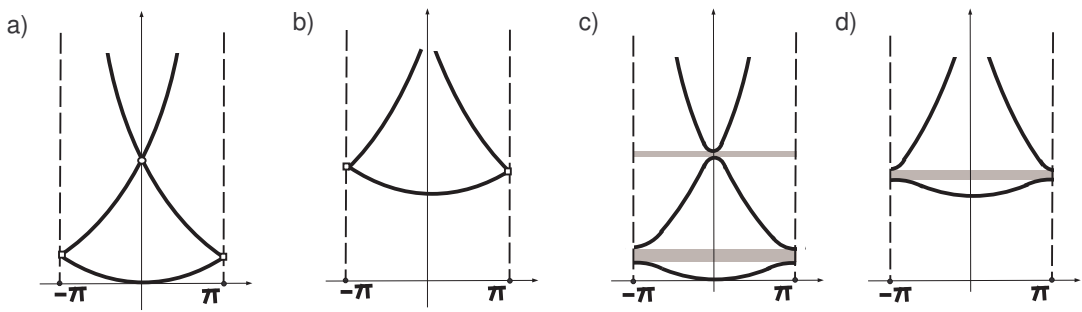}}
\caption{Disjoint trusses under the symmetry condition.}  \vspace{-0.2cm}
\label{fig7}
\end{center}
\end{figure}

\begin{figure}[th]
\begin{center}
\resizebox{!}{2cm} {\includegraphics{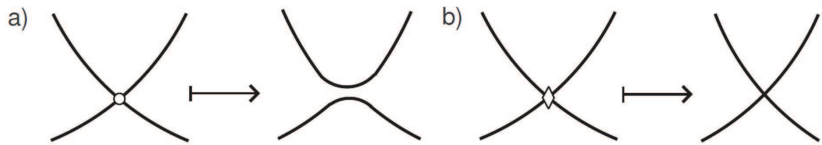}}
\caption{The perturbation of similar and dissimilar curves.}
\label{fig8_new}
\end{center}
\end{figure}

As depicted in Figure~\ref{fig3} b)--c), all nodes marked with
$\mbox{\tiny$\lozenge$}$ in Figure~\ref{fig2} do not split due to the geometrical symmetry \eqref{(symm)}.
One may hope that denying
the symmetry assumption \eqref{(symm)} provides
{\color{black}separation} of the nodes
$\mbox{\tiny$\lozenge$}$ to
open many gaps in
Figure~\ref{fig4} a)--c){\footnote{Actually these dispersion graphs are taken
from the paper \cite{na537} which analyze a quantum waveguides with
regularly perturbed walls.}}. However, we cannot confirm such  a splitting of band
edges by our present asymptotic analysis.

Another way to conclude on splitting by analyzing the first correction term in the
eigenvalue asymptotics only is to treat either inclined perforation
springs, Figure~\ref{fig9_new} a) or holes of varying size, Figure~\ref{fig9_new} b).
Again both modifications require a serious complication of calculations.

\begin{figure}[th]
\begin{center}
\resizebox{!}{1.8cm} {\includegraphics{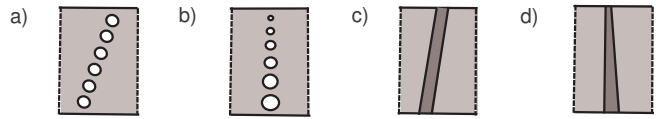}}
\caption{The distorted periodicity cells.}  \vspace{-0.5cm}
\label{fig9_new}
\end{center}
\end{figure}

\begin{figure}[th]
\begin{center}
\resizebox{!}{1.8cm} {\includegraphics{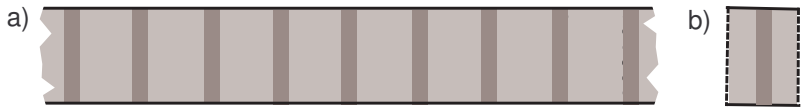}}
\caption{The waveguide with periodic strata.}
\label{fig10}
\end{center}
\end{figure}

A similar spectral problem in a stratified strip in Figure~\ref{fig10} a) with
foreign acoustic material in shaded thin rectangles can be solved
explicitly by separating variables. However, in the case of straight
and homogeneous strata as in {\color{black}Figure~\ref{fig10}} b), we again cannot conclude on
the splitting  of the nodes $\mbox{\tiny$\lozenge$}$ while
dealing with the first correction term only. To clarify the
possibility of opening corresponding spectral gaps, one can disturb
the strata as depicted in {\color{black}Figure~\ref{fig9_new}} c) and d), or even deal with curved
stratum in the periodicity cell, namely
 \begin{equation}\nonumber\begin{array}{c}
 \varsigma^\varepsilon=\{x:\,x_2\in(0,H),\,-\varepsilon h_-(x_2)<x_1-j<\varepsilon
 h_+(x_2)\},
 \end{array}\label{(Z1)}\end{equation}
where  $h\pm\in C^\infty[0,H]$ are profile functions such that
 \begin{equation}\nonumber\begin{array}{c}
 h(x_2)= h_+(x_2)+ h_-(x_2)>0,\,\,x_2\in[0,H].
 \end{array}\label{(Z2)}\end{equation}
However, this perturbation on the thin strips and those outlined in Figure~\ref{fig9_new} stay as open problems. A study of  the corresponding spectrum  will be undertaken in the forthcoming paper of the authors.

\section*{Acknowledgments}
The work has been partially supported by  MICINN  through  PGC2018-098178-B-I00,  PID2020-114703GB-I00 and     Severo Ochoa Programme for Centres of Excellence in R\&D
(CEX2019-000904-S).

\end{document}